\documentclass[10pt]{amsart}

\usepackage{graphicx, color}

\usepackage{hyperref}
\usepackage{amsmath, amssymb}
\usepackage{amsfonts, dsfont}

\DeclareMathAlphabet{\mathpzc}{OT1}{pzc}{m}{it}

\newtheorem{thm}{Theorem}[section]
\newtheorem{ass}[thm]{Assumption}
\newtheorem{coro}[thm]{Corollary}
\newtheorem{lem}[thm]{Lemma}
\newtheorem{prop}[thm]{Proposition}

\theoremstyle{definition}

\newtheorem{defn}[thm]{Definition}

\newtheorem*{nota}{Notation}

\theoremstyle{remark}

\newtheorem{remk}[thm]{Remark}

\makeatletter
\def\l@section{\@tocline{1}{0pt}{1pc}{}{}}
\def\l@subsection{\@tocline{2}{0pt}{1pc}{4.6em}{}}
\def\l@subsubsection{\@tocline{3}{0pt}{1pc}{7.6em}{}}
\renewcommand{\tocsection}[3]{%
  \indentlabel{\@ifnotempty{#2}{\makebox[2.3em][l]{%
    \ignorespaces#1 #2.\hfill}}}#3}
\renewcommand{\tocsubsection}[3]{%
  \indentlabel{\@ifnotempty{#2}{\hspace*{2.3em}\makebox[2.3em][l]{%
    \ignorespaces#1 #2.\hfill}}}#3}
\renewcommand{\tocsubsubsection}[3]{%
  \indentlabel{\@ifnotempty{#2}{\hspace*{4.6em}\makebox[3em][l]{%
    \ignorespaces#1 #2.\hfill}}}#3}
\makeatother

\newcommand{\dis}{\displaystyle}

\newcommand{\wat}{\widetilde{a}_t}
\newcommand{\watl}{\widetilde{a}_t^\ell}

\newcommand{\datl}{\dot{a}_t^\ell}

\newcommand{\ua}{\underline{\alpha}}
\newcommand{\tpsi}{\widetilde{\psi}}

\newcommand{\tg}{\widetilde{g}}

\newcommand{\Rbb}{ {\mathbb R}}
\newcommand{\Zbb}{ {\mathbb Z}}
\newcommand{\Nbb}{ {\mathbb N}}
\newcommand{\Cbb}{ {\mathbb C}}

\newcommand{\Zcal}{ {\mathcal Z}}

\newcommand{\Bcal}{ {\mathcal B}}

\newcommand{\Hcal}{ {\mathcal H}}
\newcommand{\Mcal}{ {\mathcal M}}
\newcommand{\Dcal}{ {\mathcal D}}
\newcommand{\Ccal}{ {\mathcal C}}

\newcommand{\Ecal}{ {\mathcal E}}
\newcommand{\Acal}{ {\mathcal A}}

\newcommand{\Ncal}{ {\mathcal N}}
\newcommand{\Tcal}{ {\mathcal T}}

\newcommand{\talpha}{\widetilde{\alpha}}

\newcommand{\Hbb}{{\mathbb H}}

\newcommand{\eps}{\varepsilon}

\newcommand{\spec}{{\rm spec}}
\newcommand{\dom}{{\rm dom}}
\newcommand{\vect}{{\rm vect}}

\newcommand{\alp}{\alpha}
\newcommand{\bet}{\beta}

\newcommand{\R}{{\mathbf R}}
\newcommand{\Z}{{\mathbf Z}}
\newcommand{\N}{{\mathbf N}}

\newcommand{\und}{\frac{1}{2}}
\newcommand{\dt}{\frac{2}{3}}
\newcommand{\qt}{\frac{4}{3}}
\newcommand{\unt}{\frac{1}{3}}
\newcommand{\uns}{\frac{1}{6}}
\newcommand{\td}{\frac{3}{2}}
\newcommand{\unq}{\frac{1}{4}}
\newcommand{\tq}{\frac{3}{4}}
\newcommand{\hut}{\frac{8}{3}}
\newcommand{\cit}{\frac{5}{3}}
\newcommand{\un}{\mathds{1}}
\newcommand{\wcal}{{\mathcal W}}   

\renewcommand{\geq}{\geqslant}
\renewcommand{\leq}{\leqslant}

\begin{document}

\title{Hyperbolic Triangles without Embedded Eigenvalues}

\author{Luc Hillairet}

\author{Chris Judge}

\thanks{The work of L.H. has been partly supported by the ANR programs
METHCHAOS, NOSEVOL and Gerasic-ANR-13-BS01-0007-0.}

\thanks{The work of C.J. has been partly supported by a Simons collaboration grant. C.J. Also
thanks l'Universit\'e d'Orleans for its hospitality.}

\begin{abstract}
We consider the Neumann Laplacian acting on square-integrable
functions on a triangle in the hyperbolic plane that has one cusp. We show
that the generic such triangle has no eigenvalues embedded in its continuous
spectrum. To prove this result we study the behavior of the real-analytic 
eigenvalue branches of a degenerating family of triangles. In particular,
we use  a careful analysis of spectral projections near the crossings 
of these eigenvalue branches with the eigenvalue branches of a model operator. 
\end{abstract}

\maketitle


\section{Introduction}

Though well-studied for over fifty years, the spectral theory of hyperbolic 
surfaces still presents basic unresolved questions \cite{Sarnak03}. 
For example, does there exist a noncompact, finite area hyperbolic surface whose 
Laplacian
has no nonconstant square-integrable eigenfunctions?  This question 
has been the subject of many investigations including \cite{CdV83}, 
\cite{PhillipsSarnak85}, \cite{DIPS85}, \cite{PhillipsSarnak92a},
 \cite{Wolpert92}, \cite{Wolpert}, and \cite{PhillipsSarnak94}.

\begin{figure}[h]   
\begin{center}
\includegraphics[totalheight=1.8in]{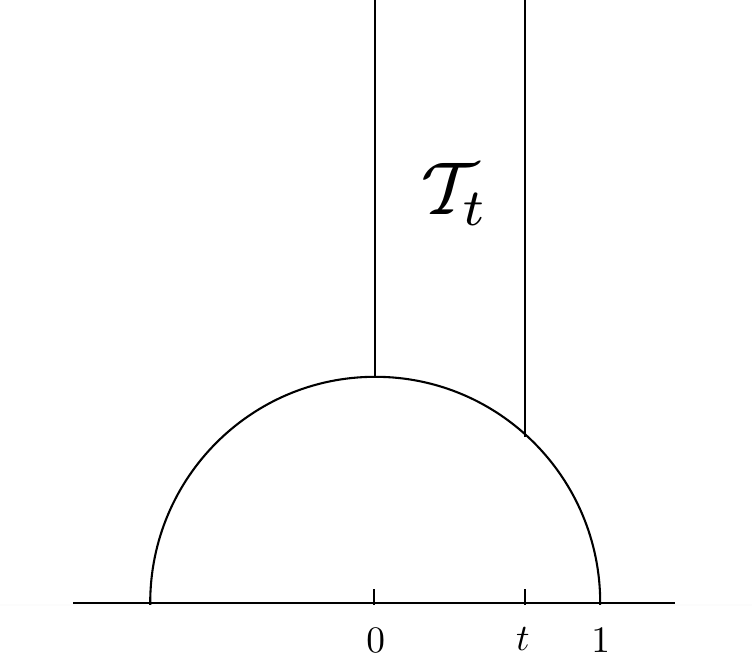}
\end{center}
\caption{\label{TriSect}The triangle $\Tcal_t$ 
defined by $x^2+y^2\geq 1$ and $0\leq x \leq t$.}
\end{figure}

As a model problem, Phillips and Sarnak \cite{PhillipsSarnak92b}
suggested studying the Neumann eigenvalue problem 
on the domain $\Tcal_t \subset \Rbb \times \Rbb^+$ 
pictured in Figure \ref{TriSect}. In this paper, we prove the following:

\begin{thm}\label{thm:main}
For all but at most countably many $t \in~ ]0,1[$, 
the Neumann Laplacian on the geodesic triangle 
$\Tcal_t$ in the hyperbolic plane has no nonconstant (square-integrable) 
eigenfunction.
\end{thm}

The group $G_t$ of hyperbolic isometries generated by reflections in the geodesic arcs 
that bound $\Tcal_t$ is discrete 
if and only if  $t = \cos(\pi/n)$ for $n \geq 3 $ an integer. 
For example, if $t=1/2$, then $G_t$ contains an index two subgroup that is 
naturally isomorphic to $PSL_2(\Zbb)$. 
It follows from the seminal work of A. Selberg \cite{Selberg} 
that if $n=3,4,$ or $6$, then the Neumann Laplacian has infinitely many
nonconstant eigenfunctions. In particular, for some special $t$, 
there do exist square-integrable solutions to the Neumann problem.\footnote{In 
the case where $n \geq 3$ is an integer not equal to $3,4$, or $6$,
Phillips and Sarnak asked whether the domain $\Tcal_{\cos(\pi/n)}$
has no nonconstant Neumann eigenfunctions \cite{PhillipsSarnak92b}. 
We should point out that since Theorem \ref{thm:main} allows for 
countably many exceptional $t$, it does not directly answer their question.}  

In \cite{Jdg95}, Theorem \ref{thm:main} was verified under an additional---and as of yet 
unjustified---assumption concerning the spectral multiplicities of the Neumann Laplacian 
acting on $L^2(\Tcal_1)$.  The proof consisted of studying the singular perturbation 
problem associated with letting $t$ tend to $1$. Similar singular perturbations
were studied in the context of degenerating hyperbolic surfaces  \cite{Wolpert}
and unitary characters over a fixed hyperbolic surface \cite{PhillipsSarnak94}.
In \cite{Wolpert}, \cite{PhillipsSarnak94}, \cite{Jdg95}, and all 
prior work on this problem, it was necessary to make assumptions about
the multiplicities of the spectrum of the limiting surface.

The angles of a geodesic triangle in the hyperbolic plane determine 
the isometry class of the triangle. The angles of $\Tcal_t$ are $(\pi/2, \arccos(t), 0)$. 
It is not difficult to extend Theorem \ref{thm:main} to
triangles with angles $(\theta_1, \theta_2, 0)$  (see \S \ref{section:alltriangles}).
\begin{thm}\label{thm:GenNocf}
The set of $(\theta_1, \theta_2)$ for which the hyperbolic triangle with 
angles $(\theta_1,\theta_2, 0)$ admits a nonconstant Neumann Laplace eigenfunction
has Lebesgue measure zero and is contained in 
a countable union of nowhere dense sets. 
\end{thm}
In other words, the generic hyperbolic triangle with one cusp has no 
non-constant Neumann eigenvalue where `generic' can be taken in both a
topological and a measurable sense. Theorem \ref{thm:main} 
gives the existence of a triple of angles $(\theta_1, \theta_2, 0)$ for which 
there are no nonconstant Neumann eigenfunctions. Theorem \ref{thm:GenNocf} 
then results from applying a general and well-understood principle concerning
analytic perturbations (See, for example, \cite{HJ09}).  
On the other hand, the proof of  Theorem \ref{thm:main} is much more 
involved. In particular, the proof will rely upon a refined analysis 
of `crossings' of eigenvalue branches. 

To prove Theorem 1.1, we further develop the method of ‘asymptotic separa-
tion of variables’ that we introduced in [HlrJdg11] to study generic simplicity
of eigenvalues. This method facilitates the study of real-analytic eigenvalue branches
in situations where a geometric domain degenerates onto a lower dimensional domain.
There is a vast literature---for instance, \cite{BF}, \cite{GriJer},
\cite{FriSol}---concerning 
perturbations involving degeneration onto lower dimensional domains,
but most of these studies do not address analytic eigenvalue branches. 
In contrast, our results depend crucially on a study of real-analytic eigenbranches
and their crossings.

\subsection{An outline of this paper} 

We now describe the content of each section.

In \S \ref{section:prelim}, we establish notation and recall basic
features of the spectral theory of the Laplacian acting on functions
on a domain in the hyperbolic plane having one cusp. We describe 
the Fourier decomposition associated to the cusp.  The zeroth Fourier mode
is responsible for an essential spectrum of $[\unq, \infty[$. 
Following  \cite{LaxPhillips} and  \cite{CdV82}, we will replace the Dirichlet
quadratic form $\Ecal(u) =\langle \Delta u, u\rangle$ 
with a modification $\Ecal_{\beta}$ obtained by `truncating' the zeroth 
Fourier coefficient above $y=\beta$. An eigenfunction $u$ of $\Ecal_{\beta}$ 
corresponds to an eigenfunction of $\Ecal$ 
if and only if the zeroth Fourier coefficient of $u$ vanishes identically. 
We will call such an eigenfunction a {\em cusp form}.\footnote{For $t= \cos(\pi/n)$,
these are `cusp forms' in the sense of the theory of automorphic forms, but otherwise 
there is no discrete group, and hence they are not cusp forms in the traditional sense.
In this paper we will always be considering `even' cusp forms, that is, eigenfunctions
satisfying Neumann conditions.} 
The operator associated to $\Ecal_{\beta}$ has compact resolvent 
and hence discrete spectrum. This makes $\Ecal_{\beta}$ a much better 
candidate for the application of methods from spectral perturbation theory. 

In \S \ref{sec:generic},  we recall and make precise some ideas familiar in
the perturbational study of cusp form existence. 
We consider a real-analytic family, $t \mapsto q_t$, of quadratic forms that 
have the same domain as $\Ecal_{\beta}$. We say that a
real-analytic family $t \mapsto u_t$ of eigenfunctions of $q_t$
is a {\em cusp form eigenbranch} if and only $u_t$ is a cusp form for each $t$.
We demonstrate a dichotomy: Either the family $t \mapsto q_t$ has a real-analytic 
cusp form eigenbranch or the set of $t$ such that $q_t$ has a cusp form is countable.

In  \S \ref{section:alltriangles}, we consider arbitrary real-analytic paths in the space 
of hyperbolic triangles with one cusp. We apply the results of \S \ref{sec:generic}
to deduce Theorem \ref{thm:GenNocf} under the assumption that there exists a
triangle with no nonconstant Neumann eigenfunction. 
The remainder of the paper is devoted to proving Theorem \ref{thm:main}
which will give the existence of such a triangle. 

In \S \ref{section:expansion} we specialize to the family $\Tcal_t$.  After  
renormalizing by a factor of $t^2$, we find that for each $u$,
the function $t \mapsto q_t(u)$ has a Taylor expansion at $t=0$.
We compute the leading order terms of this expansion.

In \S \ref{section:approx} we show that the method of 
the asymptotic separation of variables introduced in \cite{HJ10} 
may be used to analyse the family of quadratic forms $q_{\beta,t}$.
In particular, we define a reference quadratic form $a_t$ to which 
separation of variables applies and that is asymptotic to $q_{\beta,t}$ 
at `first order'. By separation of variables we mean that 
each eigenfunction of $a_t$ is of the form 
$v_t^{\ell}(y) \cdot \cos(\pi \ell x)$ with $\ell \in \Zbb$ and 
$v_t^{\ell}$ a solution to 
\begin{equation}  \label{eqn:a_t_ODE}
 -t^2 \cdot u''~ +~ \left((k \pi)^2 - \frac{\lambda}{y^2} \right) \cdot u~ =~ 0,
\end{equation}
a renormalized form of the equation for a modified Bessel function with imaginary parameter.
The potential $(k\pi)^2 - \lambda \cdot y^{-2}$ is positive for $y$ large 
and has a unique zero at $y = \sqrt{\lambda}/(k \pi)$.  To analyse 
the solutions to  (\ref{eqn:a_t_ODE}), we will relate them 
to the Airy functions, solutions to the ordinary differential equation $\partial_x^2A= x \cdot A=0$.
The remainder of the paper depends crucially 
on the analysis of (\ref{eqn:a_t_ODE}) using Airy functions
that has been placed in the Appendices.

In \S \ref{section:branches} we prove a non-concentration 
estimate---Proposition \ref{prop:noncon}---and use this estimate
to derive information concerning the real-analytic eigenbranches 
$(E_t, u_t)$ of $q_{\beta,t}$.  First, we show that there exists an 
integer $k$ so that $E_t$ limits to $(\pi k)^2$ as $t$ tends to zero.
Second, we find that if the spectral projection of 
$u_t$ onto the space $V_k$ spanned by functions of the form 
the $\psi(y) \cdot \cos(\pi k x)$ is relatively small, then 
the derivative $\partial_t E_t$ is of order $1/t$. 
Finally, we show that if $(E_t,u_t)$ is a cusp form eigenbranch, 
then $E_t$ can not limit to zero. 

In \S \ref{sec:proof} we prove Theorem \ref{thm:main}. 
By the dichotomy of \S \ref{sec:generic}, it suffices to show that real-analytic
cusp form eigenbranches do not exist. We suppose to the contrary that
the real-analytic family $q_{\beta,t}$ has a cusp form eigenbranch $(E_t,u_t)$. 
By the results of \S \ref{section:branches}, we have that 
$E_t$ limits to $(\pi k)^2$ where $k \in \Zbb^+$.
By improving the analysis of \cite{HJ10},
we show that there exists an eigenbranch, $\lambda_t^*$, of $a_t$ 
that `tracks' $E_t$ at order $t$ in the sense that
\begin{equation}\label{TrackingEstimate}
\limsup_{t \rightarrow 0}~ \frac{1}{t} \cdot \left| E_t- \lambda_t^* \right | < \infty.
\end{equation}
We will obtain a contradiction to this estimate by estimating  
$f(t):=\frac{d}{dt} \left( E_t-\lambda_t^* \right)$  from below.  

Indeed, we show that when the norm of the projection $w_t^k$ of $u_t$ 
onto $V_k$ is relatively large with respect to $\|u_t\|$, the function
$f(t)$ is of order $t^{-{\unt}}$, whereas when $\|w_t^k\|$ is relatively small, 
the function $f(t)$ is of order $t^{-1}$.
By controling the sizes of the sets where $\|w_t^k\|$ is respectively small and
large relative to $\|u_t\|$ and by integrating, we will contradict (\ref{TrackingEstimate}). 

The key observation is the following: Since $E_t$ limits to $(k\pi)^2$, 
it has to `cross' each of the eigenbranches of $a_t$ that limit to zero. 
We show that near such a crossing, the branch $u_t$ must `interact' 
with the functions in $V_0$ to such an extent that the projection onto
$V_k$ cannot be too large. 
The effect of each interaction is made precise by careful estimates of the off-diagonal 
terms in the quadratic form $q_t-a_t$ (Appendix A). By summing the effects of these
interactions, we eventually prove that there exists $c>0$ so that 
\[
  E_t - \lambda^*_t~ \geq~  c \cdot t^{\dt}
\]
thus contradicting (\ref{TrackingEstimate}). 


\section{The spectrum of a domain in the hyperbolic plane with a cusp}

\label{section:prelim}

In this section, we describe some basic spectral theory of the 
Neumann Laplace operator acting on the square-integrable functions 
on domains in the hyperbolic plane with a cusp. 
We define the Laplacian and associated Dirichlet quadratic form,
describe the Fourier decomposition of eigenfunctions along horocycles,
construct a modification of the Dirichlet form whose eigenfunctions 
include the eigenfunctions (cusp forms) 
of the standard Laplacian but whose spectrum is discrete.

\subsection{The quadratic forms associated to the Neumann Laplacian}
\label{DomainSubsection}

The half-plane $\{(x,y),~y>0\}$ equipped 
with the Riemannian  metric $y^{-2}(dx^2+dy^2)$ is 
the Poincar\'e-Lobachevsky model for the $2$-dimensional hyperbolic space $\Hbb^2$.
The measure associated to the Riemannian metric $g=y^{-2}(dx^2+dy^2)$
is given by integrating 
\[ dm =  \frac{dx  dy}{y^2}.  \]

In the present context, a {\em cusp of width $w$ and height $y_0$} is the
subset $S_{w,y_0}:= [0,w] \times [y_0, \infty[$ of the upper half plane.
A domain $\Omega \subset \Hbb^2$ is said to {\em have one cusp}  
if $\Omega$ is the union of a cusp and a compact set. 
We will assume that the boundary of $\Omega$ is the union
of finitely many geodesic arcs and that the interior of $\Omega$ is connected. 

Let $\Dcal(\overline{\Omega})$ denote the set of functions $u: \Omega \rightarrow \Cbb$ 
such that $u$ is the restriction to $\Omega$ of a compactly supported 
smooth function defined on some neighbourhood of $\Omega$. 
 
The $L^2$-inner product of two functions $u$ and $v$ in $\Dcal(\overline{\Omega})$ is defined by 
\begin{equation} \label{defn:Ncal}
\Ncal(u,v) := \int_{\Omega}  u(x,y)\cdot \overline{v(x,y)}~ dm.
\end{equation}
Abusing notation slightly, we will often write $\Ncal(u)$ in place of $\Ncal(u,u)$. 
Let $L^2(\Omega, dm)$ denote the completion of $\Dcal(\overline{\Omega})$ with
respect to the norm $\Ncal(u)^{\frac{1}{2}}$.

To define the Neumann Laplacian we consider the bilinear form
defined on $\Dcal(\overline{\Omega})$ by 
\[
\Ecal(u,v) :=  \int_{\Omega} g(\nabla u, \nabla v)~ dm 
\] 
where $\nabla$ satisfies $g(\nabla f, X)=Xf$ for all vector fields $X$ and smooth
functions $f$. Let $\Ecal(u)$ denote the value of the quadratic form 
$u \mapsto \Ecal(u,u)$.  One computes that
\[ 
\Ecal(u)~ =~
   \int_{\Omega}  \left| \partial_x u(x,y)\right |^2\,
       +\,\left| \partial_y u(x,y)\right |^2\, dxdy.
\]
Let $H^1(\Omega)$ denote the completion of $\Dcal(\overline{\Omega})$ 
with respect to the norm $(\Ecal(u) + \Ncal(u))^{\frac{1}{2}}$.
We will consider $\Ecal$ as a nonnegative symmetric bilinear 
form on $L^2(\Omega)$ with domain $H^1(\Omega)$.
As such it is densely defined and closed, and hence
there exists a unique, densely defined, self-adjoint 
operator $\Delta$ on $L^2(\Omega)$ 
such that for each $v \in H^1(\Omega)$ and $u$ in 
the domain of $\Delta$ we have\footnote{See, for example, 
Theorem VI.2.1 \cite{Kato}.}
\begin{equation} \label{Laplacian} 
  \Ncal( \Delta u, v)~ =~ \Ecal(u, v).
\end{equation}
The operator $\Delta$ is called the {\em Neumann Laplacian}. It can be shown that 
$u\in H^1(\Omega)$ if and only if $u\in L^2(\Omega,dm)$ and $\Ecal(u) < +\infty$ 
where in the definition of $\Ecal$ the partial derivatives are to be taken in the 
distributional sense.

It is well-known that $\Delta$ has essential spectrum equal to 
$[1/4, \infty[$.  (For example, see \cite{LaxPhillips} or \cite{CdV82}).   
Apart from this essential spectrum, 
$\Delta$ may also have eigenvalues either smaller than 
$1/4$ or embedded in the continuous 
spectrum.

From (\ref{Laplacian}) we see that $u$ 
is an eigenfunction of $\Delta$ with eigenvalue $E$ 
if and only if $u\in H^1(\Omega)$ and for each $v \in H^1(\Omega)$
\begin{equation}\label{eq:defcf}
 \Ecal(u,v) \,=\, E \cdot \Ncal(u,v).
\end{equation}
%


\subsection{Fourier decomposition in the cusp}\label{sec:Fourier}

For each positive integer $k$, define 
\[ e_k(x)~  :=~ 2^{\und} \cdot \cos(k\pi \cdot x)
\]
and define $e_0 \equiv 1$.
The collection $\{e_k~ |~ k \geq 0 \}$ is an orthonormal basis for $L^2([0,1])$.
Hence, the functions
\[  x~ \mapsto~   \frac{1}{\sqrt{w}} \cdot e_k \left(\frac{x}{w} \right)
\]
provide an orthonormal basis of $L^2([0,w])$.

For positive $w,~y_0$, let $S_{w,y_0}= [0,w] \times [y_0, \infty[$ be
a cusp of width $w$ and height $y_0$.
 
For each $u$ in $L^2(S_{w,y_0},dm)$ and almost every $y \geq y_0$, the function
$x \mapsto  u(x,y)$ belongs to $L^2([0,w])$. Thus we can write
\begin{equation}\label{eq:Fourdec}
u(x,y) \, =\, \sum_{k\geq 0}~ u^k(y) \cdot e_k \left(\frac{x}{w} \right)
\end{equation}
where
\begin{equation}\label{eq:Fourdec0}
u^k(y)~ :=~ \frac{1}{w}\int_0^w u(x,y) \cdot e_k \left(\frac{x}{w} \right)~ dx.
\end{equation} 
belongs to $L^2\left ( [y_0,\infty[, y^{-2}dy \right)$.
We refer to $u^k$ as the $k^{{\rm th}}$ Fourier coefficient of $u$.
More generally, if $\Omega$ is a domain with a cusp $S_{w,y_0}$,
then we define the $k^{{\rm th}}$ 
Fourier coefficient of a function $v: \Omega \rightarrow \Cbb$ 
to be the $k^{{\rm th}}$ Fourier coefficient restriction of $v$ to $S_{w,y_0}$. 
Parseval's theorem gives
\[
\Ncal\left(u \cdot \un_{[y_0, \infty[}\right)~ 
=~ \sum_{k\geq 0}~\int_{y_0}^\infty~ \left |u^k(y)\right|^2 \frac{dy}{y^2}
\]
where $\un_X$ denotes the characteristic function of a set $X$.

\begin{lem} \label{lem:ODE}
If $u \in H^1(\Omega,dm)$ is a Neumann eigenfunction of $\Ecal$ with eigenvalue $E$, 
then for each $k \in \Nbb$ and each $y > y_0$, the Fourier coefficient $u^k$ satisfies
\begin{equation}\label{eq:ODEn}
- \left(u^k\right)'' + \left(\frac{(k\pi)^2}{w^2}- \frac{E}{y^2}\right) u^k\,=\,0.
\end{equation}
\end{lem}    

\begin{proof}
If $v$ is a smooth function on $\Omega$, then since $u$ is a Neumann eigenfunction of $\Ecal$,
integration by parts gives
\[  -\int_{\Omega} \left(u \cdot \partial_x^2 v~ + u \cdot \partial_y^2v \right)~dxdy~ 
   =~ E \int_{\Omega} u \cdot v~ \frac{dxdy}{y^2}. \]
By letting  $v= \phi(y) \cdot e_k(x/w)$ where $\phi$ 
is a smooth function with compact 
support in $]y_0,\infty[$, we find that 
\[  \frac{(k \pi)^2}{w^2} \int_{y_0}^{\infty} u^k(y) \cdot \phi(y)~ dy~ 
 -~ \int_{y_0}^{\infty} u^k(y) \cdot \phi''(y)~ dy~ 
   =~ E \int_{y_0}^{\infty}  u^k(y) \cdot ~ \frac{dy}{y^2}. \]
It follows that $u^k$ satisfies (\ref{eq:ODEn}) in the distributional sense in $\Dcal'((y_0,\infty)).$ 
By elliptic regularity, $u^k$ is actually smooth and satisfies (\ref{eq:ODEn}) in the strong sense.
\end{proof}

In particular, $u^0= A y^s +B y^{1-s}$ for some constants $A$ and $B$ with $E=s(1-s)$.
If $E \geq 1/4$, then the real part of $s$ equals $1/2$, and hence $u^0$ does not 
belong to $L^2\left ( [y_0,\infty[, y^{-2}dy \right)$ unless both $A$ and $B$ equal zero.
Therefore, we have the following.  

\begin{coro} \label{coro:cuspform}
If $u$ is a Neumann eigenfunction with eigenvalue $E \geq \frac{1}{4}$, 
then the zeroth Fourier coefficient $u^0$ vanishes identically on $[y_0,\infty[$.
\end{coro}

In the classical spectral theory of a quotient of $\Hbb$ by a lattice in $SL_2(\Rbb)$,
a Laplace eigenfunction $u$ with vanishing zeroth Fourier coefficient in each cusp
is called a {\em (weight zero) Maass cusp form}.  
Even though most of the domains that we will consider  
are not fundamental domains for discrete groups of isometries,
we will adapt this terminology.

\begin{defn}
If $u$ is an eigenfunction for the Neumann Laplacian on a domain 
with a cusp, and $\left(u \cdot \un_{[y_0, \infty[}\right)^0 \equiv 0$, then we will call
$u$ a {\em cusp form}. 
\end{defn}

Traditionally, the Neumann eigenfunctions for $\Tcal_{\cos(\pi/n)}$ are called 
{\em even cusp forms} whereas the solutions to the Dirichlet eigenvalue problem
are called {\em odd cusp forms}. We will not consider odd cusp forms 
in this paper.


\subsection{A related quadratic form}\label{IdealSection}

We wish to apply analytic perturbation theory to study the behavior 
of eigenfunctions of $\Ecal$ on $\Tcal_t$ as we vary $t$. 
Because the eigenvalues of $\Ecal$ might lie inside the essential spectrum, 
standard perturbation theory does not apply directly.  
Following \cite{CdV82} and \cite{PhillipsSarnak85},
we will use a modification of $\Ecal$ first constructed
by P. Lax and R. Phillips  \cite{LaxPhillips} \cite{LaxPhillips80}.\footnote{See 
page 206 of \cite{LaxPhillips} under the heading `A related quadratic form'.}
In this section,  
we recall the construction, show that the eigenvalues of the modification 
are isolated, and relate the eigenfunctions of the modification to those of $\Ecal$.

For $\beta>y_0$, let $Z_{\beta}$ denote the set of $u \in \Dcal(\overline{\Omega})$ such that 
for each $y\geq \beta$ we have $u^0(y)=0$. Let $L^2_{\beta}(\Omega,dm)$ denote
the Hilbert space completion of $Z_{\beta}$ with respect to 
$u \mapsto \Ncal(u)^{\frac{1}{2}}$. Let $\Ncal_{\beta}$ denote the 
restriction of $\Ncal$ to $L^2_{\beta}(\Omega)$.

Let $H^1_{\beta}(\Omega)$ denote the Hilbert space completion of $Z_{\beta}$  
with respect to the norm $u \mapsto (\Ecal(u)+ \Ncal(u))^{\frac{1}{2}}$. 
The restriction, $\Ecal_{\beta}$,  of $\Ecal$ to  $H^1_{\beta}(\Omega)$
is a closed, densely defined quadratic form on $L^2_{\beta}(\Omega)$.
A simple argument shows that
\begin{equation*}
L^2_\beta(\Omega) = \left \{ u\in L^2(\Omega,dm),~|~\forall y>\beta,~u^0(y)=0.\right \},
\end{equation*}
and 
\begin{equation*}
H^1_\beta(\Omega) = \left \{ u\in H^1(\Omega),~|~\forall y>\beta,~u^0(y)=0.\right \}.
\end{equation*}

In the sequel it will be more convenient to replace $Z_\beta$ by the following other set.  

\begin{defn}\label{def:Wbet}
 Define $W_\beta$ to be the set of functions $u$ in $H^1_\beta$ such that 
\begin{itemize}
\item $u$ extends to a continuous function 
on the closure $\overline{\Omega}$ of $\Omega,$
\item $u$ is smooth on $\Omega\setminus \{ y=\beta \}.$ 
\end{itemize} 
\end{defn}

Observe that since $Z_\beta\subset W_\beta$, the closure of $W_\beta$ with respect to the norm $u \mapsto (\Ecal(u)+ \Ncal(u))^{\frac{1}{2}}$ 
is $H^1_\beta.$ The latter assertion says that $W_\beta$ is a core of the quadratic form $\Ecal_\beta$ 

Let $\Delta_{\beta}$ denote the unique operator such that $\dom(\Delta_\beta) \subset H^1_\beta$ and that 
satisfies $\Ncal_{\beta}(\Delta_{\beta}u,v)= \Ecal_{\beta}(u, v)$
for each $u\in \dom(\Delta_\beta),~ v \in H^1_{\beta}(\Omega)$. 

\begin{lem}[\cite{LaxPhillips}]  \label{lem:relation}
For each $\beta>y_0$, the resolvent of $\Delta_{\beta}$ is compact. 
Hence, the spectrum of $\Ecal_{\beta}$ with respect to $\Ncal_{\beta}$ is 
discrete and each eigenspace is finite dimensional. 
\end{lem}

\begin{proof}
Using the Fourier decomposition, 
one shows that for each $b >0$, the set of $v \in H^1_{\beta}(\Omega)$ 
such that $\Ncal(v)\leq 1$ and $\Ecal(v) \leq b$ is compact in 
$L^2_{\beta}(\Omega)$ (see Lemma 8.7 \cite{LaxPhillips}).
It follows that $\Delta_{\beta}$ has compact resolvent. 
Hence, by standard spectral theory, the spectrum is discrete 
and the eigenspaces are finite dimensional. 
\end{proof}

\begin{defn}[cusp form]\label{def:cf}
We will say that an eigenfunction $u$ of $\Ecal_{\beta}$ 
with respect to $\Ncal_{\beta}$ is a {\em cusp form} if and only
for each $y > y_0$ we have $u^0(y)=0$. 
\end{defn}

\begin{lem}\label{lem:betacf}
 The following assertions are equivalent: \vspace{.2cm}
\begin{enumerate}
\item  $u$ is a cusp form of $\Ecal$ with respect to $\Ncal$. \vspace{.2cm}
\item  There exists $\beta>y_0$ such that $u$ is a cusp form of $\Ecal_{\beta}$
with respect to $\Ncal_{\beta}$. \vspace{.2cm}

\item For each $\beta>y_0$,  the function $u$ is a cusp form for $\Ecal_{\beta}$
with respect to $\Ncal_{\beta}$. \vspace{.2cm}

\end{enumerate}
\end{lem}

\begin{proof}
$(1) \Rightarrow (2)$:
If $u$ is an eigenfunction of $\Ecal$ with eigenvalue $E$, then by Lemma \ref{lem:ODE}
the zeroth Fourier coefficient $u^0$ satisfies  the differential 
equation $0=(u^0)'' + (E/y^2) \cdot u^0$. Since $u^0(y)$ vanishes for $y>y_0$,
it must vanish for $y > \beta$.

$(2) \Rightarrow (1)$: Fix a smooth function $\chi$ such that $\chi(y)=0$ for $y\leq \frac{2y_0+\beta}{3}$ and 
$\chi(y)=1$ for $y\geq \frac{y_0+2\beta}{3}.$ If $u^0(y)=0$ for each $y>y_0$, then 
\[  \Ncal_{\beta}(u, v-\chi \cdot v^0)~ = \Ncal(u,v), \]
and
\[ \Ecal(u, v)~ =~ \Ecal(u, v-\chi \cdot v^0)~ = \Ecal_{\beta}(u, v-\chi \cdot v^0), \]
for each $v \in H^1(\Omega)$. Thus, if  $\Ecal_{\beta}(u, v)= E \cdot \Ncal_{\beta}(u,v)$, 
then $\Ecal(u, v)= E \cdot \Ncal(u,v)$.

$(2) \Leftrightarrow (3)$:  Follows from the equivalence of (1) and (2). 
\end{proof}

Not every eigenfunction $u$ of $\Ecal_{\beta}$ is an eigenfunction of $\Ecal$. 
For example, if $w= \cos(\pi/n)$ and 
$E_s$ is an Eisenstein series whose zeroth Fourier coefficient 
vanishes at $y=\beta$, then $E_s(x,y) - E_s^0(y) \cdot \chi_{ [\beta, \infty[}(y)$
is an eigenfunction of $\Ecal_{\beta}$.  This function is not smooth across
$y=\beta$, but elliptic regularity implies that each eigenfunction of $\Ecal$ is smooth.


\section{Real-analyticity and generic properties of eigenfunctions}
\label{sec:generic}
Let $S= [0,1]\times [y_0,\infty)$ and let $\beta>\underline{\alpha}>
y_0.$ In this section, we consider a fixed domain $\Omega$ that
contains the cusp $S$ and a real-analytic family $t \mapsto q_t$ of
quadratic forms defined on $H^1_{\beta}(\Omega) \subset L^2_{\beta}(S,dm)$
that represents the cusp of width $w_t$ for $y>y_0$ and some
real-analytic function $t\mapsto w_t$ (see Definition
\ref{def:represent} below).  We prove the following dichotomy:
Either there exists a real-analytic eigenfunction branch consisting of `cusp forms' or the
set of $t$ such that $q_t$ has a `cusp form' eigenfunction is countable.
This fact is fundamental to the proofs of both Theorem \ref{thm:main}
and Theorem \ref{thm:GenNocf}.

Let $S= [0,1]\times  [y_0, \infty[$ and let $\beta> \underline{\alpha}> y_0$.
For any $w>0$, we can define a transformation $\widehat{\Phi}_w$ between
$L^2_\beta(S)$ and $L^2_{\beta}(S_{w,y_0})$ by asking that, for any $u\in
L^2_\beta(S)$ the function $v:= \widehat{\Phi}_w(u)$ is defined by $v(x,y) =
\frac{1}{\sqrt{w}}u(\frac{x}{w},y).$ Since, in the sequel $y_0$ will be fixed, we will drop the index 
$y_0$ and denote by $S_w$ the cusp of width $w.$

It is straighforward that $\widehat{\Phi}_w$ is an isometry between 
$L^2_\beta(S)$ and $L^2_\beta(S_w).$ Moreover $\widehat{\Phi}_w$ also preserves $H^1_\beta$ in the sense that 
$\widehat{\Phi}_w(u) \in H^1_\beta(S_w)$ if and only if $u \in H^1_\beta(S).$
 
We may thus define $\Ecal_{\beta,w}$ the quadratic form obtained by pulling-back $\Ecal_\beta$ on $S_w$ using 
$\widehat{\Phi}_w.$ This quadratic form is then closed on the domain $H^1_\beta(S)$, 
and for each $u\in H^1_\beta(S)$
\[   \Ecal_{\beta,w}(u)~=~ 
  \int_{S}  w^{-2}\cdot \left| \partial_x u(x,y)\right |^2\,
       +\, \left| \partial_y u(x,y)\right |^2\, dxdy.
\]

\begin{defn}\label{def:represent}
Let $\Omega$ be a domain that has $S$ as a cusp  and $\beta > \ua > y_0.$ 
Let $q$ be a quadratic form closed over the domain $H^1_\beta(\Omega).$ 
We will say that $q$ represents the cusp of width $w$ for $y\geq \ua $ if, for any 
$u\in H^1_\beta$ that is supported in $\{ y> \ua \}$ we have 
\[
q(u) = \Ecal_{\beta,w}(u).
\] 
For such a quadratic form, we will say that an eigenfunction $u$ is a cuspform 
if $u^0(y)$ vanishes on $\{ y_0 \leq y \leq \beta \}.$
\end{defn}

\vspace{.25cm}
The aim of this section is to prove that being a cuspform is a 
real-analytic condition. To make this statement precise we have to consider a family 
$q_t$ of quadratic forms that satisfies the following assumptions.

\begin{ass}\label{ass:qt}
 Let $t_-<t_+$ and $\beta>\ua>y_0$. For each $t \in I:=~]t_-, t_+[$, let $w_t$ be a positive real-analytic function 
on $I.$ Let $q_t$ denote a nonnegative, closed quadratic form with domain $H^1_{\beta}(S)$ 
that represents the cusp of width $w_t$ for $y\geq \ua.$

Lastly, we assume that the family $t \mapsto q_t$ is real-analytic
of type (a)  in the sense of \cite{Kato}. That is, for each 
$u \in H^1_{\beta}(S)$, the map $t \mapsto q_t(u)$ is real-analytic.
\end{ass}

A straightforward application of analytic perturbation 
theory---\cite{Kato} \S VII---gives the following.

\begin{thm}[Existence of a real-analytic eigenbasis] \label{thm:eigenbasis_exist}
Let $t \mapsto q_t$ satisfy the assumptions above \ref{ass:qt}.
Then there exist a collection of real-analytic paths 
$\{t \mapsto u_{j,t} \in L^2(\Omega,dm)~ |~ j \in \Zbb^+ \}$ 
and a collection of real-analytic functions 
$\{t \mapsto \lambda_{j,t} \in \Rbb~ |~ j \in \Zbb^+ \}$ 
so that for each $t$, the set $\{u_{j,t}~ |~ j \in \Zbb \}$ is an orthonormal 
basis for $L^2_{\beta}(\Omega,dm)$,  and for each $(j,t)$, the function $u_{j,t}$
is an eigenfunction of $q_t$ with eigenvalue
$\lambda_{j,t}$.
\end{thm}

\begin{proof}
Since the embedding from $H^1_\bet$ into $L^2_\beta$ is compact, for any $t$ the spectrum of 
$q_t$ consists only in eigenvalues.
The proof is then similar to the proof of Theorem 3.9 in \cite{Kato} \S VIII.3.5.
\end{proof}

For $u \in W_{\beta}$, define 
\[  L(u)~ =~ \lim_{y\rightarrow \beta^-}~ \frac{u^0(y)}{y-\beta}.
\]

\begin{lem} \label{lem:nocf_deriv_criterion}
An eigenfunction $u$ of $q_t$ is a cusp form if and only if $L(u)=0$. 
\end{lem}

\begin{proof}
$L(u)$ is the left-sided derivative of $u^0$ at $\beta$. Since $u$ is an eigenfunction,
$u \in W_{\beta}$ and $u^0$ is a solution to a second order ordinary differential equation
on $[y_0,\beta]$ with $u_0(\beta)=0$. Thus, $u^0$ vanishes identically on $[y_0,\beta]$ if and only if 
$L(u)=0$. 
\end{proof}

For real-analytic eigenbranches we have the following.

\begin{lem}\label{lem:Lanalytic}
Let $(u_t,\lambda_t)$ be an analytic eigenbranch of $q_t$ then the
mapping $t\mapsto L(u_t)$ is analytic on $]t_-,t_+[.$
\end{lem}
\begin{proof}
 The zeroth mode $u^0_t$ of $u_t$ is a solution to the ODE 
\[
-u''~  -~ \frac{\lambda_t}{y^2} \cdot u~ =~ 0, 
\] 
on $[\ua,\beta]$ with Dirichlet boundary condition at $\beta.$ 
Denote by $G_\lambda$ the unique solution to this ordinary differential equation 
that satisfies $G_\lambda(\beta)=0, ~ G_\lambda'(\beta)=1.$ Since the coefficients 
of the ordinary differential equation 
depend analytically on the parameter $\lambda$, the 
mapping $\lambda \mapsto G_\lambda$ is analytic 
(with values in $\Ccal^2([\ua,\beta])$ for instance). Moreover, for each 
compact set $K\subset ]t_-,t_+[$  we can find $\ua<\alp_K< \beta$ such that, 
for each $t\in K$, $\int_{\alp_K}^\beta G_\lambda >0$. 
Since $u_t$ is a multiple of $G_{\lambda_t}$, we then have 
\begin{equation*}
\begin{split}
L(u_t) & = \left( \int_{\alp_K}^\beta G_{\lambda_t}(y) \, dy  \right)^{-1}\cdot \int_{\alp_K}^\beta u_t^0(y) \, dy,\\
& = \left( \int_{\alp_K}^\beta G_{\lambda_t}(y) \, dy  \right)^{-1}\cdot \int_0^1 \int_{\alp_K}^\beta u_t(x,y) \,dx dy, . 
\end{split}
\end{equation*}
Analyticity on $K$ then follows from the analyticity of $t\mapsto u_t$ 
and $t \mapsto G_{\lambda_t}$ and the choice of $\alp_K$.  
\end{proof}

We will say that real-analytic eigenfunction branch $u_t$ of $q_t$ 
is {\em a cusp form eigenbranch} if and only if for each $t \in I$,
the eigenfunction $u_t$ is a cusp form (see Definition \ref{def:represent}). 
Using the real-analyticity proved in Lemma \ref{lem:nocf_deriv_criterion} we obtain the 
following  

\begin{coro}  \label{coro:cf_discrete}
If $u_t$ is a real-analytic eigenfunction branch that is not a cusp form eigenbranch, 
then the set of $t \in I$ such that $u_t$ is a cusp form is discrete. 
\end{coro}

We now proceed to prove that if $q_t$ has no real-analytic cusp form eigenbranch then, 
for a generic $t$, the form $q_t$ has no cusp form. 
As it turns out, we will 
actually first prove that the spectrum of $q_t$ is generically 
simple.

Let $I_{{\rm mult}}$ denote the set of $t \in I$ such that 
$q_t$ has an eigenspace of dimension at least two.

\begin{prop}  \label{prop:mult}
 If $q_t$ does not have a real-analytic cusp form eigenbranch,
 then $I_{{\rm mult}}$ is countable.   
\end{prop}

\begin{proof}
Let $\{u_{j,t}~ |~ j \in \Zbb^+,~ t \in I\}$ 
and $\{\lambda_{j,t}~ |~ j \in \Zbb^+,~ t \in I\}$ be as in 
Theorem \ref{thm:eigenbasis_exist}. 
For each $j, k \in \Zbb^+$, let $Z_{j,k}=\{ t~ |~ \lambda_{j,t}=\lambda_{k,t}\}$.
Since each eigenspace of $q_t$ is spanned by a finite collection of 
$\{u_{j,t}\}$, the union $\bigcup_{j,k} Z_{j,k}$ equals $I_{{\rm mult}}$.

The function $t \mapsto \lambda_{j,t}-\lambda_{k,t}$ is analytic,
and hence $Z_{j,k}=\{ t~ |~ \lambda_{j,t}=\lambda_{k,t}\}$
is either countable or equals $I$. 
Thus to prove the claim, it suffices to show that it is not possible 
for $Z_{j,k}$ to equal $I$.

Suppose that there exists $j$ and $k$ so that $u_{j,t}$ and $u_{k,t}$ are real-analytic
eigenbranches so that $\lambda_{j,t}=\lambda_{k,t}$ for each $t \in~ I$.
To prove the proposition, it suffices to produce a 
linear combination $u_t$ of $u_{j,t}$ and $u_{k,t}$ so that for each $t \in I$,
the function $u_t$ is a real-analytic cusp form eigenbranch.

By hypothesis, neither $u_{j,t}$ nor $u_{k,t}$ are cusp form eigenbranches. 
By Corollary \ref{coro:cf_discrete}, the set $J$ of $t$ such that 
either $u_{j,t}$ or $u_{k,t}$ is a cusp form is discrete. 

For each $t \notin J$, define 
\[ u_t~=~  \frac{L(u_{k,t})  \cdot u_{j,t}~ -~ L(u_{j,t}) \cdot u_{k,t}}{
                \sqrt{ L(u_{j,t})^2 + L(u_{k,t})^2}}.         \]
It suffices to show that $t \rightarrow u_t$ extends to a real-analytic function on $I$.
Indeed, since $L$ is linear, we have $L(u_t)=0$ for each $t \notin J$. 
By Corollary \ref{lem:Lanalytic}, the real-analytic extension would satisfy $L(u_t)\equiv 0$.

The order of vanishing of 
$t \mapsto L(u_{j,t})$ (resp. $t \mapsto L(u_{k,t})$)  is finite
at each $t \in J$. If the order of vanishing of $L(u_{k,t})$ at 
$t_0 \in J$ is at least 
the order of vanishing of $L(u_{j,t})$ at $t_0$, then the ratio
$L(u_{k,t})/L(u_{j,t})$ has a real-analytic extension near $t_0$. Hence, 
factorizing $L(u_{j,t})$ we obtain that 
\begin{equation*}
\begin{split} 
\frac{L(u_{k,t})}{\sqrt{ L(u_{j,t})^2 + L(u_{k,t})^2}}~ & =~
 \frac{L(u_{k,t})}{L(u_{j,t})}\cdot
   \left( 1 +   \left(\frac{L(u_{k,t})}{L(u_{j,t})}\right)^2 \right)^{-\frac{1}{2}},~\mbox{and}\\
\frac{L(u_{j,t})}{\sqrt{ L(u_{j,t})^2 + L(u_{k,t})^2}}~ & =~
   \left( 1 +   \left(\frac{L(u_{k,t})}{L(u_{j,t})}\right)^2 \right)^{-\frac{1}{2}}
\end{split}
\end{equation*}
have real-analytic extensions near $t_0$. If the order of vanishing of $L(u_{k,t})$ 
at $t_0$ is at most the order of vanishing of $L(u_{j,t})$ at $t_0$, then 
a similar argument applies by factorizing $L(u_{k,t})$ everywhere. 
Thus, $u_t$ extends to a real-analytic cusp form eigenbranch.
\end{proof}

Let $I_{{\rm cf}}$ denote the set of $t \in I$ such that 
$q_t$ has at least one cusp form eigenfunction.

\begin{prop}  \label{alt:alt}
If $q_t$ has no real-analytic cusp form eigenbranch,
then $I_{{\rm cf}}$ is countable.
\end{prop} 

\begin{remk}[Dichotomy]  
If there exists a cusp form eigenbranch, then $I_{{\rm cf}}=~I$.
Therefore, we have the following dichotomy:
Either the set $I_{{\rm cf}}$ countable or the family $t \mapsto q_t$ 
has a real-analytic cusp form eigenbranch.
\end{remk}

\begin{proof}[Proof of Proposition \ref{alt:alt}]
Let $\{u_{j,t}|~ j \in \Zbb,~ t \in I\}$ be as in Theorem \ref{thm:eigenbasis_exist}.
By Corollary \ref{coro:cf_discrete}, the zero set $Z_{j}=\{t~ |~ L(u_{j,t}=0\,)\}$ 
is countable. 

If each eigenspace $E$ of $q_t$ is one-dimensional, then there exists a unique 
$j$ such that $E$ equals the span of $u_{j,t}$. Thus, if $t$ 
does not belong to  $I_{{\rm mult}}$ or to any $Z_{j}$, 
then $t$ does not belong to $I_{{\rm cf}}$.  
In other words, $I_{{\rm cf}} \subset  \left(\bigcup Z_{j}\right) 
\cup I_{{\rm mult}}$. By Proposition \ref{prop:mult}, the set $I_{{\rm mult}}$  is countable, 
and hence so is $I_{{\rm cf}}$.
\end{proof}


\section{Perturbation theory for hyperbolic triangles with one cusp}

\label{section:alltriangles}

In this section we use the results of the previous section
to explain how Theorem \ref{thm:GenNocf} can be deduced from
the existence of a triangle with a cusp that has no nonconstant Neumann 
eigenfunctions. This fact might be considered to be folklore as it follows from
the general philosophy of using analyticity to prove generic spectral 
results (see \cite{HJ09}). The main task here is to construct a real-analytic family 
of quadratic forms that is associated with each real-analytic path in the moduli space of triangles.
    
\subsection{The moduli space of triangles} 
First, we discuss the parametrization of the set of triangles with one cusp.
The statement of Theorem \ref{thm:GenNocf} makes use of the fact
that hyperbolic triangles with one cusp are parametrized by the two nonzero 
vertex angles. But in order to prove Theorem \ref{thm:GenNocf}, it will be more
convenient to use an alternate set of parameters.

For each geodesic triangle $T$ in the hyperbolic upper half plane $\Hbb^2$ 
having (exactly) one cusp, there exists a unique $c \in~ ]0,1[$ and $w \in [2c,1+c[$ 
so that $T$ is isometric to the domain
\begin{equation} \label{eq:triparameter}
\Tcal_{c,w}~ =~ \left\{  (x,y)~ |~  0 \leq x \leq w,~ \   (x-c)^2+y^2 > 1  \right\}.
\end{equation}
See Figure \ref{TriGeneric}.

\begin{figure}[h]   
\begin{center}
\includegraphics[totalheight=2.5in]{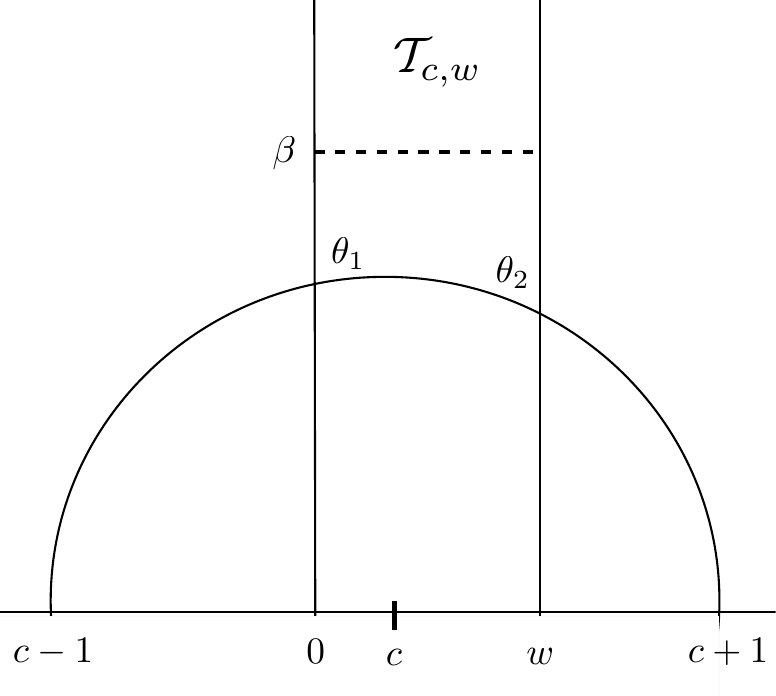}
\end{center}
\caption{\label{TriGeneric}The triangle $\Tcal_{c,w}$ in the upper half plane.}
\end{figure}

In this way, the set of hyperbolic triangles may be identified
with the Euclidean triangle 
\[ \Mcal~ =~ \{ (c,w)~ |~ 0 \leq c < 1,~  2c < w < c+1 \}. \]
We say that a subset of the set of triangles has measure zero
if and only if the corresponding subset of $\Mcal$ has measure zero.
Similarly, a subset of the set of triangles is said to be a real-analytic
curve if and only if the corresponding subset of $\Mcal$ is a real-analytic
curve.

These notions are equivalent to those used in the statement of 
Theorem \ref{thm:GenNocf} because the relationship between the angles 
$(\theta_1, \theta_2)$ and the parameters $(c,w)$ is real-analytic. 
Indeed, we have $c = \cos(\theta_1)$ and $\cos(\theta_2)= w-c$. 
See Figure  \ref{TriGeneric}.

To prove Theorem \ref{thm:GenNocf}, we will apply perturbation theory.
The following fact makes this approach feasible.

\begin{prop} \label{prop:allcuspform}
Each nonconstant Neumann eigenfunction on $\Tcal_{c,w}$
is a cusp form and hence an eigenfunction of the modified quadratic
form $\Ecal_{\beta}$. 
\end{prop}

\begin{proof}
The eigenvalue of a nonconstant Neumann eigenfunction on $T$ is at least 
1/4 \cite{JudgeLamin}.\footnote{See also 
\cite{SarnakThesis} for the case of triangles that are fundamental 
domains for the Hecke groups.} Thus, the claim follows from  
Corollary \ref{coro:cuspform} and  Lemma \ref{lem:relation}.
\end{proof}

Let $\Mcal_{{\rm cf}}$ denote the set of $(c,w) \in \Mcal$ 
such that there exists $\beta>1$ so that the modified quadratic
from $\Ecal_{\beta}$ has a cusp form. 
To prove Theorem \ref{thm:GenNocf} it will suffice 
to show that $\Mcal_{{\rm cf}}$ has measure zero and is 
a countable collection of nowhere dense sets.

\subsection{A family of diffeomorphisms} \label{sec:family-diffeo}

To show that $\Mcal_{{\rm cf}}$ is non-generic,
 we will use analytic perturbation theory and Proposition
\ref{alt:alt}. In order to use analytic perturbation theory we will have
to normalize the Hilbert space and the domains of the quadratic forms. 
To accomplish this, we let $S=[0, 1] \times [1, \infty[$ and 
for each $(c,w)$ we define a $C^1$ diffeomorphism
$\varphi_{c,w}: \Tcal_{c,w}  \rightarrow  S $
such that\footnote{
In \cite{HJ10}, we considered a simpler mapping from $\Tcal_{0,t}$ onto $S$. 
The mapping that we define here is more complicated because 
it must preserve the notion of zeroth Fourier coefficient 
for all $y$ above some point. In particular, the vertical 
displacement of vertical lines should not depend on $x$ for large $y$. 
In  \cite{HJ10}, we considered Dirichlet boundary conditions, and in that context
there is no need to truncate the zeroth Fourier coefficient.}
\begin{enumerate} 
   \item The restriction of $\varphi_{c,w}$ is the identity for 
         $y> \underline{\alpha}=(\beta+1)/2$.
   \item For each path $\gamma: I \rightarrow \Mcal$, the family 
         $t \mapsto \varphi_{c,w}$ is a real-analytic path.
\end{enumerate}

To construct $\varphi_{c,w}$, we use the fact that the map $(x,y) \mapsto x$ defines 
a fibration of $\Tcal_{c,w}$ over $[0,w]$ and a fibration of $S$ over $[0,1]$. 
We define $\varphi_{c,w}$ by sending the fiber over 
$\{x\}$ onto the fiber over $\{x/w\}$. 

\begin{lem}  \label{lem:defnB}
For each $\alpha \in~ ]0,\ua[$, there exists a unique cubic polynomial 
$B_{\alpha}$ so that 
\begin{itemize}
  \item $B_{\alpha}(\alpha)=1$,
  \item $B_{\alpha}'(0)= \alpha$.
 \item $B_{\alpha}(\underline{\alpha})=\underline{\alpha}$,
  \item $B_{\alpha}'(\underline{\alpha})=1$
\end{itemize}
The coefficients of $B_{\alpha}$ are real analytic for $\alp \in ]0,\ua[$.

Moreover if $\ua\,>\,2+\sqrt{3}$ then , for all $\alp\in (0,1],$ and $y\in
[0,\ua]$, we have $B_{\alpha}'(y)>0$. 

\end{lem}

\begin{proof}
Since $\ua\neq 0,$ satisfying the two conditions on $B'_\alp$ is equivalent to the 
existence of some $A$ such that 
\[
B'_\alp(y) \,=\, A\cdot y(\ua-y) \,+\,\frac{1}{\ua}\cdot y
\,+\,\frac{\alp}{\ua}\cdot (\ua-y).
\]
Denote by $Q_\alp$ the cubic polynomial defined by 
\[
Q_\alp(y) \,=\, \dis \int_{\alp}^y z(\ua-z)\, dz.
\]
By integration, there exists some $C$ such that 
\[
B_\alp(y)\,=\, A \cdot Q_\alp(y) \,+\, \frac{1}{\ua} \cdot \frac{y^2}{2} -\frac{\alp}{\ua}\cdot\frac{(\ua-y)^2}{2}\,+\,C. 
\] 
Evaluating at $\alp$ and using the condition on $B_\alp(\alp)$ we find 
\begin{equation*}
C = 1- \frac{\alp^2}{2\ua}\,-\,\frac{\alp(\ua-\alp)^2}{2\ua}.
\end{equation*}
Observe that $Q_\alp(\ua)\neq 0$ if $\alp\in [0,\ua[$ and, under this condition, we can solve the last equation 
on $B_\alp$ to find $A.$ We obtain   
\begin{equation*}
\begin{split}
Q(\ua)A & =\,\ua - \frac{\ua}{2}-C\\
& = \frac{1}{2\ua} \left[ \ua^2\,+\,\alp^2-\alp(\ua-\alp)^2-2\ua\right]\\
& = \frac{1-\alp}{2\ua} \left[ \alp^2+\ua^2-2\ua(\alp+1)\right],\\
& = \frac{(1-\alp)\left[ \alp^2 -2 \ua \alp + \ua^2-2\ua\right ]}{2\ua}
\end{split}
\end{equation*}

It follows that we  have a unique solution provided $\ua \neq 0$
and $0<\alp<\ua,$ and that the coefficients are real-analytic in
$\alp.$ 

We now check the last statement.
For $\alp=1$, we have $B_\alp(y)=y$ so that the claim follows. 
For $\alp<1$  we observe that the numerator of $A$ is a cubic polynomial that has three
roots at $1, \ua \pm \sqrt{2\ua}.$ Thus, if $\ua > 2+\sqrt{3}$ then $1$ is
the smallest root. Since this cubic polynomial is positive for large
negative $\alp$ and the denominator also is positive, 
it follows that $A$ is positive for $0<\alp<1.$ So 
$B_\alp'$ is a concave function, and by construction $B'_\alp(0)>0$ and
$B'_\alp(\ua)>0$. The claim follows.
\end{proof}

\begin{nota}
We will use the notation $B_\alp(y)$ as well as the notation $B(\alp,y).$
\end{nota}

Define $F_{c,x}: \Rbb \rightarrow \Rbb$ by 
\[ F_c(x,y)~ =~ \left\{ \begin{array}{cl}
     B \left(f_c(x), y \right) & \mbox{ if } y \leq \underline{\alpha} \vspace{.2cm} \\
     y &  \mbox{ if } y \geq \underline{\alpha}. \end{array} \right. %
\]
where
%
%
%
%
%
\[   f_c(x)~ =~ \sqrt{1- (x-c)^2}.
\]
Define $\varphi_{c,w}: \Tcal_{c,w} \rightarrow S$ by 
\[ \varphi_{c,w}(x,y)~ =~
   \left(x/w, F_c(x,y)\right).
\]
Observe that the conditions on $B$ imply that $F, \,\partial_x F_c$ and 
$\partial_y F_c$ are continuous on $\Tcal_{c,w}$ so that
$\varphi_{c,w}$ is $C^1.$

We will use this function $\varphi_{c,w}$ to normalize the triangle
$\Tcal_{c,w}.$ This is made possible by the following lemma.

\begin{lem}  \label{lem:Bprop}
Suppose that $\partial_y B(f_c(x),y)\,>\,0$ for each $(x,y) \in
\Tcal_{c,w} \cap \{ y\leq \ua\}$ then the map $\varphi_{c,w}$ 
is a $C^1$ diffeomorphism from $\Tcal_{c,w}$ onto $S$. \\
In particular, for each $\ua>2+\sqrt{3}$ and each $(c,w)\in \Mcal$, the
mapping $\varphi_{c,w}$ is a $C^1$ diffeomorphism from $\Tcal_{c,w}$ onto $S.$
\end{lem}

\begin{proof}
It suffices to show that the map 
$F_{c,x}$ is a $C^1$ diffeomorphism from $[f_c(x), \infty[$ onto $[1, \infty[$.
By assumption $\partial_yB\left(f_c(x),y\right )>0$ for each $x$.   
We have $B(f_c(x),f_c(x))\,=\,1$,
$B(f_c(x),\underline{\alpha})=\underline{\alpha}$ and
$\partial_yB(f_c(x),\ua) =1$. 
Since $F_{c,x}$ is the identity for $y > \underline{\alpha}$, we find that
$F_{c,x}$ is a $C^1$ diffeomorphism from $[f_c(x), \infty[$ onto $[1, \infty]$. 
\end{proof}

For each $\ua$ and each $M\subset \Mcal$,  we define $X_{\ua,M}$ and $A_{\ua,M}$ by 
\begin{gather*}
X_{\ua,M} := \left \{ (x,y,c,w)~|~ (c,w)\in M,~ (x,y)\in \Tcal_{c,w},~
  y\leq \ua \right\},\\
A_{\ua,M}:= \left \{ (a,b,c,w)~|~ (c,w)\in M,~ (a,b)\in S,~
  b\leq \ua \right\}.
\end{gather*}

We then have
\begin{lem}  \label{lem:phi_analytic}
For each $\ua, M$, each of the following maps is analytic on $X_{\ua,M}$:
\begin{enumerate}
\item  $(x,y,c,w) \mapsto \varphi_{c,w}(x,y)$  
\item  $(x,y,c,w) \mapsto \partial_x \varphi_{c,w}(x,y)$ 
\item  $(x,y,c,w) \mapsto \partial_y \varphi_{c,w}(x,y)$
\end{enumerate}
If, for each $(c,w)\in M,$ the assumption of Lemma \ref{lem:Bprop}
holds then  the map $(a,b,c,w) \mapsto \varphi_{c,w}^{-1}(a,b)$ is
also analytic on $A_{\ua,M}$.\\
Moreover, each restriction extends analytically to an open neighbourhood. 
\end{lem}

\begin{proof}
The coefficients of the cubic polynomial $B_{\alpha}$
depend analytically on $\alpha$ and hence 
$(\alpha, y) \mapsto B(\alp,y)$ is analytic.
The map $(c,x) \mapsto f_c(x)$ is analytic 
and hence it follows that map (1) is analytic.
Maps (2) and (3) are therefore analytic.

Since $(\alpha, y) \mapsto B(\alp,y)$ is analytic 
and $\partial_y B(\alp, y)>0$ for $y>0$, the implicit
function theorem (Theorem 2.1.2 in \cite{Horm}) implies that there exists a function 
$(\alpha, b) \mapsto Y_{\alpha}(b)$ which is analytic and a solution
to 

\[ B_{\alpha}(Y_{\alpha}(b))~ -~ b~ =~ 0. \]
We then have
\[ \varphi_{c,w}^{-1}(a,b)~ =~  \left(w\cdot a, Y_{f_c(w \cdot a)}(b) \right), 
\] 
and, since $(c, x) \mapsto f_c(x)$ is analytic, 
the claim follows.
\end{proof}

In the rest of the section, $\ua>2+\sqrt{3}$ will be fixed so that we
can use lemmas \ref{lem:Bprop} and \ref{lem:phi_analytic}.

\subsection{The quadratic form with fixed domain} \label{sec:quadratic-form}

We use the family of diffeomorphisms $\varphi_{c,w}$ to define
a quadratic form $q_t$ with domain $H^1_{\beta}(S) \subset L^2_{\beta}(S)$  
that is unitarily equivalent to  $\Ecal_{\beta}$ on 
$H^1_{\beta}(\Tcal_{c,w}) \subset L^2_{\beta}(\Tcal_{c,w})$.

Define $\Phi_{c,w}: L^2(S, da \, db/b^2) \rightarrow L^2(\Tcal_{c,w}, dx \, dy/y^2)$ by 
\[  \Phi_{c,w}(u)~ =~ 
     y \cdot \sqrt{|\det({\rm Jac}(\varphi_{c,w}))|} \cdot
      \left(\frac{u}{b} \circ \varphi_{c,w} \right) \]
where ${\rm Jac}$ is the operator that returns the Jacobian matrix of a map.

\begin{lem} \label{lem:domain}
The isometry $\Phi_{c,w}$ is a unitary isomorphism from 
$L^2_{\beta}(S)$ onto $L^2_{\beta}(\Tcal_{c,w})$ and it maps
$H^1_{\beta}(S)$ onto $H^1_{\beta}(\Tcal_{c,w})$. \\
On functions that are supported in $b\geq \ua$, $\Phi_{c,w}$ coincides 
with $\widehat{\Phi}_w.$
\end{lem}

\begin{proof}
We have 
\begin{eqnarray*} 
 \int_{\Tcal_{c,w}} |\Phi_{c,w}(u)|^2~ \frac{dx~ dy}{y^2}
 & = & \int_{\Tcal_{c,w}} \left( \frac{u}{b} \circ \varphi_{c,w} \right)^2 
   |\det({\rm Jac}(\varphi_{c,w}))|\cdot y^2 \cdot \frac{dx~ dy}{y^2} \\
 & =&  \int_{S} \left( \frac{u}{b} \right)^2~ da~ db. 
\end{eqnarray*}
It follows that $\Phi_{c,w}$ is a unitary isomorphism from 
$L^2_{\beta}(S)$ onto $L^2_{\beta}(\Tcal_{c,w})$.

Let $u\in H^1_\beta(S).$ Since $\varphi_{c,w}$ is a $\Ccal^1$ diffeomorphism and 
$\sqrt{|\det({\rm Jac}(\varphi_{c,w}))|}$ is continuous on $\Tcal_{c,w}$ and smooth away 
from $y=\beta$, then $\Phi_{c,w}(u)$ is continuous and in
$H^1_{\beta}(\Tcal_{c,w}\setminus \{ y=\beta\})$. The jump formula
implies that  $\Phi_{c,w}(u)\in H^1_{\beta}(\Tcal_{c,w}).$ \\
Since, for $y>\ua$, $\varphi_{c,w}(x,y) \,= \,(\frac{x}{w},y)$, the last
statement is a direct verification. 
\end{proof}

\begin{defn}
Define the quadratic form $q_{c,w}$ on $H^1_{\beta}(S) \subset L^2(S, da \,db/b^2)$ by
\begin{equation*}
q_{c,w}(u)~  :=~ \Ecal_{\beta} \circ \Phi_{c,w}(u). 
\end{equation*}
\end{defn}

\begin{lem}  \label{lem:cuspequivalence}
$u$ is a cusp form for $q_{c,w}$ if and only if $v=\Phi_{c,w} \circ u$ is a cusp
form for $\Ecal$ on $\Tcal_{c,w}$.
\end{lem}

\begin{proof}
If $y\geq \underline{\alp}$, then  $\varphi_{c,w}(x,y)=(x/w,y)$. It follows
that if $y\geq \underline{\alp}$, then $u^0(y)=0$ if and only if $v^0(y)=0$.
For $y\geq 1$, the function $v^0$ is a solution to a second order ordinary
differential equation, and hence $v^0(y)=0$ for $y \geq \underline{\alp}$
if and only if $v^0(y)=0$ for $y \geq 1$.
\end{proof}

It will be convenient to have the following alternate form for $q_t$.

\begin{prop} \label{prop:alternate}
We have 
\begin{equation}\label{eq:defq0}
q_{c,w}(u)~ =~ \int_{S}  
\nabla (\rho_{c,w} \cdot u) \cdot 
Q_{c,w} \cdot  \overline{\nabla(\rho_{c,w} \cdot u)}^*~  da \, db
\end{equation}
where $\rho_{c,w}: S \mapsto \Rbb$ is defined by
\[   \rho_{c,w}~ =~  
  \frac{\left( y \cdot \sqrt{|\det({\rm Jac}(\varphi_{c,w})|} \right) \circ \varphi_{c,w}^{-1}}{b},  \]
and $Q_{c,w}: S \rightarrow GL_2(\Rbb)$ is defined by 
\begin{equation}\label{eq:defQ}
Q_{c,w} \circ \varphi_{c,w}~ =~  
    \frac{1}{\det{({\rm Jac}(\varphi_{c,w})})} 
    \cdot {\rm Jac}(\varphi_{c,w}) \cdot {\rm Jac}(\varphi_{c,w})^*.
\end{equation}
Moreover, $q_{c,w}$ represents the cusp of width $w$ for $y\geq \ua.$ 
\end{prop}

\begin{proof}
This is a straightforward calculation using the chain rule and the change 
of variables formula. 
\end{proof}


\subsection{Analytic paths in $\Mcal$} \label{subsec:analyticity}

Let $I =~ ]t_-,t_+[$ and let $\gamma: I \rightarrow \Mcal$ be a  
real-analytic path. 

\begin{thm}  \label{thm:analyticity}
The family of quadratic forms $t \mapsto q_{\gamma(t)}$ is analytic of type (a)
in the sense of \cite{Kato}.
\end{thm}

\begin{proof}
For each $t$, the 
quadratic form $q_{\gamma(t)}= \Ecal_{\beta} \circ \Phi_{\gamma(t)}$
is a closed form with domain $H_{\beta}^1(S)$.  It suffices to show that
for each $u \in H^1_{\beta}(S)$, the function $t \mapsto q_{\gamma(t)}(u)$ is
real-analytic. 

By Proposition \ref{prop:alternate}, we have 
\begin{equation} \label{eq:splitintegral} 
   q_{c,w}(u)~ =~ \int_{1}^{\underline{\alpha}} \int_{0}^1~   I_t~  da \, db~ 
      +~ \int_{\underline{\alpha}}^{\infty} \int_{0}^1~ I_t~ da\, db.
\end{equation}
where
\[  I_t~ =~ \nabla (\rho_{\gamma(t)} \cdot u) \cdot 
Q_{\gamma(t)}  \cdot \overline{\nabla (\rho_{\gamma(t)} \cdot u)}^*.\]
If $(a, b) \in [0,1] \times [ \underline{\alpha}, \infty[$, then  
the matrix $Q(a,b)$ is given by 
\[
Q\,=\,
\left( 
\begin{array}{cc}
\frac{1}{w_t^2} & 0\\
0& 1 
\end{array}
\right).
\]
and $\rho_{\gamma(t)}(a,b)= 1$.
Thus, the second integral on the right of (\ref{eq:splitintegral})
depends analytically on $t$.

It remains to consider the integral over $[0,1]\times [1,\underline{\alpha}]$.
The integrand $I_t$ can be expanded into a finite 
sum of terms of the form
\begin{equation} \label{eq:term}
\int_1^{\ua}\int_0^1 w(a,b)\cdot H(t,a,b) da db,
\end{equation}
where $H$ is a function that is obtained by multiplying   
$\rho,$ or its derivatives and the entries of $Q$ and $w$ is one of the  $L^1$ 
functions obtained by making the product $v_1v_2$ where both $v_i$ are 
either $u$ or one of its partial derivatives. 

By Lemma \ref{lem:phi_analytic},  the coordinates of 
$\varphi_{c,w}$ and $\varphi^{-1}_{c,w}$
are analytic functions of $(c,w)$. It follows that 
$(t,a,b) \mapsto \rho_{\gamma(t)}(a,b)$
and $(t,a,b) \mapsto Q_{ij}(t)(a,b)$ are analytic
(in a neighborhood of $I \times [0, 1] \times [1, \underline{\alpha}]$). 
In all possible choices, the function $H$ then is analytic.

The analyticity of $t \mapsto q_{\gamma(t)}(u)$
follows from Lemma \ref{lem:intanalytic} below.
\end{proof}

\begin{lem} \label{lem:intanalytic}
If $H: I \times [0, 1] \times [1, \underline{\alpha}]$ is analytic, 
then for each $p \in L^1([0, 1] \times [1, \underline{\alpha}])$,
the function
\begin{equation} \label{eq:intanalytic}
 t~ \longmapsto~ \int_{1}^{\underline{\alpha}} \int_{0}^1 p(a,b) \cdot
   H(t,a,b)~ dadb
\end{equation}
is analytic on $I$.
\end{lem}

\begin{proof}
There exists an open neighborhood $U \subset \Cbb^3$ of 
$I \times [0, 1] \times [1, \underline{\alpha}]$
such that the map $h$ extends to a holomorphic function on $U$.
Since $[0,1] \times [1,\underline{\alpha}]$ is compact, 
\[ \frac{H(t,a,b)-H(s,a,b)}{t-s}  \]
converges uniformly to $\frac{d}{dt} H(s,a,b)$
as $t$ approaches $s$. It follows that 
the (complex) $t$-derivative of the map in (\ref{eq:intanalytic})
exists at each $t \in U$.
\end{proof}

\subsection{Generic absence of cusp forms}\label{subsec:generic}

Given Theorem \ref{thm:analyticity}, we now explain why
the generic triangle $\Tcal_{c,w}$ has no cusp forms
provided that one triangle has none. 
  
\begin{thm}
If there exists a point $(c_0,w_0) \in \Mcal$ such that $\Ecal$ on 
$L^2_{\beta}(\Tcal_{c_0,w_0},dm)$ has no nonconstant eigenfunction, 
then $\Mcal_{{\rm cf}}$ has measure zero and is a countable
union of nowhere dense sets.
\end{thm} 

\begin{proof}
By Proposition \ref{prop:allcuspform}, the quadratic form $\Ecal_{\beta}$ 
on $L^2_{\beta}(\Tcal_{c_0,w_0},dm)$ has no cusp form, and hence by Lemma
\ref{lem:cuspequivalence}, the quadratic form $q_{c_0,w_0}$
has no cusp form. 

To show that $\Mcal_{{\rm cf}}$ has measure zero, we apply 
Fubini's theorem in a fashion similar to \cite{HJ09}:
Let $\gamma_{c_0}(t)= (c_0,w_0+t)$ and apply Lemma \ref{alt:alt}
to find that the set $B$ of $w$ such that $(c_0,w)\in \Mcal_{{\rm cf}}$ 
is countable. For each $w \notin B$, let 
$\gamma_w(s)= (c_0+s,w)$ and apply Lemma \ref{alt:alt}
to find that the intersection $I_w$ of the line $\{(c,w)~ |~ c\in \Rbb\}$
with $\Mcal_{{\rm cf}}$ is countable. Hence for each $w \notin B$,
the set $I_w$ has measure zero with respect to the linear measure $da$. 
Hence, the measure of $\Mcal_{{\rm cf}}$ equals
the measure of $\bigcup_{w \in B} I_{w}$. Since $B$ is countable, 
the measure equals zero.

For $N \in \Zbb$, let $\Mcal^N_{{\rm cf}}$ be the set of $(c,w) \in \Mcal$ such that 
$\Ecal$ on $L^2(\Tcal_{c,w},dm)$ has a cusp form with eigenvalue at most $N$.
Using the continuity of $(c,w) \rightarrow q_{c,w}$ and the continuity of 
linear functional $L$, one can show that $\Mcal^N_{{\rm cf}}$ is closed. 
Thus, it suffices to show that $\Mcal^N_{{\rm cf}}$
is nowhere dense. 

Given a point $(c,w) \in \Mcal^N_{{\rm cf}}$, let  
$\gamma:~ [0,1] \rightarrow \Mcal$ be a real-analytic path  
joining $(c_0,w_0)$ to $(c,w)$. Since $\Ecal_{\beta}$ on 
$L^2_{\beta}(\Tcal_{c_0,w_0},dm)$ has no cusp forms, the family
$t \mapsto q_{\gamma(t)}$ has no cusp form eigenfunction 
branch. It follows from Lemma \ref{alt:alt},
that for each open neighborhood $U$ of $(c,w)$,
there exists $t\in [0,1]$ such that $\gamma(t) \in U$ 
and $q_{\gamma(t)}$ has no cusp forms.  
Hence $\Mcal^N_{{\rm cf}}$ is nowhere dense.
\end{proof}


\section{The family $\Tcal_t$}  \label{section:expansion}

In the remainder of this paper we consider the specific family of 
triangles $\Tcal_t=\Tcal_{0,t}$ defined in the introduction. 
In particular, we will study the spectral properties of $q_{0,t}$ 
for small $t$. The family $q_{0,t}$ of quadratic forms 
does not converge as $t$ tends to zero nor do its real-analytic eigenbranches.
But a simple renormalization will give convergence. 
    
Fix $\beta>1$ and $\ua$ such that $1<\ua<\beta$. Let $B$ be the
function defined in Lemma \ref{lem:defnB}. When $\alp$ tends to
$1$, the function $y\mapsto \partial_yB(\alp,y)$ converges to $1$
uniformly for $y\in[0,\ua].$ Thus, there exists $\eta$ such that if
$1-\eta\leq \alp\leq 1$ and $0\leq y \leq \ua$ then
$\partial_yB(\alp,y)\geq \frac{1}{2}.$ 
Choose $t_0$ such that $\sqrt{1-t_0^2}<\eta$
then, for each $t<t_0$ and each $(x,y)\in \Tcal_t\cap \{ y\leq \ua\}$ 
$f_0(x)<\eta$ so that $\partial_yB( f_0(x),y)>0.$ We may thus use 
Lemmas \ref{lem:Bprop} and \ref{lem:phi_analytic}. 
The methods and results of section \ref{subsec:analyticity} then apply
and we define the quadratic form $q_{0,t}$ as previously.
  
For each $t \in~ [0,t_0[$\,, define the renormalized quadratic form by
\[    q_t~ :=~  t^2 \cdot q_{0,t} \]
with domain $H^1_{\beta}(S)$. By Theorem \ref{thm:analyticity}, the family 
$t \mapsto q_t$ is real-analytic of type (a) for $t \in~ ]0,1[$. 
In particular, the results of \S \ref{subsec:generic} apply.

To study the limiting properties of the family $q_t$, we re-express
$q_t$ in a more convenient form: For each $C^1$ function $w:S \rightarrow \Cbb$
define
\begin{equation} \label{eq:gradient}
 \widetilde{\nabla}_t~ w~ = \left( \partial_x w,~ t \cdot \partial_y w \right). 
\end{equation}
Recall that $Y_{\alpha}$ is the inverse of $B_{\alpha}$ and 
set $f(x)= f_0(x)=\sqrt{1-x^2}$. Define 
\begin{equation} \label{defn:rhotilde}
 \widetilde{\rho}_t(a,b)~ =~ \frac{Y({f(ta)}b)}{b} \cdot \left( \partial_y Y(f(at),b)\right)^\und 
\end{equation}
and 
\begin{gather}\label{eq:deftQ}
 \widetilde{Q}_t(a,b) = \\
\nonumber \left( \partial_y(Y(f(at),b)\right)^{-1}   \cdot \left(
    \begin{array}{cc}  1  & \left(\partial_{\alpha}B \circ Y \right)\cdot f'(ta) \vspace{.25cm} \\   
     \left(\partial_{\alpha}B \circ Y \right)\cdot f'(ta)  & \ 
      \left(\left(\partial_{\alpha}B \circ Y \right)\cdot f'(ta)\right)^2 + 
     \left(\partial_yB \right)^2  \\
     \end{array}     \right) 
\end{gather}

where the subscript (or first argument) in each $Y$ and $B$ is $f(t\cdot a)$. 
When comparing $\rho$ and $\tilde{\rho}$ ($Q$ and $\tilde{Q}$) we see that 
we only miss some powers of $t$ that eventually cancel in the computation leading to 
Proposition \ref{prop:alternate}.
  
This shows that for each $u \in H^1_{\beta}(S)$ 
\begin{equation}  \label{eq:qdefn}
 q_{t}(u)~ =~ \int_{S_-}  \widetilde{\nabla} (\widetilde{\rho}_t \cdot u) \cdot 
\widetilde{Q}_t \cdot \widetilde{\nabla} (\widetilde{\rho}_t \cdot \overline{u})^*~  
  da~ db~ +~ \int_{S_+} \widetilde{\nabla} u \cdot \left(\widetilde{\nabla} \overline{u}\right)^*
  da \, db
\end{equation}
where $S_-=[0,1]\times [1, \underline{\alpha}]$ and $S^+=[0,1]\times [\underline{\alpha}, \infty]$.

By arguing as in the proof of Theorem \ref{thm:analyticity}, 
one can show that $t\mapsto q_{t}(u)$ is analytic at $t=0$.\footnote{
However, because $q_0$ is not closed on the domain $H^1_{\beta}(S)$,
the family $q_t$ is {\em not} analytic at $t=0$ in the sense of \cite{Kato}.}

We will now compute the first few terms in the Taylor series in $t$ for 
$\widetilde{\rho}$ and $\widetilde{Q}.$ These functions are analytic on 
a neighbourhood of $[0,t_0[\times S_-.$ In particular, in the
following, the expressions like 
$O(t^2)$ are uniform with respect to $(a,b)\in S^-$ and may be differentiated with respect 
to $t,a $ and $b$. 
 
We first compute
\begin{equation} \label{eq:fexp}
  f(ta)~ =~ 1~ -~ \und \cdot t^2 \cdot a^2~ +~ O(t^4),
\end{equation}
and
\begin{equation} \label{eq:fderivexp}
  f'(ta)~ =~ -t \cdot a~ +~ O(t^3).
\end{equation}
Since $\alpha \mapsto Y_{\alpha}$  is analytic and $Y_1(b)=b$, it follows from
(\ref{eq:fexp}) that
\begin{equation} \label{exp:Yalpha}
  Y_{f(ta)}(b)~ =~ b~ +~ O(t^2).
\end{equation}
Moreover, using analyticity, this asymptotic expansion may be differentiated with respect 
to $(a,b).$ 
We thus obtain,  
\begin{equation} \label{exp:Yderivalpha}  Y_{f(ta)}'(b)~ =~ 1~ +~ O(t^2).
\end{equation}
Substituting these into (\ref{defn:rhotilde}), and differentiating, we find that 
\begin{equation} \label{exp:rho}
 \widetilde{\rho}_t(a,b)~ =~ 1~ +~ O(t^2),~~ \nabla_{a,b} \widetilde{\rho}_t(a,b)\,=\,O(t^2) .
\end{equation}

Using (\ref{eq:fexp}), (\ref{eq:fderivexp}), (\ref{exp:Yalpha}), and 
(\ref{exp:Yderivalpha}) we find that
\begin{equation} \label{exp:Q}
\widetilde{Q}_t(a,b)~ =~ I~ +~ t \cdot a \cdot p(b)  \cdot 
 \left(\begin{array}{cc} 0 &1\\ 1 & 0 \end{array}\right)~ +~ O(t^2)
\end{equation}
where $I$ is the identity matrix, $O(t^2)$ is a matrix whose operator norm 
is bounded by a constant times $t^2$ as $t$ tends to zero, and 
$p$ is the polynomial 
\begin{equation}  \label{defn:p}
  p(b)~ =~ -\left. \partial_{\alpha} B_{\alpha}(b) \right|_{\alpha=1}. 
\end{equation}

To prove Theorem \ref{thm:main} we will need to know that $p(1)\neq0$.

\begin{lem} \label{lem:p_positive}
$p(1) = 1$.
\end{lem}

\begin{proof}
By construction we have $B(\alp,\alp)=1.$ 
By differentiating with respect to $\alp$ and setting $\alp=1$ we get 
\[
\partial_\alp B(1,1) + \partial_yB(1,1) =0.
\]
Since $\partial_yB(\alp,y)= 1+O((\alp-1)^2)$ the claim follows.
\end{proof}


\section{Asymptotic separation of variables}
\label{section:approx}

In this section we apply the method of asymptotic separation of 
variables developed in \cite{HJ10} (see also \cite{HJ10-erratum}) to the family $q_t$. Using the small $t$
asymptotics derived in \S \ref{section:expansion},
we approximate $q_t$ to first order with a family of quadratic forms
$a_t$ for which separation of variables apply.
We also derive a non-concentration estimate for eigenfunctions of $q_t$.

\begin{nota}
In this section and the following sections, we will use $(x,y)$ in place of 
$(a,b)$ as coordinates for $S=[0,1] \times [1, \infty[$ and unless it
is specified otherwise $\| \cdot  \|$ is the norm in $L^2(S, y^{-2}dxdy).$
\end{nota}

\subsection{Asymptotic approximation}

We begin by using the expansions obtained in \S \ref{section:expansion} to 
determine the forms used to approximate $q_t$.  In particular, by substituting
the expansions (\ref{exp:rho}) and (\ref{exp:Q}) 
into (\ref{eq:qdefn}) we are led to define
\begin{equation}\label{eq:defa_t}
a_t(u,v)~ =~ \int_S \widetilde{\nabla} u \cdot \widetilde{\nabla} \overline{v}~ dx \, dy~ =~ 
\int_S \left( u_x \cdot \overline{v}_x~ +~
 t^2 \cdot  u_y  \cdot \overline{v}_y \right)~ dx \, dy 
\end{equation}
and
\begin{equation}  \label{eq:defb_t}
b_t(u,v)~ =~ \int_{S_-}
\widetilde{\nabla}u \cdot \Bcal(x,y) \cdot \widetilde{\nabla} \overline{v}~  dx \, dy
\end{equation}
where the operator $\widetilde{\nabla}$ is 
defined by (\ref{eq:gradient}), $S_-=[0,1] \times [1, \underline{\alpha}]$,
and 
\[  \Bcal(x,y)~ =~  x \cdot p(y)  \cdot 
 \left(\begin{array}{cc} 0 &1 \\ 1 & 0 \end{array} \right).
\]

We wish to show that $q_t$ is asymptotic to $a_t$ at first order 
in the sense of \cite{HJ10}. It will also be used to help derive
a key estimate for crossing eigenbranches. However, although 
$a_t$ is a positive quadratic form, the bottom of its spectrum tends to $0$ 
so that it is more convenient to use the quadratic form $\widetilde{a_t}$ that we now define 
to control quantities.

\begin{defn}\label{def:tildea}
The quadratic form $\widetilde{a_t}$ is defined on $\dom(a_t)$ by 
\[
\widetilde{a_t}(v)\,=\,a_t(v)\,+\,\|v\|^2.
\]  
\end{defn}

The following proposition can be seen as the beginning of an asymptotic expansion for $q_t.$

\begin{prop}\label{prop:q-a-tb}
There exists $C$ such that for each $u,v \in H^1_\beta(S)$
\begin{equation}
\left| q_t(u,v) - a_t(u,v)-t\cdot b_t(u,v) \right | \, \leq \, 
C \cdot t^2 \cdot \wat(u)^\und \wat(v)^\und.
\end{equation}
\end{prop}

\begin{proof}
We have
\begin{equation}\label{eq:nablarhow} 
\widetilde{\nabla}_t~  \widetilde{\rho}_t \cdot u~ 
=~ \widetilde{\rho}_t \cdot \widetilde{\nabla}_t u~ 
     +~ u \cdot \widetilde{\nabla}_t \,\widetilde{\rho}_t.  
\end{equation}
If $y\geq \underline{\alpha}$, 
then $\widetilde{\rho}_t$ is identically equal to $1$ and $\widetilde{Q}_t$ 
is identically equal to $I$. Hence,  by subsituting (\ref{eq:nablarhow}) into 
(\ref{eq:qdefn}), we find that $q_t(u,v)-a_t(u,v)-t\cdot b_t(u,v)$ is 
the sum of the following four terms
\begin{eqnarray}
\label{leading} & & 
\int_{S_-} \widetilde{\nabla}_tu \cdot(\tilde{\rho}^2 \cdot \widetilde{Q}_t-I-t\cdot \Bcal) 
  \cdot \widetilde{\nabla}_tv~ dx~ dy~, \\
& & \label{mixed}
\int_{S_-} 
    \widetilde{\rho}_t \cdot v \cdot (\widetilde{\nabla}_t \widetilde{\rho}_t \cdot \widetilde{Q}_t \cdot \widetilde{\nabla}_t u)~
 dx~ dy~, \\
& & \label{mixed2}
\int_{S_-}    \widetilde{\rho}_t \cdot u \cdot 
\left( \widetilde{\nabla}_t \widetilde{\rho}_t \cdot \widetilde{Q}_t \cdot \widetilde{\nabla}_t v \right)~ 
 dx~ dy~, \\
& & \label{uv}
\int_{S_-} 
  \left(\widetilde{\nabla}_t \widetilde{\rho}_t \cdot 
\widetilde{Q}_t \cdot \widetilde{\nabla}_t \widetilde{\rho}_t \right) \cdot u \cdot v~
  dx~ dy.
\end{eqnarray}
In order to estimate these four terms, we use the asymptotic expansions of 
\S \ref{section:expansion}. For example, by (\ref{exp:Q}) we have 
that (\ref{leading}) is equal to 
\[  \int_{S} \widetilde{\nabla}_tu \cdot O(t^2) 
  \cdot \widetilde{\nabla}_tv~ dx~ dy~ 
\] 
Since the operator norm of the matrix $O(t^2)$ is bounded by a constant $C$ times $t^2$,
we can apply the Cauchy-Schwarz inequality to find that the norm of (\ref{leading})
is bounded by $C\cdot t^2\cdot a_t(u)^{\und} \cdot a_t(v)^{\und}$.

Similar arguments show that there is a constant $C$ so that  
\begin{itemize}
%
%
\item (\ref{mixed}) is bounded above by $C \cdot t^2 \cdot  a_t(u)^\und \cdot \|v\|^\und$
\item (\ref{mixed2}) is bounded above by $C \cdot t^2 \|u\|^\und \cdot a_t(v)^\und$
\item (\ref{uv}) is bounded above by $C \cdot t^2 \cdot \|u\|^\und \cdot \|v\|^\und.$
\end{itemize}
The claim follows.
\end{proof}


\subsection{The spectrum of $a_t$ via separation of variables}  \label{subsec:separation}
We recall the Fourier decomposition of section \ref{sec:Fourier}. 
Since now $w=1$ we thus have, for each $u\in L^2(S, \frac{dxdy}{y^2})$ 
\[  u^k(y)~ =~ \int_{0}^1 u(x,y) \cdot  e_k(x)~ dx. \] 
where the latter makes sense for almost every $y$ and defines an element of 
$L^2((1,\infty), \frac{dy}{y^2}).$

As above, let $\Dcal(\overline{S})$ denote the set of functions $v:S \rightarrow \Cbb$ 
such that $v$ is the restriction of a compactly supported,
smooth function defined in a neighborhood of $S$.
If $u\in \Dcal$, then each $u^k$ is smooth, and a straightforward computation 
shows that \footnote{Here $\otimes$
is the operation defined by $(v \otimes w)(x,y) = v(y) \cdot w(x)$.}
\begin{eqnarray}
  a_t(u) &=& \sum_{k \in \Nbb}  a_t \left( u^k \otimes e_k \right)   \nonumber \\
 &=&   \sum_{k \in \Nbb} \displaystyle \int_1^\infty \left( t^2 \cdot \partial_y u^k(y)^2~ 
     +~ \left(k\pi\right)^2 \cdot u^k(y)^2 \right)~ dy \label{Decomposition}
\end{eqnarray}
We define $\Dcal([1,\infty[)$ to be the set of compactly supported, smooth 
functions defined on $[1,\infty[$. 
For $v \in \Dcal([1,\infty))$, each integer $k$, and each $t>0$, we define 
\begin{equation}  \label{defn:ak}
  a^k_t(v)~ =~  \int_1^{\infty} 
    \left( t^2 \cdot v'(y)^2~ +~ \left(k\pi\right)^2 \cdot v(y)^2 \right)~ dy. 
\end{equation} 

For $v, w$ in $L^2([1, \infty[~ , y^{-2} \,dy)$, the inner product is defined by 
\[ \langle u, v \rangle_y~ =~ \int_{1}^{\infty} u(y) \cdot v(y)~ \frac{dy}{y^2}.\]
Let $L^2_{\beta}$ denote the subspace consisting of those functions whose support lies in $[1, \beta]$.

For each $k \in \Nbb$, the quadratic form $a_t^k$ extends to a closed, 
densely defined form on the completion of $\Dcal([1, \infty[)$ with respect to 
$v \mapsto a_t^k(v)^{\frac{1}{2}} +\langle v,v\rangle_y^{\frac{1}{2}}$.
For $k=0$, we will restrict the domain of $a_t^k$ to be the completion
of those smooth functions whose support lies in $[1, \beta]$. 

If $u$ is an eigenfunction of $a_t$ with eigenvalue $\lambda$, 
then for each $v$ in the domain of $a_t^k$, we have
\[  a_t^k(u^k,v)~ =~ a_t(u,v \otimes e_k)~ =~ \lambda \cdot \langle u,v \otimes e_k \rangle~ =~
      \lambda \cdot \langle u^k,v \rangle_y,
\]
and hence $u^k$ is an eigenfunction of $a_t^k$ with eigenvalue $\lambda$
with respect to $\langle \cdot, \cdot \rangle_y$. 
Thus, each eigenfunction $u$ of $a_t$ may be written uniquely
as 
\[   u~ =~ \sum_{k \in \N} u^k \otimes  e_k \]
where the $k^{{\rm th}}$ Fourier coefficient $u^k$ is an eigenfunction of $a_t^k$.
Moreover, the spectrum of $a_t$ with respect to 
$\langle \cdot, \cdot \rangle$ is the union of the spectra of $a_t^k$ 
with respect to  $\langle \cdot, \cdot \rangle_y$. In what follows we will often 
suppress the subscript $y$ from the notation.

The next two lemmas identify the eigenfunctions of $a_t^k$ for each $k$.

The zero modes are given by the following lemma.

\begin{lem} \label{ZeroLemma}
The spectrum of $a_t^0$ with respect to $\langle \cdot, \cdot \rangle_y$ is the set 
\begin{equation} \label{ZeroEigenvalue}
  \left\{ \left. t^2 \cdot \left( \frac{1}{4} + r^2 \right)~
  \right|~ r>0 \mbox{ and } 2 r = \tan(r \cdot \ln(\beta)) \right\}.
\end{equation}
The eigenspace associated to $t^2 \cdot ( \frac{1}{4} + r^2)$ is spanned by the eigenvector
\begin{equation} \label{ZeroEigenvector}
 \psi(y)~ =~ 
 y^{\frac{1}{2}} \cdot \cos \left(r \ln(y)\right)~ -~
 \frac{y^{\und}}{2r} \cdot \sin\left(r \ln(y)\right).
\end{equation}
\end{lem}

\begin{proof}
Suppose that $v$ is an eigenfunction, that is 
$a_t^0(v,w)= \lambda \cdot \langle v, w \rangle$ for all $w$.
This implies first that 
\[  - t^2 \cdot v''(y)~ =~  - \frac{\lambda}{ y^2} \cdot v(y). 
\]
holds in the distributional sense. Ellipticity then 
yields that $v$ is smooth. Moreover, by integrating by parts against a smooth function that 
is identically equal to $1$ near $y=1$, we also find that $v'(0)=0.$ 
Let $s$ be such that $s\cdot (1-s)= \lambda / t^2$. Then 
two linearly independent solutions are given by $y^s$ and $y^{1-s}$ if $s \neq \frac{1}{2}$
and by $y^{\und}$ and $y^{\und} \cdot \ln(y)$ if $s=\frac{1}{2}$. 
The condition that $\lambda/t^2$ is real and nonnegative implies
either that $s= \frac{1}{2}+ir$ with $r>0$, that $s \in [0,1/2)$ or that $s=\frac{1}{2}$.
If ${\rm Re(s)= 1/2}$, then the boundary conditions $v'(1)=0$ and $v(\beta)=0$ 
imply that the solutions take the form given in
(\ref{ZeroEigenvector}) with $2 r = \tan(r \cdot \ln(\beta))$. 
If $s \in [0, 1/2]$, then there are no solutions that 
satisfy the  boundary conditions.
\end{proof}

The nonzero modes are given by the following lemma.

\begin{lem} \label{lem:nonzero_fourier}
For each $t$ and $k$ and eigenvalue $\lambda$ of $a_t^k$ with respect to 
$\langle \cdot, \cdot \rangle_y$, the associated eigenspace consists of
functions of the form $y \mapsto f(\pi k \cdot y/t)$ such that 
\begin{equation*}
\begin{split}
i)&~  f''(z)~ 
=~ \left(1 - \frac{\lambda}{(t \cdot z)^2} \right) \cdot f(z), \\  
ii) &~  f\in L^2\left ( [1, \infty), \frac{dy}{y^2}\right), \\
iii) &~ f'(\pi \cdot k/t)=0.
\end{split}
\end{equation*}

Moreover, when $t$ varies, the spectrum is organized into eigenvalues branches $\lambda_i(t).$ 
For each $i$, the function $t \mapsto \lambda_i(t)$ is increasing, and 
\[
\lim_{t\rightarrow 0} \lambda_i(t)~  =~  (\pi \cdot k)^2.
\] 
\end{lem}
\begin{proof}
Integrate by parts as in the proof of Lemma \ref{ZeroLemma},
make the change of variables $z \mapsto  \pi k \cdot y/2t$, 
and use the boundary conditions.
The existence of a complete set of real-analytic branches of eigenvalues and
eigenvectors follows from Theorem \ref{thm:eigenbasis_exist}.
 
For each real-analytic eigenbranch $t\mapsto \lambda_i(t)$, the first order variation
is obtained by the classical formula 
\[
\dot{\lambda}_i(t)\,=\,\dot{a}^k_t(v_t)\,=\,2t\int_1^\infty v'_t(y)^2 \, dy,
\]
where $v_t$ is the normalized eigenfunction associated with
$\lambda_i(t).$ Since the integrand is nonnegative, the eigenvalue branches are
increasing.

The fact that each eigenbranch $t\mapsto \lambda_i(t)$ limits to
$ (\pi \cdot k)^2$ follows from the methods of \cite{HJ10}.
Alternatively, one can transform the eigenvalue problem into a problem 
involving a Schr\"odinger operator with a potential whose minimum value
equals $ (\pi \cdot k)^2$.
\end{proof}

\begin{remk}
The eigenfunctions of $a_t^k$ for positive $k$ may be regarded as Bessel
functions since the corresponding differential equation can be
transformed into a Bessel equation on the half-line $[i,i \cdot \infty).$    
\end{remk}

As a consequence of the identification of the eigenfunctions
we have the following Poincar\'e type inequality.

\begin{lem}\label{lem:Poincare}
For each $t\leq 2\pi,$ and each $u\in H^1_\beta(S)$ we have
\begin{equation}\label{eq:Poincare}
     a_t(u)~ \geq~ \frac{t^2}{4} \cdot n(u).
\end{equation}
\end{lem}

\begin{proof} 
We have 
\[
a_t(u)~ =~ \sum_k a_t^k(u^k) \ \ \ \mbox{ {\rm and}} \ \ \ 
  n(u)~ = \sum_K \int_1^\infty \left| u^k\right|^2 \frac{dy}{y^2}.
\]
Lemmas \ref{lem:nonzero_fourier} and \ref{ZeroLemma}  imply 
\begin{equation*}
a^0_t(u^0)~  \geq~ \frac{t^2}{4}\cdot \int_1^\beta \left| u^0\right|^2 \frac{dy}{y^2},
\end{equation*}
and 
\begin{equation}  \label{est:a^k_spectral_bound}
   a^k_t(u^k)~  \geq~  (\pi \cdot k)^2\cdot \int_1^\infty \left| u^k \right|^2 \frac{dy}{y^2}  
\end{equation}
for $k>0$.
\end{proof}

In the sequel we will use different kind of projections associated
either with the Fourier decomposition $u = \sum u^k\otimes e_k$ or 
with the spectral decomposition of $a_t.$ 

More precisely, for each $\ell \in \Nbb$ define the orthogonal projection 
$\Pi_{\ell}: L^2_{\beta}(S) \rightarrow L^2_{\beta}(S)$ by
\begin{equation}  \label{defn:Pi_ell}
 \Pi_{\ell}(v)~ =~ v^{\ell}(y) \cdot e_{\ell}(x),
\end{equation}
and let $V_{\ell}$ denote the image of $\Pi_{\ell}$.
For each $k \in \Zbb$, we define
\begin{equation} \label{defn:low_mode_project}
 \Pi_{\ell<k}(v)~ =~  \sum_{\ell < k}  v^{\ell}(y) \cdot e_{\ell}(x),
\end{equation}
the projection onto $\bigoplus_{0 \leq \ell<k} V_\ell$.

We also define $P^{\lambda}_{a_t}$ to be the orthogonal projection onto the eigenspace of $a_t$ 
associated to the eigenvalue $\lambda$ of $a_t$.  For a fixed interval $I$,
define the $a_t$-spectral projection in the energy interval $I$ to be
\[   P_{a_t}^I(v)~ :=~ \sum_{\lambda \in {\rm spec}(a_t) \cap I}~  P^{\lambda}_{a_t}(v). \]
For each eigenvalue $\lambda$ of $a_t$, the associated $a_t$-eigenspace $W_{\lambda}$ 
is the orthogonal direct sum $\oplus_{\ell} ( W^{\ell}_{\lambda} \otimes \vect(e_\ell))$
where $W^{\ell}_{\lambda}$ is the $\lambda$-eigenspace of $a_t^{\ell}$
and $\vect(e_\ell)$ is the span of $e_{\ell}$. It follows that   
\begin{equation} \label{eq:projection_relation}  
\Pi_{\ell} \left( P_{a_t}^I(v) \right)~ :=~ 
     \sum_{\lambda \in {\rm spec}(a_t^{\ell}) \cap I}~ P^{\lambda}_{a_t}(\Pi_{\ell}(v)). 
\end{equation}

More generally, for a quadratic form $b$ the notation $P^I_{b}$ will always denote the spectral 
projection onto the interval $I.$
 
\subsection{Asymptotic at first order}
In the following, we let $\dot{q}_t$ (resp. $\dot{a}_t$) 
denote the derivative of $q_t$ (resp. $a_t$) in $t$. The following
proposition will allow us to compare $a_t$ and $q_t$ in the limit
$t\rightarrow 0.$

\begin{prop}[Asymptotic at first order] \label{prop:q-a_asymp}

There exists a constant $C$ and $t_0$ such that, for all $u,v \in H^1_\beta$ and all $t\leq t_0,$ 
\begin{eqnarray} \label{eq:q-a}
\left | q_t(u,v)-a_t(u,v) \right | & \leq &C\cdot t \cdot a_t(u)^\und \cdot a_t(v)^\und \\
\left | \dot{q}_t(u)- \dot{a}_t(u)\right | & \leq & C \cdot a_t(u).  \label{eq:qdot-adot}
\end{eqnarray}
\end{prop}

In \cite{HJ10}, two real-analytic families of quadratic forms $a_t$ and
$q_t$ satisfying (\ref{eq:q-a}) and (\ref{eq:qdot-adot}) were said to
be asymptotic at first order. We will use the same terminology here.  

\begin{proof}
One argues as in the proof of Proposition \ref{prop:q-a-tb} paying a little more 
attention to the terms (\ref{mixed}), (\ref{mixed2}), and (\ref{uv}). 
For example, to estimate (\ref{mixed}), use 
the Cauchy-Schwarz inequality and (\ref{exp:rho}) to obtain 
\begin{equation}  \label{est:nabla_rho}
  | \widetilde{\nabla}_t \widetilde{\rho}_t \cdot \widetilde{\nabla}_t u |~ 
  \leq | \widetilde{\nabla}_t \widetilde{\rho}_t| \cdot | \widetilde{\nabla}_t u |~
=~ O(t^2) \cdot  \left|  \widetilde{\nabla}_t u  \right|.
\end{equation}
The Cauchy-Schwarz inequality and Lemma \ref{ZeroLemma} give
\begin{equation}  \label{est:v_nabla}
 \int_{S} |v| \cdot |\widetilde{\nabla}_t u|~
 dx \, dy~ \leq~ \frac{1}{t \cdot \sqrt{1/4+r_0^2}} \cdot
   a_t(u)^{\und} \cdot a_t(v)^{\und}. 
\end{equation}
By combining (\ref{est:nabla_rho}) and (\ref{est:v_nabla}) and 
using (\ref{exp:rho}) we find that
\[\int_{S} 
    |\widetilde{\rho}_t| \cdot |v| \cdot
   |\widetilde{\nabla}_t \widetilde{\rho}_t 
\cdot \widetilde{Q}_t \cdot  \widetilde{\nabla}_t u|~
 dx \, dy~  =~
O( t) \cdot a_t(u)^{\und} \cdot a_t(v)^{\und}.
\]
Switching the roles of $u$ and $v$, we obtain the same
bound for the expression in (\ref{mixed2}).
Similar methods apply to bound the other terms.

The estimate for $\dot{q}-\dot{a}$ is obtained in a similar way.
\end{proof}

\section{Limits of eigenvalue branches} \label{section:branches}

Since $q_t$ is asymptotic to $a_t$ at first order and $a_t$ and $\dot{a}_t$
are nonnegative quadratic forms,  each real-analytic eigenvalue
branch $E_t$ of $q_t$ converges to a finite limit $E_0$ as $t$ tends to zero
(Theorem 3.4 of \cite{HJ10}). 
For the Dirichlet eigenvalue problem on $\Tcal_t$, we showed 
in \cite{HJ10} \cite{HJ10-erratum} that each limit $E_0$ has the form $(\pi k)^2$ 
where $k$ is an integer. The methods of  \cite{HJ10} can be applied 
to show that the same fact is true in the present context.  
In this section we  highlight the necessary modifications. 
We also show that if the eigenvalue branch is associated to a 
cusp form, then $k$ must be positive. This latter fact will be used 
crucially in the proof of Theorem \ref{thm:main}.

\subsection{Non-concentration and first variation} \label{sec:noncon}

The proof of convergence depends crucially on the following `non-concentration' 
result proved for the Dirichlet problem in \cite{HJ10}. 

We will let $\Dcal_{\ell}$ denote the domain of the quadratic form $a^{\ell}_t$.
We define a quadratic form $\watl: \Dcal_{\ell} \to \Cbb$ by setting
\[
 \watl(v)\,=\,a_t^\ell(v)\,+\, \|v\|^2.
\]
for each $ v \in \Dcal_{\ell}$.

\begin{prop}\label{prop:real_noncon}(Compare Proposition 9.1 of \cite{HJ10})
Let $\ell \in \Nbb$, let $K$ be a compact subset of $](\pi \ell)^2, \infty[$, 
and let $C>0$.
There exist positive  constants $t_0$ and $\kappa$ (that only depend on $\ell,\,K$ and $C$) so that if $E \in K$, if $t<t_0$,
and if for each $w \in \Dcal_\ell$, the function $v \in \Dcal_\ell$ satisfies
\[
\left| a_t^\ell(v,w) - E \cdot \langle v, w\rangle\right|~
    \leq~ C \cdot t \cdot \|w\| \cdot \|v\|,
\]
then 
\begin{equation} \label{nonconcentrate}
  \int_1^{\infty} \left( \frac{E}{y^2}  -(\ell\pi)^2 \right) \cdot |v(y)|^2~ dy~
   \geq~ \kappa \cdot \|v\|^2.
\end{equation}
\end{prop}

\begin{proof} 
If $\ell =0$ and we let $\kappa = \inf(K)>0$, then
(\ref{nonconcentrate}) holds. 
If $\ell>0$ this follows from Proposition 9.1 of \cite{HJ10} 
with $\mu=(\pi \ell)^2$ and 
$\sigma(y)=y^{-2}$. 
See the end of \cite{HJ10-erratum} for a proof of Proposition 9.1 of \cite{HJ10}. 
\end{proof}

In the language of semi-classical analysis, 
Proposition \ref{prop:real_noncon} asserts that a quasimode $v$ of
order $t$ at energy $E$ does not concentrate at $y= \sqrt{E}/(\ell\pi)$ if $\ell \neq 0$. 
In \S 12 of \cite{HJ10}, we used non-concentration to derive indirect estimates 
for $\dot{a}_t$. The following Proposition and Corollary make these estimates 
more transparent and simpler to apply.

\begin{prop}\label{prop:noncon}
Let $\ell \in \Nbb$ and let $K\subset~ ] (\pi \ell)^2,\infty[$ be compact. 
For each $\epsilon>0$, there exists  $\kappa'>0$ 
and $t_0>0$ such that for each $v \in \Dcal_\ell$ and $t<t_0$
\begin{equation}\label{eq:an_noncon}
 \| v\|^2 \,\leq \, \frac{t}{\kappa'} \cdot \dot{a}_t^\ell(v) 
  \,+\,\frac{\epsilon}{t^2} \cdot  N_\ell(v,E)^2.
\end{equation}
where 
\[   N_{\ell}(v,E)~ 
=~ \sup_{w \in \Dcal}~ \frac{ |a_t^{\ell}(v, w)- E\langle v, w\rangle|}{
          \watl(w)^\und}
\]
\end{prop}  

\begin{proof}
From (\ref{defn:ak}) we find that $\dot{a}^{\ell}_t(v)=  2t\dis \int_1^\infty v'(y)^2$ 
and hence
\begin{equation}  \label{eq:adot_express}
t \cdot \dot{a}^{\ell}_t(v)~ =~ 
   2 \cdot \int_1^\infty \left( \frac{E}{y^2} -(\pi \ell)^2\right)  \cdot v^2~
 +~ 2 \left( a_t^{\ell}(v)- E \cdot \|v\|^2\right).
\end{equation}
If the claim is not true, then for each $\kappa'>0$, 
there exists a sequence $(t_n)_{n\geq 1}$ tending to zero and sequences
$(\tilde{v}_n)_{n\geq 1},~ \tilde{v}_n \in \Dcal,$ $(E_n)_{n\geq 1},~E_n \in K$  
such that 
\begin{equation}  \label{eq:contra}
    \|\tilde{v}_n\|^2~ 
     \geq~ \frac{t_n}{ \kappa'} \cdot \dot{a}_{\ell}(\tilde{v}_n)~ 
    +~ \frac{ \epsilon \cdot N_{\ell}(\tilde{v}_n,E_n)^2}{t_n^2}. 
\end{equation}
In particular, since $\dot{a}\geq 0$, we have 
$N(\tilde{v}_n,E_n)^2 \leq (t_n^2/\epsilon) \cdot  \|\tilde{v}_n\|^2$.
It follows that for each $w \in \Dcal$
\begin{equation}\label{eq:tvnisQM}
 |\watl(\tilde{v}_n, w)- (E_n+1) \cdot \langle \tilde{v}_n, w\rangle|~ \leq~ 
   \frac{t_n}{\sqrt{\epsilon}} \cdot \|\tilde{v}_n\| \cdot \watl(w)^\und. 
\end{equation} 
Fix $\delta>0$ such that $[-\delta,\delta]+K \subset ((\ell \pi)^2, \infty)$.
Set $I_n= [E_n-\delta,E_n+\delta]$ and $v_n = P_{a_t}^{I_n}(\tilde{v}_n)$.
Note that $v_n = P_{\watl}^{I_n+1}(\tilde{v}_n).$ where $I_n+1:= [E_n-\delta+1,E_n+\delta+1]$.

We now argue as in the proof of Lemma 2.3 in \cite{HJ10}: 
The estimate (\ref{eq:tvnisQM}) implies that 
\[
\sum_{i\geq 0} \frac{(\lambda_i-E_n+1)^2}{\lambda_i}\left |\langle
  \tilde{v}_n,\psi_i^\ell\rangle\right |^2\,\leq\,\frac{t_n^2}{\epsilon}\| \tilde{v}_n\|^2
\]  
where $(\psi_i^\ell)_{i\geq 0}$ is a complete orthonormal set of
eigenfunctions of $a_t^\ell$ (with corresponding eigenvalues
$\lambda_i$).
By retaining in the sum only the terms for which $\lambda_i\notin
I_n'=[E_n+1-\delta,E_N+1+\delta]$, we find that 
\[
\watl(\tilde{v}_n-v_n)~ \leq~ \frac{t_n^2}{\epsilon} \cdot
   \| \tilde{v}_n\|^2 \cdot \left( 1+\frac{E_n}{\delta}\right).
\]
Observe that the sequences $(v_n)_{n\geq 1},~ (t_n)_{n\geq 1}$ and
$(E_n)_{n\geq 1}$ depend on the initial choice of $\kappa'$ 
but the preceding estimate gives a constant $C$ that is independent of $\kappa'$ such that 
\[
\watl(\tilde{v}_n-v_n) \,\leq \, C\cdot t_n^2 \cdot \|\tilde{v}_n\|^2.
\]
This implies in particular $\|\tilde{v}_n -v_n\|^2 \leq C\cdot t_n^2 \cdot \|\tilde{v}_n\|^2$ 
so that, for $n$ sufficently large,  
we have $\|v_n\| \leq \|\tilde{v}_n \| \leq 2 \|v_n\|.$

In equation (\ref{eq:tvnisQM}) we replace the test function $w$ by 
$P_{a_t^\ell}^{I_n}(w)$ 
and use that the spectral projector is self-adjoint and commutes with $\watl.$ 
We obtain that for each $w\in \Dcal_\ell$, 
\begin{equation}\label{eq:vnisQM}
\begin{split}
 |\watl(v_n, w)- (E_n+1) \cdot \langle v_n, w\rangle|~ & \leq~ 
   \frac{t_n}{\sqrt{\epsilon}} \cdot \|\tilde{v}_n\| \cdot \watl
\left(P_{a_t^\ell}^{I_n}(w) \right)^\und \\
& \leq C \cdot t_n \cdot \| v_n\| \|w\|,
\end{split} 
\end{equation}
where we have used that $\| v_n\|$ is controlling $\| \tilde{v}_n\|$ and that 
\[
\watl \left (P_{a_t^\ell}^{I_n}(w)\right) \,\leq\, (\sup(K)+\delta) \cdot \|w\|^2
\]
by definition of a spectral projector.

Since $\dot{a}_t^\ell \leq \frac{2}{t}\cdot a_t^\ell$ and $\dot{a}_t^\ell$
is  a non-negative quadratic form, we also have 
\begin{equation*}
\begin{split}
\left | \dot{a}_t^{\ell}(\tilde{v}_n)^\und - \dot{a}_t^{\ell}(v_n)^\und \right | 
  & \leq  \dot{a}_t^{\ell}(\tilde{v}_n-v_n)^\und \\
& \leq \left( \frac{C}{t} a_t^\ell(\tilde{v}_n-v_n)\right)^\und \\
& \leq C \cdot \sqrt{t} \cdot \| v_n\|. 
\end{split}
\end{equation*}

Equation (\ref{eq:vnisQM}) implies that we may use Proposition \ref{prop:real_noncon} to find 
\[
\displaystyle \int_1^\infty \left(\frac{E_n}{y^2}~ -~
   (\ell \pi)^2\right) \left| v_n(y)\right |^2 \, dy~ \geq~ \kappa \cdot \| v_n\|^2.
\]
Since (\ref{eq:vnisQM}) also implies  
$\left | a_t^\ell(v_n) -E_n \|v_n\|^2\right | \leq C\cdot t_n\cdot \| v_n\|^2,$ 
using (\ref{eq:adot_express}) we find that
\begin{equation}\label{eq:lowerNC}
t_n \cdot \datl(v_n) \geq (\kappa  -C\cdot t_n) \cdot \| v_n\|^2.
\end{equation}
On the other hand, the contradiction assumption implies that 
\begin{equation}\label{eq:upperNC}
\begin{split}
\kappa' \cdot \|v_n\|^2 & \geq t_n \cdot \datl(\tilde{v}_n)\\
& \geq \left( \sqrt{t_n} \datl(v_n)^\und-Ct_n\| v_n\|^2\right)^2\\
& \geq \left( \left( \kappa -C\cdot t_n\right)^\und-Ct_n\right)^2 \cdot \|v_n\|^2 \\
& \geq \left(\kappa -C\cdot t_n\right)\|v_n\|^2.
\end{split} 
\end{equation}
The implied constant $C$ does not depend on $\kappa'$ so if we take $\kappa'<\kappa$ then choosing $t_n$ 
small enough yields the contradiction. 
\end{proof}

This proposition yields an estimate for $\dot{a}(w)$ from below in terms of the projection 
$\Pi_{\ell<k}w.$ 
 
\begin{coro}\label{coro:noncon}
Let $k \in \Zbb^+$ and let $K\subset \R^+$ be a compact subset of 
$](\pi k)^2, \infty[$. For each $\eps'>0$, there exists $\kappa>0$
 $t_0>0$ such that if $E \in K$,  $w\in \dom(a_t)$, and  $t < t_0$, then  
\begin{equation}\label{eq:crucilanoncon}
\dot{a}_t(w)~ \geq~  \frac{\kappa}{t} \cdot \left( \left\| \Pi_{\ell<k} (w) \right\|^2~
   -~ \frac{\eps'}{t^2} \cdot  N(w,E)^2 \right),
\end{equation}
where 
\begin{equation*}
N(w, E)~ =~
\sup_{v \in {\rm dom}(a_t)} 
  \frac{ |a_t(w,v)-E  \cdot \langle w,v \rangle|}{\wat(v)^\und}. 
\end{equation*}
\end{coro}

\begin{remk}
The functional $v \mapsto N(v,E)$ is equivalent to the $H^{-1}$-norm 
of $(A_t- E)(v)$ where here $A_t$ is the operator 
such that $\langle A_t u,v\rangle = a_t(u,v)$ for each $u,v \in {\rm  dom}(a_t)$. 
\end{remk}

\begin{proof}[Proof of \ref{coro:noncon}]
Since $\dot{a}_t$ is block diagonal with respect to the sum $\bigoplus_{\ell} V_\ell$ 
and $\dot{a}^{\ell}_t \geq 0$, we have 
\[ \dot{a}_t(w)~ =~ \dot{a}_t\left( \sum_{\ell=0}^{\infty} w^{\ell} \otimes e_{\ell}\right)~
=~ \sum_{\ell}\, \dot{a}_t^{\ell}(w^{\ell})~
\geq~ \sum_{\ell=0}^{k-1}\, \dot{a}_t^{\ell}(w^{\ell}). \]   
We may apply Proposition \ref{prop:real_noncon} with $\eps = \eps'/k$ 
to each term on the right hand side to find that 
\[  \dot{a}_t(w)~ \geq~  \frac{\kappa}{t} \cdot \left(
\sum_{\ell=0}^{k-1}\, \|w^{\ell}\|^2~
      -~  \frac{\eps'}{k \cdot t^2} \cdot \sum_{\ell=0}^{k-1}\, 
N_{\ell}(w^{\ell},E)^2  \right)   
\]
where $\kappa$ is the minimum of the $\kappa'$ coming from 
Proposition \ref{nonconcentrate}.
For each $\ell$, and $v\in \Dcal_\ell$, we have 
\begin{eqnarray*}    \frac{| a_t^{\ell}(w^{\ell},v)-
    E \cdot \langle w^{\ell}, v\rangle |}{\watl(v)^\und}~
&= &  \frac{| a_t (w^{\ell} \otimes e_{\ell},v \otimes e_{\ell})
  -E \cdot \langle w^{\ell} \otimes e_{\ell},  v \otimes e_{\ell}\rangle|.}{
  \wat(v\otimes e_\ell)^\und} \\
&=& \frac{| a_t(w,v \otimes e_{\ell})
  -E \cdot \langle w, v \otimes e_{\ell}\rangle|.}{
  \wat( v\otimes e_{\ell})^\und}
\end{eqnarray*}
and hence $N_{\ell}(w^{\ell},E) \leq N(w,E)$.  
We also have $\sum_{\ell<k} \|w^{\ell}\|^2 = \|\Pi_{\ell< k}(w)\|^2$,
and the claim follows. 
\end{proof}


\subsection{The spectral projection $w_t$}\label{sec:nota}

The bounds proved in \S \ref{sec:noncon} depend on a bound on $N(w,E)$.
In this subsection, we show that if $w$ is an $a_t$-spectral projection of a 
$q_t$-eigenfunction in an interval containing the eigenvalue $E$, then 
$N(w,E)$ is of order $t$.

We start with a real-analytic
eigenfunction branch $u_t$ for $q_t$ with associated real-analytic
eigenvalue branch $E_t$. We let 
\begin{equation} \label{defn:wtI}
  w_t^I~ :=~  P_{a_t}^I(u_t) 
\end{equation}
where we recall that $P_{a_t}^I$ denotes the spectral projector on the
interval $I$ that is associated to $a_t$. In the sequel, in arguments
for which the interval $I$ is fixed, this notation will
be often abbreviated to $w_t.$

Let $E_0$ denote the limit of $E_t$ as $t$ tends to zero.
The following two lemmas express the fact that the projection $w_t^{I}$ 
is an order $t$ quasimode for $a_t$ at energy $E_t$.

The following lemma is similar to Lemma 2.3 in \cite{HJ10}.

\begin{lem}  \label{lem:u_w_close}
If $I$ is a compact interval whose interior contains $E_0$, then 
there exist $t_0>0$ and $C$ such that if $t<t_0$, then 
\[   a_t\left(u_t-w_t^I \right)\, +\,
   \left\|u_t-w_t^I \right\|^2~ \leq~ C \cdot t^2 \cdot \|u_t\|^2
\]
\end{lem}

\begin{proof}
Using the fact that $u_t$ is an eigenfunction of $q_t$ 
and that $a_t$ and $q_t$ are asymptotic 
at first order, for each $w\in H^1_\beta$
\begin{equation}\label{eq:afo}
    ~ \left | a_t(u_t,w)-E_t\langle u_t,w\rangle \right |%
\leq 
C\cdot t\cdot \wat(u_t)^\und \wat(w)^\und.
\end{equation}
Observe that letting $w=u_t$ yields that 
$a_t(u_t) \leq \frac{E_t}{1-Ct^2} \cdot \|u_t\|^2.$ 
Moreover the former equation can be rewritten as 
\begin{equation*}
\left |\wat(u_t,w)-\tilde{E}_t\langle u_t,w\rangle \right |%
\leq 
C\cdot t\cdot \wat(u_t)^\und \cdot \wat(w)^\und,
\]
where $\tilde{E}_t:= E_t+1.$ We may now follow the proof of Lemma 2.3 
in \cite{HJ10} observing that $P_{a_t}^I = P_{\wat}^{I+\{1\}}.$ 
This yields a constant $C$ such that 
\begin{equation*}
\begin{split}
a_t(u_t-w_t^I)\,+\,\|u_t-w_t^I\|^2\, & \leq \, C\cdot t^2 \cdot \wat(u_t)\\
& \leq \frac{C\cdot t^2}{1-C\cdot t^2} \cdot \|u_t\|^2,\\
& \leq C'\cdot t^2 \cdot \|u_t\|^2.
\end{split}
\end{equation*}
The claim follows.
\end{proof}

\begin{remk}
Lemma \ref{lem:u_w_close} implies that most 
of the mass of $u_t$ lies in its projection, $w_t^I$,
onto the energy interval $I$. More precisely, for $t$ small we have
\begin{equation}\label{eq:normwequivnormu}
(1-C\cdot t) \cdot \|u_t\| \,\leq \, \left\|w_t^I\right\| \,\leq\,\|u_t\|.
\end{equation}
where $C$ is the constant in Lemma \ref{lem:u_w_close}.
\end{remk}

\begin{lem} \label{lem:quasimode}
If $I$ is a compact interval whose interior contains $E_0$, then 
there exist $t_0>0$ and $C$ such that if $t<t_0$, then  
\[  N\left(w_t^I,E_t\right)~ \leq~ C \cdot t \cdot \left\|w_t^I \right\|. \] 
\end{lem}

\begin{proof}
For each $w \in H^1_\beta ,$ we have 
\[
a_t(w_t,w) =a_t(u_t,w)-a_t(u_t-w_t,w)
\]
so that the Cauchy-Schwarz inequality and the preceding lemma imply 
\[
\left | a_t(w_t,w)-a_t(u_t,w)\right | \leq C \cdot t \|u_t\| a_t(w)^\und.  
\]
We also have using Cauchy-Schwarz and the preceding lemma
\[
\left | \langle u_t,w \rangle- \langle w_t^I,w\rangle \right | \leq
C\cdot t\cdot \|u_t\|\|w\|.
\]
We now start again from (\ref{eq:afo}). First, in the bounding term, we
have already seen that we could replace $\tilde{a}_t(u_t)^\und$ by
$C\|u_t\|.$ 
Thus from the triangle inequality, (\ref{eq:afo}) and the two
preceding estimates we obtain 
\begin{equation*}
  \begin{split}
\left | a_t(w_t^I,w)-E\langle w_t^I,w \rangle \right | & \leq \,C\cdot t
\|u_t\|\cdot \left (a_t(w)^\und+\|w\|\right) \\
&\leq \, C\cdot t\cdot \|u_t\| \cdot \wat(w)^\und.     
  \end{split}
\end{equation*} 
The claim follows using (\ref{eq:normwequivnormu}).   
\end{proof}

Lemma \ref{lem:quasimode} has the following corollary that expresses, 
in the language of semiclassical analysis, 
that $w_t^I$ is an order $t$ quasimode.

\begin{coro}\label{coro:wisqm}
If $I$ is a compact interval that contains $E_0$ there exists $C$ and $t_0>0$ such that, for 
$t<t_0$ and each $v\in \dom(a_t)$, we have 
\[
\left | a_t(w_t^I,v) -E_t \langle w_t^I,v\rangle \right | \, \leq \,
C\cdot t \cdot \|w_t^I\|\cdot \|v\|.
\] 
\end{coro}

\begin{proof}
Since $P_{a_t}^I$ is a spectral projector, we have 
\[
\left | a_t(w_t^I,v) -E_t \langle w_t^I,v\rangle \right |\, 
=\,\left | a_t(w_t^I,P_{a_t}^Iv) -E_t \langle w_t^I,P_{a_t}^Iv\rangle \right |,
\]
and hence Lemma \ref{lem:quasimode} implies
\[
\left | a_t(w_t^I,v) -E_t \langle w_t^I,v\rangle \right | \, 
\leq \, C\cdot \|w_t^I\|\cdot \wat(P_{a_t}^I(v))^\und.
\] 
Since $\wat(P_{a_t}^I(v))^\und \leq \left( 1+\sup(I)\right)^\und \cdot \| v\|$, the claim follows.
\end{proof}


\subsection{Limits of eigenvalue branches}

By combining Lemma \ref{lem:quasimode} with Corollary \ref{coro:noncon}
we prove the following.
 
\begin{thm}[Compare Theorem 13.1 \cite{HJ10}] \label{thm:limitsq}
Let $(E_t,u_t)$ be an eigenbranch of $q_t$ then 
there exists $k\in \Nbb$ such that 
\begin{equation}  \label{eq:elimit}
 \lim_{t\rightarrow 0} E_t~ =~  (k \cdot \pi)^2.
\end{equation}
\end{thm}  
 
\begin{proof}
Suppose to the contrary that  $E_0$ is not of the form $(k \cdot \pi)^2$
where $k$ is an integer. Let $n= \inf \{\ell \in \Nbb~ |~ (\pi \ell)^2 > E_0 \}$. 
Choose a compact interval $I\subset ](n-1)^2\pi^2,n^2\pi^2[$  whose interior contains $E_0$. 

Let $u_t$ be a real-analytic eigenfunction branch of $q_t$ associated to $E_t$.
As before, let $w^I_t= P_{a_t}^I(u_t)$ be the projection of $u_t$ onto the modes of $a_t$ that have
energy lying in $I$. Since $I$ is fixed in the rest of this argument,
we abbreviate the notation and simply write $w_t:=w_t^I.$ 

If $\ell \geq n$, according to Lemma \ref{lem:nonzero_fourier}, the
eigenvalues branches of $a_t^n$ are increasing and
limit to $n^2\pi^2.$ It follows that for each $t,$ each eigenvalue of $a_t^{\ell}$ is at least $(\pi n)^2$.
Thus, since $\sup(I) < (\pi n)^2$, equation (\ref{eq:projection_relation})
implies that
\[    \Pi_{\ell < n}\, (w_t)~ =~   w_t. \]

Let $C$ be as in Lemma \ref{lem:quasimode}, and apply Lemma \ref{coro:noncon}
with $\eps'=  1/(2C^2)$ to obtain $\kappa$ so that 
\[  \dot{a}_t(w_t)~ \geq~ \frac{\kappa}{2t} \cdot \|w_t\|^2
\]
It follows that $\dot{a}_t(w_t)/\|w_t\|^2$ is not integrable. 
This contradicts Theorem 4.2 of \cite{HJ10}
which we state below as Theorem \ref{thm:convergent}.
\end{proof}

\begin{thm}[\cite{HJ10} Theorem 4.2] \label{thm:convergent}
Let $q_t$ be asymptotic to $a_t$ at first order,
and suppose that for each $t>0$, we have 
\begin{equation} \label{Logarithm2}
 0~ \leq~  \dot{a}_t(v)~ \leq~ t^{-1} \cdot a_t(v).
\end{equation}
Let $t\mapsto E_t$ be a real analytic eigenbranch of $q_t$ that converges to a
limit $E_0$ as $t$ tends to $0$ and let $V_t$ be the associated
eigenspace.
 
If  $t \mapsto u_t \in V_t$ is continuous on 
the complement of a countable set, then the function
\begin{equation} \label{festimate}
t~ \mapsto~ \frac{\dot{a}_t\left(P^I_{a_t}(u_t) \right)}{ 
 \left\| P^I_{a_t}(u_t) \right\|^2}      
\end{equation}
is integrable on each interval of the form $(0,t^*]$.  
\end{thm}

The next proposition will be the starting point of the contradiction
argument in the following sections. It says that a cusp form
eigenbranch cannot limit to $0$. Heuristically, the zeroth Fourier
coefficient of a cusp form vanishes identically whereas an 
eigenvalue branch that limits to $0$ must eventually
have nontrivial zeroth Fourier mode. However, because 
we have made a nontrivial change of variables, this fact requires
an argument.

\begin{prop}\label{prop:cflimnot0}
If  $E_t$ is a real-analytic cusp form eigenvalue branch of $q_t$, then 
the integer $k$ appearing in (\ref{eq:elimit}) is positive.  
\end{prop}

\begin{proof}
Suppose to the contrary that $\lim_{t \rightarrow 0} E_t=0$.
Set $I=[0,1]$ and consider $w_t=w^I_t$ defined as in (\ref{defn:wtI}).
If $\ell>0$, the restriction of $a$ to $V_{\ell}$ is bounded below by
$\pi^2 > 1$, thus we have $\Pi_0(w_t)=w_t.$ On the other hand, 
the projection of $u_t$ onto $\bigoplus_{\ell>0} V_{\ell}$ equals
$u_t- u_t^0 \otimes 1$. Let $v_t^0$ denote the projection of 
$u_t-w_t$ onto $V_0$.  

Since each $V_{\ell}$ is a direct sum of eigenspaces of $a$,
we have  
\[
  a_t(u_t-w_t)~ =~ a_t(v_t^0)~ +~ a_t(u-u_t^0 \otimes 1). 
\]
The quadratic form $a$ is nonnegative and the  
restriction of $a_t$ to $\bigoplus_{\ell>0} V_{\ell}$ is bounded below by 
$\pi^2$. Hence $ a(u_t-w_t) \geq \pi^2 \cdot \|u_t-u_t^0 \otimes 1\|^2$. 
By Lemma \ref{lem:u_w_close} we have 
\[  a_t(u_t-w_t)~ \leq~  C \cdot t^2 \cdot \|u_t\|^2.  \]
Therefore, $\| u_t - u_t^0\otimes 1\|^2 \leq C' \cdot t^2 \cdot \|u_t\|^2$,
and hence 
\begin{equation} \label{est:zerolarge}
 \|u^0_t\|_{L^2([1,\beta])}^2~ \geq~ (1-C' t^2) \cdot \|u_t\|^2 
\end{equation}
for small $t$. 

To obtain a contradiction, we will bound $\|u_t^0\|$ from above. 
Towards this end, we will compare $u_t^0$ with $\Phi_{0,t}(u_t)$ (see \S 4.3). 
In particular, Lemma \ref{lem:cuspequivalence} implies that $\Phi_{0,t}(u_t)$
is a cusp form for $\Ecal$, and hence for each $y \in [1, \beta]$  
\begin{equation}  \label{eq:zero_zero}
 \int_0^1 \eta(t \cdot x,y) \cdot u_t\left(x, b(t\cdot x,y) \right)\, dx~ =~ 0 
\end{equation}
where $b(x,y)=B(\sqrt{1-x^2},y)$ and 
$\eta(x,y) = (y/b(x,y)) \cdot \sqrt{\partial_yB(\sqrt{1-x^2},y)}$. 
(See \S 4.2 for the defintion of $B$.)
Thus, for all $y\in [1,\beta]$ we have 

\begin{eqnarray*}
u_t^0(y) \,&=&\, \int_0^1 u_t(x,y)\, dx\,\\
&=&\,\int_0^1 u_t(x,y)-\eta(t\cdot x,y)u_t(x,b(t\cdot x,y))\, dx \\
&=&\,\int_0^1 \left(1-\eta(t\cdot x,y)\right ) u_t(x,y) dx \\
& & ~~+ \int_0^1  \eta(t\cdot x,y)\left [u_t(x,y)-u_t(x,b(t\cdot
  x,y))\right] \, dx.
\end{eqnarray*}
Let $r_1$ and $r_2$ denote, respectively, the two integrals on the right-hand side
this equation. Thus, $u_t^0 = r_1 + r_2$ where  
we regard both $r_1$ and $r_2$ as functions of $y \in [1,\beta].$ 

To bound $\|r_1\|$ and $\|r_2\|$, we will repeatedly use the following well-known fact:
For each  $f \in L^2([0,1])$, we have
\[
\left |\int_0^1 f(x)\, dx\right |^2 \,\leq\,\int_0^1 |f(x)|^2 \, dx. 
\]

Using the properties of $B$, one finds that there exists $C$ such that 
\[
\sup \left\{ |\eta(t \cdot x,y)-1|,~(x,y)\in
[0,1]\times[1,\beta]  \right\} \, \leq \, C \cdot t
\] for small $t$.
It follows that there exists a (perhaps different) constant $C$ so that
\begin{equation} \label{est:etaminus1}
\| r_1\|_{L^2([1,\beta])}^2 
 \leq ~ \int_1^\beta \int_0^1 \left |
  \left( 1-\eta(t\cdot x,y)\right ) u_t(x,y)\right |^2 dxdy\, 
\leq~ C \cdot t^2 \cdot \| u_t\|^2
\end{equation}

Using the fundamental theorem of calculus and the Cauchy-Schwarz inequality
we find that
\[ \left| u_t\left(x, b(t \cdot x,y) \right)~
  -~ u_t(x, y) \right|^2~ \leq~  \left| b(t \cdot x,y)-y\right|
\cdot  \int_{y}^{b(t\cdot x,y)}  \left| {\partial_2 u_{t}} 
         \left(x, z \right) \right|^2~ dz
\]
where here $\partial_2$ denotes the derivative with respect to the second variable.  
Using the properties of $B$, we find that there exists $C'$ so that 
\[
\sup\left \{ \left| b(t \cdot x,y)-y\right| ,~(x,y)\in
  [0,1]\times[1,\beta]\right \}~ \leq~  C' \cdot t
\]
for small $t.$ It follows that the interval $[y,b(t\cdot x,y)]$ is
always a subset of the interval $[y-C'\cdot t,y+C'\cdot t],$ so that we obtain the
bound 
\[
\left| u_t\left(x, b(t \cdot x,y) \right)~
  -~ u_t(x, y) \right|^2~ \leq~  C'\cdot t \cdot 
 \int_{|y-z|\leq C'\cdot t} \left | \partial_2 u_{t} 
         \left(x, z \right) \right|^2 ~ dz.
\]
Thus we may estimate $\|r_2\|$ as follows:
\[
\begin{split}
\| r_2\|_{L^2([1,\beta])}^2 \,&\leq~ \int_{1}^\beta 
\int_0^1 \left | \eta(t\cdot x,y)
\left [u_t(x,y)-u_t(x,b(t\cdot x,y))\right] \right |^2 dxdy \\
&\leq ~ C\cdot t \int_1^\beta \int_0^1 \int_{|y-z|\leq C'\cdot t}
\left | \partial_2 u_{t}  \left(x, z \right) \right|^2 ~ dz dx dy\\
& \leq ~C\cdot t^2\cdot \int_1^\beta \int_0^1 \left | \partial_2u_{t} 
         \left(x, z \right) \right|^2 ~ dz dx
\end{split}
\]
Here the constant $C$ may vary from line to line. 
Since $a_t$ and $q_t$ are asymptotic at first order
\[  t^2 \int_S  \left| \partial_2 u_{t} 
 \left(x, z \right) \right|^2~ dx \, dz~ \leq~ a_t(u_t)~ \leq~  \left
 ( E_t\,+\,C \cdot t\right ) \cdot \|u_t\|^2.
\]
for sufficiently small $t$. This yields
\begin{equation} \label{est:gradient}
 \| r_2\|_{L^2([1,\beta])} \, \leq \, C  \left ( E_t\,+\,C \cdot t\right )^\und \cdot \|u_t\| 
\end{equation}

By combining estimates (\ref{est:etaminus1}) and
(\ref{est:gradient}), we find that
\begin{equation} \label{est:zerosmall}
  \| u_t^0\|_{L^2([1,\beta])}~ \leq~  C \cdot (E_t + t)^\frac{1}{2} \cdot \|u_t\|. 
\end{equation} 
for $t$ small and some constant $C$.  If $E_t$ were to tend to zero as $t$ tends to zero, 
then we would have a contradiction to (\ref{est:zerolarge}) for small $t$. 
\end{proof}


\subsection{Bounds on the first variation of the eigenvalue}

We can also use the nonconcentration of the spectral projection 
to give an $O(t^{-1})$ lower bound on the first variation, $\dot{E}_t$,
when the projection of the eigenfunction onto the `small' modes is significant:

\begin{prop}\label{prop:adotbound}
Let $I$ be a compact interval whose interior contains $E_0$ and
$I \subset ( (k-1)^2\pi^2, (k+1)^2\pi^2)$ and let $w_t=P_{a_t}^I(u_t)$ (see (\ref{defn:wtI})).
For each $\delta>0$, there exists $\kappa'>0$ and $t_0>0$ so that 
if $t< t_0$ and 
\[ \left\| \Pi_{\ell <k} (w_t) \right\|~ \geq~ \delta  \cdot
\|u_t\|, \]
then 
\begin{equation}\label{adotbound}
  \dot{E}_t~ \geq~ \frac{\kappa'}{t}.
\end{equation}
\end{prop}

\begin{proof}
Using the Cauchy-Schwarz inequality and the nonnegativity of
$\dot{a}_t$ we have 
\begin{equation}  \label{est:adotu_adotw}
  \dot{a}_t(u_t)~ \geq~  \dot{a}_t(w_t)~ -~  \dot{a}_t(w_t)^{\und}
    \cdot  \dot{a}_t(u_t-w_t)^{\und}.
\end{equation}
It follows from (\ref{eq:defa_t}) that for all $v \in {\rm Dom}(a_t)$
\begin{equation} 
 \dot{a}_t(v)~ \leq~ 2t^{-1} \cdot a_t(v).  
\end{equation}
and hence $\dot{a}_t(w_t)^{\und} \leq \sqrt{2} \cdot  t^{-\und} \cdot a(w_t)^{\und} 
\leq t^{-\und} \cdot (2 \sup(I))^{\und} \cdot \|w_t\|$. Moreover, by combining this
with Lemma \ref{lem:u_w_close} we find $\dot{a}_t(u_t-w_t)\,\leq\, Ct\|u_t\|^2.$ 

Thus, from (\ref{est:adotu_adotw}) we obtain 
\[  \dot{a}_t(u_t)~ \geq~  \dot{a}_t(w_t)~ -~ 
  C\cdot \|u_t\|^2.
\]
Hence by applying Lemma \ref{lem:u_w_close} we have 
\begin{eqnarray}
   \dot{E}_t \cdot \|u_t\|^2~  &=&  \dot{q}(u_t)  \nonumber \\
   &\geq&   \dot{a}_t(u_t)~ - C \cdot a(u_t)   \label{est:E_dot_lower} \\
   &\geq&   \dot{a}_t(w_t)~
     -~ C\cdot \|u_t\|^2~ -~ C \cdot q(u_t)  \nonumber   \\
   &\geq&    \dot{a}_t(w_t)~ -~ C\cdot \|u_t\|^2  \nonumber 
\end{eqnarray}
for $t$ sufficiently small.

As in the proof of Theorem \ref{thm:limitsq}, we have
$\Pi_{\ell<k+1}(w_t)\,=\,w_t\,=\,\Pi_{\ell<k}(w_t)\,+\,\Pi_k(w_t)$.
Since $\dot{a}_t$ is non-negative and `block-diagonal' we have 
\[
\dot{a}_t(w_t)~ \geq~ \dot{a}_t(\Pi_{\ell<k}(w_t)).
\]
Let $C$ be as in Lemma \ref{lem:quasimode}, and apply Lemma \ref{coro:noncon}
with $\eps'= \delta^2/(2C^2)$ to obtain $\kappa$ so that 
\begin{equation*}
\begin{split}
 \dot{a}_t(w_t)& ~\geq~ \frac{\kappa}{t}\left( \| \Pi_{\ell<k}(w_t)\|^2- \frac{\delta^2}{2}\|w_t\|^2\right)\\ 
& ~\geq ~ \frac{\kappa \cdot \delta^2}{2t} \cdot \|w_t\|^2
\end{split}
\end{equation*}
for $t$ sufficiently small. Estimate (\ref{eq:normwequivnormu}) implies  
that $\|w_t\|^2 \geq \und \|u_t\|^2$ for $t$ small, and therefore by
combining the above inequalities, we prove the claim.
\end{proof}

In contrast to Proposition \ref{prop:adotbound}, we have the following.  

\begin{lem}\label{EdotUpperEstimate}
There exists  $t_0>0$ and $C$ such that if $t< t_0$, then 
\[
\frac{\dot{E}_t}{E_t}~ \leq~ \frac{2}{t}~  +~ 3C.
\]
\end{lem} 

\begin{proof}
It follows from (\ref{eq:defa_t}) that for all $v \in {\rm Dom}(a_t)$
\begin{equation} 
 \dot{a}_t(v)~ \leq~ 2t^{-1} \cdot a_t(v).  
\end{equation}
From (\ref{eq:q-a}), there exists $C$ so that for sufficiently small $t$
\[  a_t(v)~ \leq~  (1+C\cdot t) \cdot q_t(v), \]
and 
\[  \dot{q}_t(v)~ \leq~  \dot{a}_t(v)~ + C \cdot a_t(v). \]
Thus, if $u_t$ is the real-analytic eigenfunction branch of $q_t$ associated 
to $E_t$, then 
\begin{eqnarray*}
\dot{E}_t \cdot \|u_t\|^2 & = &\dot{q}_t(u_t) \\
   & \leq  &  \dot{a}_t(u_t)~ +~  C  \cdot a_t(u_t)  \\
   & \leq &   (2 \cdot t^{-1} + C) \cdot a_t(u_t)  \\
   & \leq  & (2 \cdot t^{-1} + C) \cdot (1+ C \cdot t) \cdot  q_t(u_t)   \\
   & = &  \left(2 \cdot t^{-1} +C\right) \cdot (1+ C \cdot t)\cdot  E_t \cdot \|u_t\|^2. 
\end{eqnarray*}
By choosing $t_0$ sufficiently small, we obtain the claim.
\end{proof}


\section{Proof of the main theorem} \label{sec:proof}

Proposition \ref{alt:alt} reduces the proof of Theorem \ref{thm:main}
to the following. 

\begin{thm} \label{thm:nobranch}
The family $t \mapsto q_t$ does not have a real-analytic cusp 
form eigenbranch.
\end{thm}

The proof of Theorem \ref{thm:nobranch} will be by contradiction. 
We will assume that there exists a real-analytic cusp form eigenvalue branch,
$E_t$. By the results of \S \ref{section:branches},
we have
\begin{equation} \label{eq:limitk}
  \lim_{t \rightarrow 0}~  E_t~ =~ (\pi \cdot k)^2. 
\end{equation}
where $k$ is positive. 
We aim to contradict the positivity of $k$.


\subsection{Choosing   $\beta$} \label{subsec:beta}
We recall that our construction of the quadratic form $q_t$ depends on
a parameter $\beta,$ by forcing the zeroth
Fourier coefficient of the elements of $\dom(q_t)$ to vanish for
$y>\beta.$   

The proof of Theorem \ref{thm:nobranch}
will rely on the estimates of solutions to ordinary differential equations 
made in Appendix \ref{appendix:airyB}. To make the estimates less tedious,
we will choose $\beta$ to be sufficiently close to 1, where `sufficiently close' 
will be determined by the integer $k$ that appears in (\ref{eq:limitk}).

However, the construction of the quadratic form $q_t$ depends on $\beta$.\footnote{We 
have suppressed this dependence from notation until now.}
Therefore, the integer $k$ that appears in (\ref{eq:limitk}) 
depends a priori on $\beta$.
In order to avoid circularity of reasoning, we will prove the following. 

\begin{prop} \label{prop:choosebeta}
Let $E_t$ be a real-analytic eigenvalue branch 
associated to a real-analytic cusp form eigenfunction branch of $t \mapsto q^{\beta}_t$. 
For each $\beta'>1$, the family $t \mapsto q_t^{\beta'}$ has 
a real-analytic cusp form eigenfunction branch
with associated eigenvalue branch $E_t$.\footnote{The
respective eigenfunction branches will not be the same if $\beta \neq \beta'$.}
\end{prop}

\begin{proof}
For each fixed $t$, since $E_t$ corresponds to a cusp form, it belongs
to the spectrum of $q_{t}^{\beta'}$ for all $\beta'$ (see Lemma \ref{lem:betacf}).
By Theorem \ref{thm:eigenbasis_exist},  there exists a real-analytic eigenvalue branch 
$s \mapsto \bar{E}_s^t$ of $s \mapsto q_s^{\beta'}$ such that $\bar{E}_t^t= E_t$. 
Since $s \mapsto q^{\beta'}_{s}$ has only countably many real-analytic eigenvalue branches,
there exists some branch $s \mapsto \bar{E}_s^t$ such that 
the set of $t'$ with $\bar{E}_{t'}^{t}=E_{t'}$ has an accumulation point. 
Thus, by real-analyticity,  we have $\bar{E}_{t'}^{t}=E_{t'}$  for all $t'$.  

If for each $t$, the dimension of the eigenspace $V_t$ of $q_t^{\beta'}$ 
associated to $E_{t}$ is greater than one, then one can argue as in the proof 
of Proposition \ref{prop:mult} to obtain a real-analytic cusp form 
eigenfunction branch of $q_t^{\beta'}$ associated to $E_t$. 

Otherwise, by real-analyticity, for each $t$ in the complement of a 
discrete set $A$ of $t$ we have $\dim(V_t)=1$. Let  $t \mapsto u_t^{\beta'}$ 
be a real-analytic eigenfunction branch of $q_t^{\beta'}$ associated to $E_t$. 

Let $t \mapsto u_t^{\beta}$ denote a real-analytic eigenfunction branch
of $q_t^{\beta}$ associated to $E_t$. 
For each $t$, the pull-back of $u_t$ by the diffeomorphism $\varphi_{0,t}^{\beta}$
is a cusp form of $\Ecal$ on $\Tcal_t$. In turn, for each $t$, the pull-back of 
$u_t \circ \varphi_{0,t}^{\beta}$ by  $(\varphi_{0,t}^{\beta'})^{-1}$ is a 
cusp form eigenfunction of $q_t^{\beta'}$. Hence, the eigenfunction
$u_t^{\beta'}$ is a cusp form for $t \notin A$. Thus, by Corollary
\ref{coro:cf_discrete}, the branch $t \mapsto u_t^{\beta'}$ is a
real-analytic cusp form eigenbranch of $q_t^{\beta'}$.  
\end{proof}

As a consequence of Proposition \ref{prop:choosebeta}, we may fix 
$\beta$ to satisfy\footnote{
This choice of $\beta$  is most probably not necessary 
but it will simplify the arguments 
in the appendix. In particular it implies that, on $[1,\beta]$ 
and for $\ell<k$, the Sturm-Liouville 
equations associated with $a_t^\ell$ have no turning point.}
\begin{equation}  \label{eq:beta}
   1  <  \beta <  \frac{k}{k-1}. 
\end{equation}
It follows that for each $\ell<k$ and $y \in~ [1, \beta]$ we have 
\begin{equation}\label{eq:condbeta1}
 (\pi \cdot \ell)^2-\frac{E_t}{y^2}~ <~ 0,
\end{equation} 
as soon as $t$ is small enough.

In what follows, we will drop $\beta$ from the notation for $q_t^{\beta}$.

\subsection{Tracking}

In this section we show that there exists a
real-analytic eigenvalue branch of $a_t^k,$ that we denote by
$\lambda_t^*$ such that $|\lambda_t^*-E_t|$
is at most of order $t$.  In \S \ref{sec:relative}, we will show to the contrary that 
$|\lambda_t^*-E_t|$ is at least of order $t^{\frac{2}{3}}$. This will provide the
desired contradiction.

\begin{thm}[Tracking] \label{Tracking}
If $E_t$ is a cusp form eigenvalue branch of $q_t$ 
with positive limit $(k\pi)^2$, then there 
exists $t_0>0$, $C>0$, and a real-analytic eigenvalue branch $\lambda_t^*$ of $a_t^k$ 
so that for each $t<t_0$, 
\begin{equation} \label{eqn:track-defn}
  \spec\left(a_t^k\right)~ \bigcap~ \left[ E-Ct, E+Ct \right]~ =~ \{\lambda^*_t\}. 
\end{equation}
\end{thm}

\begin{proof}
Let $u_t$ denote a real-analytic eigenfunction branch of $q_t$ 
associated to the eigenvalue branch $E_t$. Let $I \subset
((k-1)^2\pi^2,(k+1)^2\pi^2)$ be a compact interval that 
contains $k^2\pi^2$ in its interior.
Let $w_t=P_{a_t}^I(u_t)$ the associated spectral projection.
 
By Corollary \ref{coro:wisqm} and the fact that 
$\Pi_k$ is an orthogonal projection that commutes with 
$a_t$, there exist positive constants $C_{qm}$ and $t_1$ 
so that if $t<t_1$ and $v \in \Dcal$, then
\begin{equation} \label{QuasimodeTrack}
  \left| a_t\left( \Pi_k(w_t), \Pi_k(v)\right)~ 
  -~ E_t \cdot  \langle \Pi_k(w_t), \Pi_k(v) \rangle \right|~
\leq~ \frac{C_{qm}}{2} \cdot t \cdot \|w_t \| \cdot \| \Pi_k(v)\|. 
\end{equation}
Let $G$ be the set of $t>0$ such that 
$\spec(a_t^k)\, \cap\, [E_t - C_{qm} \cdot t, E_t+ C_{qm} \cdot t]$
is nonempty. To prove the claim, it suffices to show that there exists $t_0>0$ 
so that $G\, \cap\, ]0, t_0[\, =\, ]0,t_0[$ and for each $t <t_0$ the 
intersection $\spec(a_t^k)\, \cap\, [E_t - C_{qm} \cdot t, E_t+ C_{qm} \cdot t]$
is a single point.

Let $B$ denote closure of the complement of $G$, namely the set of $t>0$ so that the 
distance from $E_t$ to the spectrum of $a_t^k$ is at least $C_{qm}\cdot t$.
For each $t \in B\, \cap \, ]0,t_1[$, we apply a resolvent estimate to 
(\ref{QuasimodeTrack}) and obtain $ \| \Pi_k(w_t) \| \leq \| w_t \|/2$.
Orthogonality then implies that for 
each $t \in B \cap \, ]0,t_1[$, we have 
\begin{equation} \label{LowModeHigh}
   \left\| \Pi_{\ell<k}~ w_t \right\|~ 
     \geq~ \frac{\sqrt{3}}{2} \cdot \left\| w_t \right\|. 
\end{equation}
By estimate (\ref{eq:normwequivnormu})  we have
$\|w_t\|\sim \|u_t\|$, and so we can 
apply Proposition \ref{prop:adotbound} to find $\kappa>0$ so
that 
\begin{equation} \label{est:E-dot-bad-set}
\dot{E}_t~ \geq~ \kappa\cdot \frac{\un_B(t)}{t} 
\end{equation}
where $\un_B$ is the indicator function for $B$. 

We first observe that (\ref{est:E-dot-bad-set}) 
implies that $0$ is a limit point of $G$. 
Indeed, since $\lim_{\epsilon \to 0}E_{\epsilon} =E_0$,
the fundamental theorem of calculus implies that for 
each positive integer $n$, the limit 
\[ \lim_{\epsilon \to 0}  \int_{\epsilon}^{\frac{1}{n}} \dot{E}_t \, dt \]
exists and is finite. Thus, by (\ref{est:E-dot-bad-set}), the limit
\[ \limsup_{\epsilon \to 0}  \int_{\epsilon}^{\frac{1}{n}} \frac{\un_B(t)}{t}\, dt  \] 
is finite. Therefore, since $1/t$ is not integrable on $]0, 1/n]$, 
the set $]0, 1/n]\, \cap\, G$ is nonempty for each $n$.

Next, we note that there exists a positive $t_2 \leq t_1$  
such that if $t <t_2$ and $t \in G$, then  
$\spec(a_t^k)\, \cap\, [E_t-C_{qm} \cdot t, \,E_t+C_{qm} \cdot t]$
consists of a single point.  
This is a consequence of the {\em super-separation} phenomenon 
described in Theorem 10.4 of \cite{HJ10}.
We recall the exact statement here:
{\em Let $(s_n)_{n = 1}^{\infty}$ be a sequence of positive 
numbers that tends to zero as $n$ tends to infinity. For each positive integer $n$,
let $\lambda^+_{n}$ and $\lambda^-_n$ be distinct eigenvalues 
of $a_{s_n}^k$. If $\lambda_{n}^{\pm}$ tends to $(\pi k)^2$ as $n$ tends to $\infty$,
then }
\[
 \lim_{n \to \infty}\, s_n^{-1} \cdot \left|\lambda^+_n-\lambda_n^-\right|~
   =~ + \infty.
\] 

The set $G\, \cap\, ]0,t_2]$ is a disjoint union of intervals.  
Since $0$ is a limit point of $G$, to prove the theorem it suffices 
to show that the number of intervals is finite. 
Since the eigenbranches of $a_t^k$ are real-analytic for $t>0$, 
the intervals can be enumerated  $I_1, I_2, I_3 \ldots$
so that $\sup I_{j+1} < \inf I_{j}$ for each positive integer $j$.
Because the spectrum of each $a_t^k$ is discrete, simple, 
and nonnegative, the eigenvalue branches of $a_t^k$ can be enumerated
$\lambda^1, \lambda^2, \lambda^3, \ldots$  so that 
$\lambda^{\ell}_t < \lambda^{\ell+1}_t$ for each $t>0$. 
By super-separation, for each interval $I_j$ there exists 
a unique positive integer $\ell(j)$ so that for each $t \in I_j$ 
we have 
\[ \spec(  a_t^{k}  )\, \cap\,  \left[ E_t - C_{qm} \cdot t,\, E_t+ C_{qm} \cdot t \right]~ 
  =~ \left\{  \lambda^{\ell(j)}_t  \right\}.
\] 
To finish the proof of the theorem, it suffices to show that
the function $j \mapsto \ell(j)$ is strictly decreasing.

We claim that there exists $t_3 \leq t_2$ such that if $t \in \partial G$,
 $t < t_3$, and $\lambda_t = E_t \pm C_{qm}\cdot t$, then
\begin{equation} \label{eqn:cross-into-Z}
\dot{\lambda}_{t}~ <~ \dot{E}_{t} -C_{qm}~ <~ \dot{E}_{t} + C_{qm}~.
\end{equation}
Indeed, if $t \in \partial G$, then $t \in B$ and so (\ref{est:E-dot-bad-set})
gives that $\dot{E}_{t} \geq \kappa \cdot t$. 
Thus, if the claim were not true, then we would have
a sequence $t_n$ tending to zero such that  
$t_n \cdot \dot{\lambda}_{t_n}$ would be 
bounded below by a positive constant $\kappa$.
But this would contradict Lemma \ref{LambdaDot} below. 

Let $a<t_3$ be the left endpoint of an interval $I_j$.
Then $a \in \partial G$ and hence estimate (\ref{eqn:cross-into-Z})
implies that there exists $\epsilon>0$ so that if 
$a-\epsilon< t < a$, then $\lambda^{\ell(j)}_{t} > E_t + C_{qm}\cdot t$.
Let $b$ be the infimum of $s< a$ such that for all $t \in\, ]b,a[$ 
we have $\lambda^{\ell(j)} > E_t + C_{qm}\cdot t$. We can not have $b>0$. 
For then $b \in \partial G$ and estimate (\ref{eqn:cross-into-Z}) 
would give a contradiction. Therefore, $\lambda^{\ell(j)}_t> E_t + C_{qm}\cdot t$
for all $t < a$. See Figure \ref{fig:tracking}. 
Thus, the branch $\lambda^{\ell(j+1)}$ must lie below the branch $\lambda^{\ell(j)}$. 
That is, $j \mapsto \ell(j)$ is strictly decreasing as desired. 
\end{proof}

\begin{figure} 
\includegraphics[scale=.6]{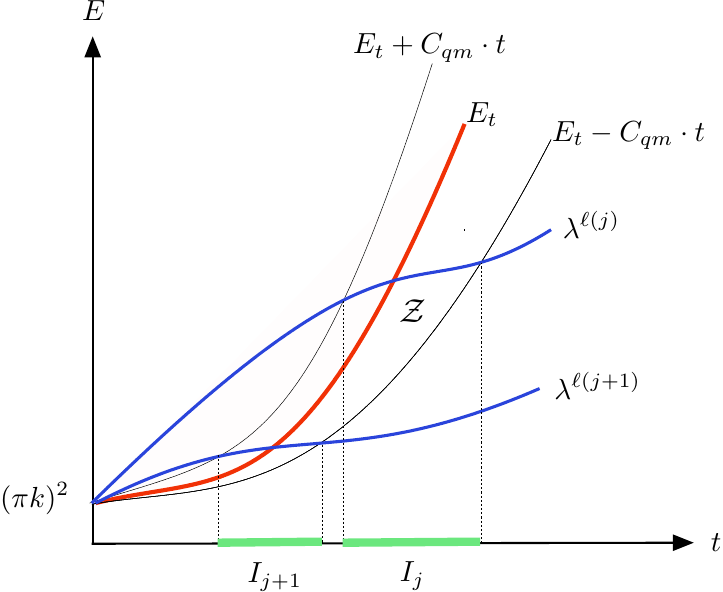}
\caption{Let $\Zcal$ denote the set of $(t,E)$ such that 
$E \in [E_t-C_{qm} \cdot t, E_t+C_{qm}]$.
 As $t$ decreases from $t_3$ to $0$, each eigenvalue branch $a_t^k$ that enters 
$\Zcal$ must enter from below, and if it exits $\Zcal$, then it must exit from above.
\label{fig:tracking}}
\end{figure}

\begin{lem} \label{LambdaDot}
Let $k>0$ and let $t_n$ be a sequence converging to zero. 
For each $n \in \Z^+$, let $t\mapsto \lambda_n(t)$ be a real-analytic 
eigenbranch of the family $a_t^k$. 
If $\lim_{n \rightarrow \infty} \lambda_n(t_n)=  ( \pi k)^2$, 
then 
\[   \lim_{n \rightarrow \infty}~  t_n \cdot \dot{\lambda}_{n}(t_n)~  =~ 0.
\]
\end{lem}

\begin{remk}
One may replace the assumption in Lemma \ref{LambdaDot} 
that $t \mapsto \lambda_n(t)$ is real-analytic
 with the assumption that $t \mapsto \lambda_n(t)$ continuous. 
Indeed, for each $t>0$, the eigenvalue problem for $a_t^k$ 
corresponds to an eigenvalue problem for an ordinary differential
equation, and hence the eigenvalues are simple. 
It follows that each continuous eigenvalue branch $t \mapsto \lambda_n(t)$  of 
the real-analytic family $a_t^k$ is necessarily real-analytic. 
\end{remk}
\begin{remk} 
Lemma \ref{LambdaDot} is not a direct consequence of 
Lemma \ref{lem:lambda_asymp} because the eigenvalue
branch $\lambda_n$ may vary with $n$. By keeping track of the constants 
in the proof of Lemma \ref{lem:lambda_asymp}, one can produce a 
version that directly implies Lemma \ref{LambdaDot}. We prefer
to give a direct proof here. 
\end{remk}

\begin{proof}
For each $n$ we denote by $y\mapsto \psi_n(y)$ a unit norm eigenfunction of $a_{t_n}^k$ 
with eigenvalue  $\lambda_n(t_n)$. 
By the standard variational formula and (\ref{defn:ak})
\[
  \dot{\lambda}_n(t_n)~ =~ \dot{a}_{t_n}\left(\psi_n(y)\right)~
 =~ 2t_n \cdot  \int_{1}^{\infty}  \left| \psi_n'(y)\right|^2~ dy. 
\]
Since $\psi_n$ is an eigenfunction of $a_{t_n}^k$ with eigenvalue
$\lambda_n(t_n)$, using (\ref{defn:ak}),
we have
\begin{equation} \label{Variation}
 t_n^2 \cdot  \int_{1}^{\infty}  \left| \psi_n'(y)\right|^2~ dy~ =
   \int_{1}^{\infty} \left( \frac{\lambda_n(t_n)}{y^2} -(\pi k)^2\right)  \left| 
 \psi_n(y)\right|^2~ dy.
\end{equation}
It suffices to show that the right hand side of
(\ref{Variation}) tends to zero as $n$ tends to infinity.

Let $\epsilon>0$.  Since $\lambda_n(t_n)$ tends to $(\pi k)^2$, 
there exists $\delta>0$ so that if $|y-1|< \delta$,
then  $\left( \lambda_n(t_n) \cdot y^{-2} -(\pi k)^2\right) < \epsilon/2$ and thus
\begin{equation} \label{Delta}
\int_{1}^{1+\delta} \left( \frac{\lambda}{y^2} -(\pi k)^2\right)  \left| 
 \psi_n(y)\right|^2~ dy~ \leq~ \frac{\epsilon}{2} \cdot 
 \int_{1}^{1+\delta} \left|  \psi_n(y)\right|^2~ dy~.
\end{equation}

To estimate the remaining integral over $[1+\delta, \infty)$, we will apply a
standard convexity estimate 
from the theory of ordinary differential 
equations.\footnote{See, for example, Lemma 6.3 in \cite{HJ10}.}  
If $n$ is sufficiently large, $\lambda_n(t_n)/ (\pi k)^2 < (1+\delta/4)^{2}$, 
and hence there exists $\eta>0$ so that if $y> z > 1+\delta/2$, then 
\[  |\psi_n(y)|^2~ \leq~ |\psi_n(z)|^2 \cdot  {\rm exp}\left( -\frac{\eta}{t}
  \cdot\left(y -z \right) \right).
\]
It follows that there exists $t_0$ so that if $t < t_0$, then 
\begin{equation} \label{DeltaPlus} 
  \int_{1+\delta}^{\infty}   \left|  \psi_n(y)\right|^2~ dy~ \leq~
\frac{\epsilon}{4(\pi k)^2} 
  \cdot  \int_{1}^{\infty}   \left|  \psi_n(y)\right|^2~ dy.
\end{equation}

Since $\left|\lambda_n(t_n) /y^2 - (\pi k)^2\right| \leq (\pi k)^2$ 
for sufficiently large $n$, we may combine (\ref{DeltaPlus}) with  
(\ref{Delta}) to show that (\ref{Variation}) is less than 
the given $\epsilon$ for sufficiently large $n$.
\end{proof}

\begin{nota}
For the convenience of the reader, we recall the notation that 
is being used in the proof of the main theorem. 
We are considering an eigenbranch $(u_t,E_t)$ of $q_t$ such that 
$\lim_{t\rightarrow 0} E_t\,=\, k^2\pi^2\,>0\,.$ We have chosen a
compact interval $I\subset\, ](k-1)^2\pi^2,(k+1)^2\pi^2[$ that contains
$k^2\pi^2$ in its interior. We have set $w_t=P_t^Iu_t$ where $P_t^I$
is the spectral projector on $I$ associated with $a_t.$ Finally,
$w_t^k$ is the orthogonal projection of $w_t$ onto $V_k$ so that there
exists $v_t^k\in L^2((1,+\infty), y^{-2}dy)$ such that
$w_t^k=v_t^k\otimes e_k.$  
\end{nota}


\subsection{Crossings}

In this subsection, we show that $\|w_s^k\|$ is smaller than $\|u_s\|$ for $s$ near 
a {\em crossing time}, a value of the parameter $t$ such that $E_t$ belongs 
to the spectrum of $a_t^0$.  
Then we show that there exists a sequence of crossing times $t_n$ 
and intervals of width $O(t^{\frac{8}{3}})$ about the crossing times 
on which $\|w_t^k\|$ is smaller than $\rho\|u_t\|$ for some fixed $\rho$.

The proof of the first result depends on the analysis contained in
Appendix \ref{appendix:airyB} that provides estimates on the
 {\em off-diagonal part} of the quadratic form $b_t$ which, we
recall, has been defined in (\ref{eq:defb_t}) in such a way that it is the
leading part of $q_t-a_t$ (see Proposition \ref{prop:q-a-tb}).

\begin{prop}\label{prop:Xings-kpart}
Given $\rho<1$, there exists $\eta>0$ and $t_0>0$ such that if $t<t_0$ and  
\begin{equation} \label{Dist53}
{\rm dist}\left(E_t, \spec\left(a_t^0\right) \right)~ \leq~ \eta \cdot t^{\frac{5}{3}},
\end{equation}
then 
\[ \left\| w_t^k \right\|~ \leq~ \rho \cdot \|u_t\|. \]
\end{prop}

\begin{proof}
Let $\psi_t^0$ be an eigenfunction of $a_t^0$ with eigenvalue $\lambda_t^0$
satisfying $|E_t- \lambda_t^0| < \eta \cdot t^{\frac{5}{3}}$. We have  
\begin{equation*} 
(E_t-\lambda^0_t) \cdot \langle u_t, \psi_t^0\rangle =
   (a_t-q_t)(u_t, \psi_t^0) 
= t\cdot b_t(u_t,\psi_t^0) \,+\,O(t^2) \cdot \| u_t\| \cdot 
   \|\psi_t^0\|.  
\end{equation*}
and hence
\begin{equation}
  \label{eq:bRHS}  \left|E_t-\lambda^0_t \right| \cdot \left|\langle u_t, \psi_t^0\rangle \right|~ 
\geq~ t\cdot \left| b_t(u_t,\psi_t^0)\right|  \,-\,O(t^2) \cdot \| u_t\| \cdot 
   \|\psi_t^0\|. 
\end{equation}
In Appendix \ref{appendix:airyB} Prop. \ref{prop:blower}, we prove that there exists $\kappa >0$
so that 
\begin{equation*}
\left | b_t(u_t,\psi^0_t)\right | \,\geq\, \kappa \cdot t^{\dt}\cdot 
\left( \| w_t^k\|-t^\delta\cdot \|u_t\|\right) \cdot \|\psi_t^0\|.   
\end{equation*} 
Hence by applying the Cauchy-Schwarz inequality to the left hand side
of  (\ref{eq:bRHS}), we find that   
\[ 
\left | E_t-\lambda_t^0\right | \| u_t\| \cdot \| \psi_t^0\|~
 \geq~ \left(\kappa \cdot t^{\frac{5}{3}}\cdot
 \left( \| w_t^k\|-t^\delta \cdot \|u_t\|\right) -O(t^2)\|u_t\|\right)\cdot \|\psi_t^0\|.
\]
Let $\eta= \rho \cdot \kappa/2$, and use (\ref{Dist53}) to find that
\[
  \frac{\rho}{2}\cdot \|u_t\|~ \geq~ \|w_t^k\|\, -\, O(t^\delta) \cdot \|u_t\|\,
   -\, O(t^{\unt}) \cdot \|u_t\|.
\]
The claim follows by choosing $t_0$ sufficiently small.
\end{proof}

\begin{prop}\label{prop:Xings}
For all $\eta>0,$ let $\delta= \frac{\eta}{12(\pi k)^2}.$ There exists  $s_0>0$ such 
that if $s<s_0$, $E_s \in  {\rm spec}(a_s^0)$,  
and $t \in \left[s, s+  \delta \cdot s^{\frac{8}{3}} \right]$, then 
\[ {\rm dist}\left(E_t,\,  {\rm spec}(a_t^0) \right)~ \leq~
\eta \cdot s^{\frac{5}{3}}.
\]
\end{prop}

\begin{proof}
Let $\lambda_{t}^0$ be the eigenvalue branch of $a_t^0$ 
such that $E_s= \lambda_s^0$.  By Lemma \ref{ZeroLemma}, we have
$\lambda^0_{t}=c \cdot t^2$ for some $c>0$, and hence
$\dot{\lambda}^0_{t}=  2 \cdot t^{-1} \cdot \lambda^0_{t}$. 
Using the fact that $a_t$ and $q_t$ are asymptotic and the fact that
$\dot{a}_t$ is non negative, there exists a constant $C$ such that  
$\dot{E}_t \geq -C E_t$ for all sufficiently small $t$. Thus, for even
smaller $t$ we obtain
\[  \frac{d}{dt}  \ln \left( \frac{ \lambda^0_{t}}{E_t} \right)~
  \leq~ 3 \cdot t^{-1}
\]
Since $E_{s}=\lambda_{s}$, integration over $[s,t]$ and exponentiation gives 
\begin{equation} \label{est:ratio_eigen}  \frac{\lambda^0_{t}}{E_t}~
  \leq~  \left(\frac{t}{s}\right)^{3}.
\end{equation}
If $t \leq s+ \delta \cdot s^{\hut}$, then
\[ \left(\frac{t}{s}\right)^{3}~ \leq~ \left(1+ \delta \cdot s^{\cit}\right)^3~ 
   \leq~ 1 +4 \cdot \delta \cdot s^{\cit}
\]
where the last inequality holds for $s \leq s_1=(2\delta)^{-\frac{3}{5}}$.
By combining this with (\ref{est:ratio_eigen}), one finds that 
for $t \in \left[s, s+ \delta \cdot s^{\hut} \right]$, we have 
\begin{equation} \label{est:lambda_first}
   \lambda^0_{t}~ -~ E_t~ \leq~ E_t \cdot 4 \cdot \delta \cdot s^{\cit} 
\end{equation}

Using Lemma \ref{EdotUpperEstimate}, and $\dot{\lambda}\geq 0$,
we have that
\begin{equation} \label{est:E_first}
  \frac{d}{dt}  \ln \left( \frac{ E_t}{\lambda_{t}^0} \right)~
  \leq~ 3 \cdot t^{-1}.
\end{equation}
An argument similar to the one above gives that 
\[ E_t~ -~ \lambda^0_{t}~ \leq~ \lambda_{t}^0 \cdot 4 \cdot \delta \cdot  s^{\cit} \]
for $t \in \left[s, s+ \delta \cdot s^{\hut} \right]$.

Since by assumption $\lim_{t \rightarrow 0} E_t = (\pi k)^2$, there exists $s_2$ so that
if $t<s_2+\delta \cdot s_2^{\hut}$, then  $E_t \leq 2 \cdot (\pi k)^2$. 
Therefore, by (\ref{est:lambda_first}), 
we have that $\left\{\lambda^0_t~ |~  s\leq t \leq s+ \delta \cdot s^{\hut} \right\}$
is bounded above by $3 (\pi k)^2$ for $s< s_0=\min\{s_1,s_2\}$.  
In sum, if $s < s_0$ and 
$t \in \left[s, s+ \delta \cdot s^{\hut} \right]$,
then 
\[  \left|\lambda^0_{t}~ -~ E_t \right|~ \leq~  12(\pi k)^2 \cdot
\delta \cdot s^{\cit}~\leq ~ \eta \cdot s^{\cit} ,\]
by the choice of $\delta.$
\end{proof}

We wish to estimate from below the size of the set of $t$ for which 
(\ref{Dist53}) holds true. This is accomplished by the following proposition. 

\begin{prop}\label{prop:seqXings}
Let $\delta>0$. There is a sequence $t_n$ of crossing times  such that
\begin{equation} \label{limit_ntn}
  \lim_{n \rightarrow \infty}~ n \cdot t_n~ =~ k \cdot \ln(\beta). 
\end{equation}
and if  $n\neq m$ are large enough, then 
the intervals $\left[t_n,t_n+\delta \cdot t_n^{\frac{8}{3}} \right]$ and 
$\left[t_m, t_m+\delta \cdot t_m^{\frac{8}{3}}\right]$ are disjoint.
\end{prop}

\begin{proof}  
By Lemma \ref{lem:lambda_asymp}, there exists, $\nu^*>0$ so that 
$\lambda^*_t = (\pi k)^2 + \nu^* \cdot t^{\dt} + O\left(t^{\qt} \right)$. 
It also follows from Proposition \ref{Tracking} that there exists $M$
so  that
\begin{equation} \label{est:fenetre}
   (\pi k)^2~ +~ \nu^* \cdot t^{\dt}~ -~ M \cdot t~ <~  E_t~
 <~    (\pi k)^2~ +~ \nu^* \cdot t^{\dt}~  +~ M \cdot t
\end{equation}
for sufficiently small $t$.  By Lemma \ref{ZeroLemma}, the eigenvalues 
of $a_t^0$ have the form $c_n \cdot t^2$ where $c_n= (1/4+r_n^2)$ 
and  $r_n$ is the increasing sequence of positive solutions to the equation 
$2r= \tan(r\ln(\beta))$.  Standard asymptotic analysis shows that
\begin{equation}   \label{eqn:r_n_asymp}
 r_n~ =~  \frac{n\pi+\frac{\pi}{2}}{\ln(\beta)}\,+\,o(1).
\end{equation}

Fix $0<\nu_0^-<\nu^*<\nu_0^+.$ 
For each $\nu \in [\nu_0^-,\nu_0^+]$ and each $n \in \Zbb$, there 
exists a unique $t_n^\nu \in \Rbb^+$  so that
\begin{equation}  \label{eqn:imp_+}
    c_{n} \cdot (t_n^{\nu})^2~ =~ (\pi k)^2~  
   +~ \nu  \cdot \left(t_n^{\nu} \right)^{\dt}.
\end{equation}
We drop the dependence in $\nu$ from the notation for a moment.
If we set $x_n= c_n^{-\uns}$ and $y_n=c_n^{\frac{1}{6}} \cdot t_n^\unt$
then (\ref{eqn:imp_+}) becomes
\[   y_n^6~ =~ (\pi \cdot k)^2~ +~ 
   \nu \cdot x_n^2 \cdot y_n^2.
\]

The analytic (polynomial) function $F$ defined by 
$F(x,y)=y^6~ -~ (\pi \cdot k)^2~ -~ 
   \nu \cdot x^2 \cdot y^2$ satisfies $F(0, (\pi
   k)^{\frac{2}{3}})=0$ and $\partial_y F(0, (\pi
   k)^{\frac{2}{3}})\neq 0.$
Thus, by the analytic implicit function theorem, for $x$ near $0$,
there exist a unique analytic function $Y(x)$ so that 
\[    Y(x)^6~ =~ (\pi \cdot k)^2~ +~ \nu 
   \cdot x^2 \cdot Y(x)^2. \]
By inspecting the first few coefficients in the Taylor
expansion of $Y^3$, we find that  
\[  Y(x)^{3}~ =~ \pi \cdot k~ +~ 
  \frac{\nu}{2 \cdot (\pi k)^{\unt}} \cdot x^2~ +~ O(x^3).
\]
Thus, since $\lim_{n \rightarrow \infty} x_n=0$, and $t_n =
c_n^{-\und} \cdot Y^3(c_n^{-\und})$ we find that
\begin{equation}  \label{eq:tn+_asymp}
  t_n^{\nu}~ =~ (\pi k) \cdot c_n^{-\und}~ +~ \tau \cdot c_n^{-\frac{5}{6}}~ +~ O(c_n^{-1}). 
\end{equation}
where $\tau= \nu\cdot (\pi k)^{\unt}/2$.

Choose $\epsilon>0$ so that if $\nu^{\pm} = \nu^* \pm \epsilon$,
then $\nu_0^-\leq \nu^-<\nu^*<\nu^+\leq \nu_0^+$. Define $t_n^{\pm} = t^{\pm \nu}_n$.
By applying the intermediate value theorem to $\lambda_t-E_t$, there exists
$t_n \in~ ]t_n^-, t_n^+[$ so that $E_{t_n}= \lambda_{t_n}$. See Figure \ref{Intermediate}. 
Since $c_n$ is increasing to infinity, the sequence
$t_n$ is decreasing to zero.

\begin{figure} 
\includegraphics[scale=.4]{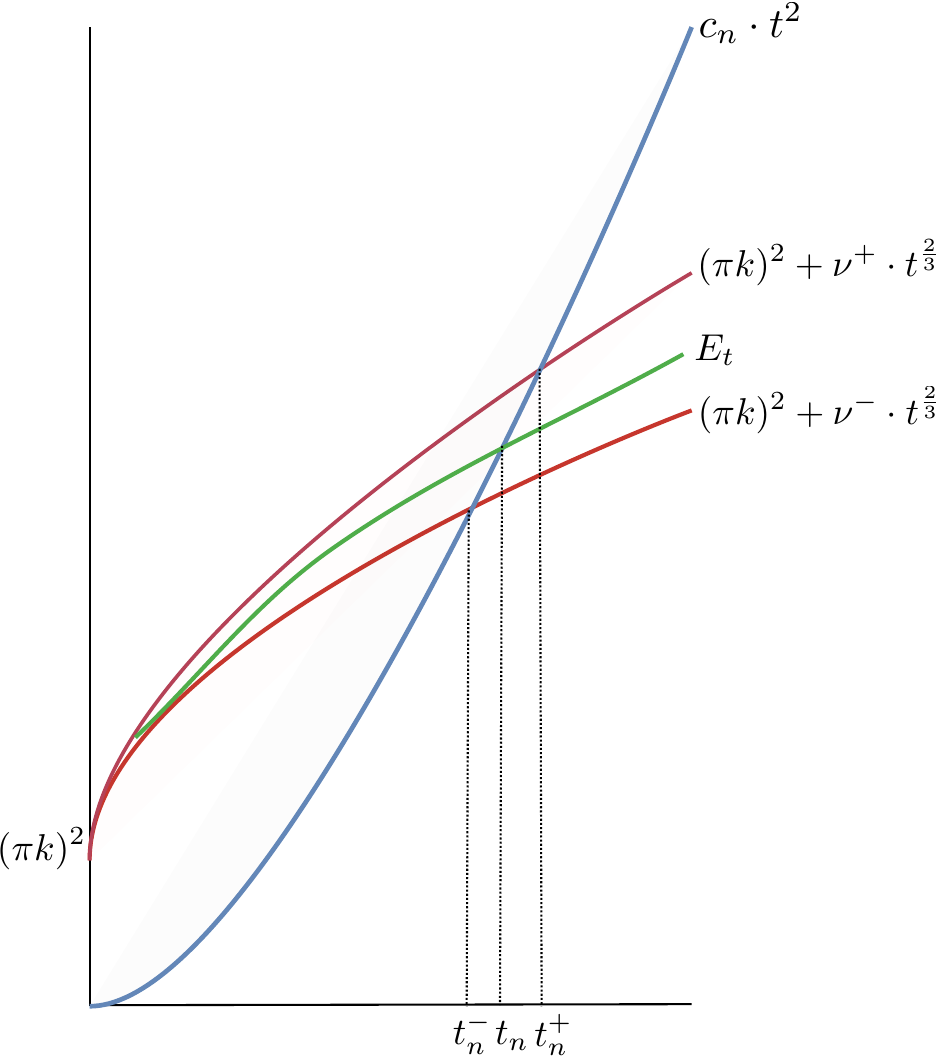}
\caption{The crossing $t_n$.  \label{Intermediate}}
\end{figure}
Moreover, since $\nu^\pm= \nu^*\pm \eps$  
\begin{equation}  \label{eq:tn_asymp}
  t_n~ =~ (\pi k) \cdot c_n^{-\und}~ +~ \tau^* \cdot c_n^{-\frac{5}{6}}~ +~ o(c_n^{-\frac{5}{6}}). 
\end{equation}
where $\tau^*= \nu^*\cdot (\pi k)^{\unt}/2$.
From (\ref{eqn:r_n_asymp}) we have 
\[  c_n^{-1}~ =~ 
    \frac{\sigma^2}{n^2} \cdot\left(1 + O\left(\frac{1}{n}\right)\right) \]
where $\sigma= \ln(\beta)/\pi$. By substituting this into
(\ref{eq:tn+_asymp}) and (\ref{eq:tn_asymp}) 
we find that
\[ t_n^{\pm}~ =~ (\pi k) \cdot n^{-1}~ 
    +~ \tau \cdot \sigma^{\cit} \cdot  n^{-\cit}~ +~ O_{\pm}(n^{-2}) 
\] 
and 
\[
t_n~ =~ (\pi k) \cdot n^{-1}~ 
    +~ \tau^* \cdot \sigma^{\cit} \cdot  n^{-\cit}~ +~o(n^{-\cit}). 
\]
The first claim follows. Moreover, since
$\nu^{\pm}= \nu^*\pm \eps$, we have 
\begin{gather*}
t_n^+-t_n~ \sim~ \eps\cdot \frac{(\pi k)^{\unt}}{2}\cdot n^{-\cit}\\
t_n-t_n^-~ \sim~ \eps\cdot \frac{(\pi k)^{\unt}}{2}\cdot n^{-\cit}\\
t_n^--t_{n+1}^+~ \sim~  \eps\cdot (\pi k)^{\unt}\cdot n^{-\cit}\\
t_n^{\frac{8}{3}}~ =~ O(n^{-\frac{8}{3}}).
\end{gather*}
It follows that, for all sufficiently large $n$, we have
$[t_n,t_n+\delta \cdot t_n^{\frac{8}{3}}]\subset [t_n^-,t_n^+]$. 
Since the intervals $\{[t_n^-,t_n^+]\}$ are disjoint, the claim is proven.
\end{proof}


\subsection{Relative variation and the contradiction} \label{sec:relative}

In this section we derive the desired contradiction.
In particular, we prove the following.

\begin{thm} \label{prop:dt}
Suppose that $E_t$ is a cusp form eigenvalue branch with a positive limit.
If $\lambda^*_t$ is the eigenvalue branch of $a_t$ that 
satisfies (\ref{eqn:track-defn}), then 
there exists $t_0>0$ and $c> 0$ so that if $t< t_0$, then
\[   E_t~ -~ \lambda_t^*~ >~  c \cdot t^{\dt}.
\]
\end{thm}

The proof will consist of two types of lower estimates. 
The first depends on the fact that near each crossing 
the `relative variation' $\dot{E}_t - \dot{\lambda}_t^*$
is at least of order $O(t^{-1})$.  The second shows that away
from the crossings the relative variation is not too negative.

Define
\[  K(t,\rho)~ =~ 
\left\{ s \in~ ]0,t]~ \left|~
   \left\| w_s^k \right\|~ \leq~ \rho \cdot \|u_s\| \right.  \right\}. \] 
If $\rho<1$, then it follows from Proposition \ref{prop:adotbound}
that there exists $\kappa>0$ so that for $s \in K(t, \rho)$ we have 
\begin{equation} \label{eq:preEdotkappa}
  \dot{E}_s~   \geq~  \kappa  \cdot s^{-1}. 
\end{equation}
Hence, since  $\dot{\lambda}_t^* = O(t^{-\frac{1}{3}})$, there exists 
$t^*>0$ so that if $t<t^*$ and $\rho<1$, then 
\begin{equation} \label{eq:Edotkappa}
  \dot{E}_s~  -~ \dot{\lambda}_s^*~  \geq~  \frac{\kappa}{2}  \cdot s^{-1}. 
\end{equation}
for each $s \in K(t, \rho)$. We will integrate this estimate near the crossings
to obtain the following.

\begin{lem} \label{lem:onK}
For each $\rho<1$, there exists $t_0>0$ and $\gamma(\rho)>0$ so that
for each $t<t_0$, we have
\[  \int_{K(t, \rho)} \left( \dot{E}_s - \dot{\lambda}_s \right)~ ds~
    \geq~  \gamma(\rho) \cdot t^{\frac{2}{3}}.
\]

\end{lem}

\begin{proof}
By (\ref{eq:Edotkappa}) the integrand is positive on $K(t, \rho)$,
it suffices to show the same estimate holds for a subset $G$ of $K(t, \rho)$.

To define this subset, we first combine Proposition \ref{prop:Xings-kpart},
Proposition \ref{prop:Xings},  and Proposition \ref{prop:seqXings} 
to find $\delta>0$, $N' \geq 2$, and a monotone sequence $\{t_n\}$ so that for each $n$ 
\begin{equation*}  
  I_{n, \delta}~ =~ \left[ t_n, t_n + \delta \cdot t_n^{\frac{8}{3}} \right]
\end{equation*}
belongs to $K(1/2, \rho)$, the intervals $I_n$ and $I_{n+1}$ are disjoint,  
and for each $n \geq N'$ 
\begin{equation} \label{t_n_bound}    
 \frac{ \tau}{2n}~ \leq~  t_n \leq~ \frac{2 \tau}{n}
\end{equation}
where $\tau= k\cdot  \ln(\beta)$. The subset $G$ will be defined 
as a union of $I_{n}$ over sufficiently large $n$. 

We have $\int_{I_{n, \delta}}s^{-1}ds = \ln ( 1 + \delta \cdot t_n^{\frac{5}{3}} )$
and hence there exists $N^* \geq N'$ so that if $n \geq N^*$ we have 
\begin{equation}  \label{eq:intoverI}
   \int_{I_{n, \delta}}~ s^{-1}~ ds~ \geq~ \frac{\delta}{2} \cdot t_n^{\frac{5}{3}}.
\end{equation}
Thus, from (\ref{t_n_bound}) we find that  if $N\geq N^*$, then 
\begin{equation} \label{eq:sumn53}
   \left(\frac{2}{\tau}\right)^{\frac{5}{3}}
   \sum_{n \geq N+2}~ t_n^{\frac{5}{3}}~ \geq~  \sum_{n \geq N+2}~ n^{-\frac{5}{3}}~
 \geq \int_{N+2}^{\infty} x^{{-\frac{5}{3}}}~ dx~ =~ (N+2)^{-\frac{2}{3}}~ 
   \geq~ \left( \frac{t_{N}}{4 \tau} \right)^{\dt}.
\end{equation}
Since the intervals $I_{n, \delta}$ are disjoint,
by combining (\ref{eq:Edotkappa}), (\ref{eq:intoverI}),
and (\ref{eq:sumn53}),  we find that 
\begin{equation}  \label{est:int_G}
  \int_{G_N} \left(\dot{E}_s~ -~ \dot{\lambda}^*_s\right)~ ds~
    \geq~ \gamma \cdot t_N^{\dt}, 
\end{equation}
where $\gamma= \kappa \cdot \delta \cdot \tau \cdot 2^{-\frac{16}{3}}$
and 
\[  G_N :~ =~ \bigcup_{n \geq N+2}~ I_{n, \delta}.  \]

Let $t_0=t_{N^*}$. If $t< t_0$, then $t \in [t_{N+1},t_N]$ for some $N\geq N^*$.
We have $t_{N+2}+  t_{N+2}^{\hut}\leq t_{N+1}\leq t$ and hence $G_N \subset K(t,\rho)$
and $t_{N}^{\dt}\geq   t^{\dt}$. Therefore, (\ref{est:int_G}) implies the claim.
\end{proof}

To bound the relative variation on the complement of $K(t, \rho)$, we will use the 
following. 

\begin{prop}  \label{prop:FdotbetweenXings}
There exists $C$ and $t_0>0$ such that, if $t<t_0$, then
\begin{equation}  \label{eq:dotFbetwXings}
  \dot{E}_t~  
   \geq~ \frac{\left\|w_t^{k} \right\|^2}{\|u_t\|^2} \cdot  \dot{\lambda}^*_t~  -~ C \cdot t^{-\frac{1}{9}}.
\end{equation}
\end{prop}

\begin{proof}
By arguing as in (\ref{est:E_dot_lower}) we have 
\begin{equation} \label{eq:Edot}
      \dot{E}_t \cdot \|u_t\|^2~ \geq~  \dot{a}_t\left(w_t^k\right)\, -\,
    O\left(\|u_t\|^2\right).  
\end{equation}
Let $w_t^*$ denote the orthogonal projection of $w_t^k$ 
onto the eigenfunction branch of $a_t$ that corresponds to 
$\lambda^*_t$ from  Theorem \ref{Tracking}.
Let $w_t^\perp := w_t^k-w_t^*$. For $\diamond= k, *, \perp,$ we define
$v_t^\diamond$  so that $w_t^\diamond= v_t^\diamond \otimes e_k$. 
Observe that by definition, $v_t^*$ is an eigenfunction
of $a_t^k$ with eigenvalue $\lambda^*_t$.

Using the Cauchy-Schwarz inequality and the nonnegativity of
$\dot{a}_t$ we have 
\begin{equation} \label{est:adot}
 \dot{a}_t(w_t^k)\,=\,\dot{a}_t^k\left(v_t^k\right)~  \geq~ 
  \dot{a}_t^k\left(v_t^* \right)~ -~ 
   2 \cdot \left|\dot{a}_t^k\left(v_t^*, v_t^\perp \right)\right|.
\end{equation}
Since $v_t^*$ is an eigenfunction, we have 
$\dot{a}_t^k\left(v_t^* \right)= \dot{\lambda}^* \cdot \|v_t^*\|^2$.
Using (\ref{defn:ak}) and the fact that $v_t^*$ is an eigenfunction that 
is orthogonal to $v_t^\perp$, we find that
\begin{eqnarray}
    \dot{a}_t^k(v_t^*,v_t^\perp)  & = &
    2t^{-1} \cdot \left( a_t^k( v_t^*, v_t^\perp)~ 
  -~ (\pi k)^2 \int_1^{\infty} v_t^*(y) \cdot v_t^\perp(y)\, dy\right) 
 \nonumber  \\
  \label{eq:at_ak_inner}   & = &
    -2 (\pi k)^2 \cdot t^{-1}  \int_1^{\infty} v_t^*(y) \cdot v_t^\perp(y)\, dy.  
\end{eqnarray}
Since $\langle v^*_t, v_t^\perp\rangle=0$, we have 
\[\int_1^{\infty} v_t^*\cdot v^\perp_t\,dy~  =~ \int_1^{\infty} v^*_t \cdot v^\perp_t \cdot (1- y^{-2})\, dy \]
The large $y$ asymptotics of $v_t^*$ and $v_t^\perp$ can be analysed using the 
same methods as in Appendix \ref{appendix:airyB} for $v_t^k$. 
We thus define $r_t^\diamond$, for $\diamond = k,*,\perp,$ by 
\[
r_t^{\diamond}\,=\,t^2(v_t^\diamond)''\,+\, \left(\frac{E_t}{y^2}-k^2\pi^2 \right)v_t^\diamond, 
\]
so that Prop. \ref{prop:expdecay} gives 
\begin{equation}\label{eq:estvdiamond}
\int_{1+2t^\alp}^{\infty} \left| v^\diamond_t(y)\right |^2\,dy~
\leq~ C \cdot \left(  t^{-2\alp} \cdot \int_1^{\infty} |r_t^\diamond|^2~
+~ \exp\left( -t^{\frac{3\alp-2}{2}}\right)\cdot \int_{1+t^\alp}^\infty
|v_t^\diamond(y)|^2 \, y^{-2} dy \right).
\end{equation}
We can now estimate $r_t^\diamond$ with the same techniques as in
Lemma \ref{lem:restimate} : we test again a smooth function
$\phi$ to obtain 
\[
\int_1^\infty r_t^\diamond(y) \phi(y) \,=\,
-a_t^k(v_t^\diamond,\phi)+E_t\cdot \langle v_t^\diamond ,\phi \rangle. 
\]
We now observe that $v_t^\diamond\otimes  e_k= P_t^\diamond w_t$ where 
$P_t^\diamond$ is some spectral projector associated with $a_t.$
Arguing as in Corollary \ref{coro:wisqm}, we thus obtain 
\[
\left | \int_1^\infty r_t^\diamond(y) \phi(y) \, dy \right | \,\leq \,
C'\cdot t\cdot \| w_t\| \cdot \| \phi\|.
\]  
This now implies (see the proof of Lemma \ref{lem:restimate}) 
\[
\int_1^\infty |r_t^\diamond(y)|^2\ dy \, \leq C\cdot t^2\cdot \|w_t\|^2.
\]

We plug this estimate into (\ref{eq:estvdiamond}) (see also the proof of
Corollary \ref{coro:normlocal}) to obtain that, 
for each $\alpha< \dt$, there exists a constant $C$
so that for sufficiently small $t$
\[  \int_{1+2t^{\alp}}^{\infty} |v_t^*|^2\, dy~  \leq~
     C \cdot t^{2-2 \alp} \cdot \|w_t\|^2,
\] 
and 
\[  \int_{1+2t^{\alp}}^{\infty} |v_t^\perp|^2\, dy~  \leq~
     C \cdot t^{2-2 \alp} \cdot \|w_t\|^2.
\] 
If $y\leq 1+2t^{\alp}$, then $(1-y^{-2}) \leq 4 t^{\alp}$ for sufficiently small $t$.
Therefore, by splitting the domain of integration into $[1, 1+2t^{\alp}]$
and $[1+2t^{\alp}, \infty[$ and using the Cauchy-Schwarz inequality,
we find that
\begin{equation} \label{est:pre_resolvent}
  \left|\int_{1}^{\infty} v_t^* \cdot v_t^\perp\, dy\right|~  
    \leq~  5 \cdot t^{\alp}  \cdot \left\|v^*_t\right\| \cdot \left\|v_t^\perp\right\|~ +~
     C \cdot t^{2-2 \alp} \cdot \|w_t\|^2
\end{equation}
for sufficiently small $t$.

We claim that $\|v_t^\perp\| = O(t^{\unt}) \cdot \|w_t\|$. Indeed,
by applying Lemma \ref{lem:quasimode} with $v \in V_k$, 
we have 
\[ |a_t^k(v_t^k,v)~ -~ E \cdot \langle v_t^k,v\rangle|~ \leq~ C \cdot t \cdot \|w_t\| \cdot \|v\| \]
for some constant $C$.  Thus, since the eigenvalue $\lambda^*$ satisfies 
$|E_t-\lambda^*|< C' \cdot t$ we find that
\begin{equation} \label{est:preresolv}
 |a_t^k(v_t^\perp,v)~ - E~ \cdot \langle v_t^\perp,v\rangle|~ \leq~ 2 C \cdot t \cdot \|w_t\| \cdot \|v\|. 
\end{equation}
By definition, $v_t^\perp$ is a spectral projection onto eigenspaces of $a_t^k$ whose associated
eigenvalues are distinct from $\lambda^*$. By Lemma \ref{lem:lambda_asymp}, 
there exists $\delta>0$ so that such eigenvalues differ from $\lambda^*$ by at least
$\delta \cdot t^{\dt}$. Because of (\ref{est:preresolv}), we can thus apply a resolvent 
estimate (e.g. Lemma 2.1 \cite{HJ10}) to find that
\begin{equation}  \label{est:resolvent_est}
  \|v_t^\perp\|~  \leq~  \frac{2C}{\delta} \cdot t^{\unt} \cdot \|w_t\|.
\end{equation}

By substituting (\ref{est:resolvent_est}) into (\ref{est:pre_resolvent})
and setting $\alpha=5/9$, we find a constant $C'$ so that 
\[  \left|\int_{1}^{\infty} v_t^* \cdot v_t^\perp\, dy\right|~  
    \leq~  C' \cdot t^{\frac{8}{9}} \cdot  \|w_t\|^2.
\]
By combining this estimate with (\ref{eq:at_ak_inner}),  (\ref{est:adot}), and  (\ref{eq:Edot}),
we obtain a constant $C''$ so that  
\[   \dot{E}_t \cdot \|u_t\|^2~ \geq~  \dot{\lambda}^* \cdot \|w_t^*\|^2~  
    -~ 2 C'' \cdot t^{-\frac{1}{9}} \cdot \|w_t\|^2.
\]
By orthogonality $\|w_t^*\|^2 = \|w_t^k\|^2 - \|w_t^\perp\|^2$, and hence by 
(\ref{est:resolvent_est})  and Lemma \ref{lem:lambda_asymp}, we have a constant $C'''$
so that 
\[  \dot{\lambda}^* \cdot \|w_t^*\|^2~ =~ \dot{\lambda}^* \cdot \|w_t^k\|^2~ -~ C''' \cdot t^{\unt} \cdot \|w_t\|^2.  
\]
The desired result follows.
\end{proof}

\begin{coro}  \label{coro:nottoonegative}
There exists $C'$ such that for each $\rho \in~]0,1[$, there 
exists $t_0>0$ such that if $0<t<t_0$, then 
\begin{equation} \label{eq:nottoonegative}
  \int_{[0,t] \setminus K(t,\rho)} \left(\dot{E}_s~ -~ \dot{\lambda}^*_s\right)~ ds~
   \geq~  C' \cdot \left(\rho^2- 1 \right) \cdot t^{\dt}.  
\end{equation}
\end{coro}

\begin{proof}
By definition, if $s \in [0,t]\setminus K(t,\rho)$, then $\|w_t^k\|^2/\|u_t\|^2\geq \rho^2$
and hence from Prop. \ref{prop:FdotbetweenXings}, we find that
\begin{equation*}  
  \dot{E}_t~ -~ {\dot{\lambda}}_t^*~ 
   \geq~ \left( \rho^2~ -~ 1 \right ) \cdot \dot{\lambda}^*_t~  -~
   C\cdot t^{-\frac{1}{9}}.
\end{equation*}
By using Lemma \ref{lem:lambda_asymp}, we find $C'$ and $t_0$ so that for $t<t_0$
\begin{equation*} 
  \dot{E}_t~ -~ {\dot{\lambda}}_t^*~ 
   \geq~ \dt \cdot C' \cdot \left( \rho^2~ -~ 1 \right ) \cdot t^{-\unt}.
\end{equation*}
The claim follows from integration.
\end{proof}

Finally, we  use Lemma \ref{lem:onK} and Corollary \ref{coro:nottoonegative} to prove
Theorem \ref{prop:dt}.  This will complete the proof of the main theorem. 

\begin{proof}[Proof of Theorem \ref{prop:dt}]
Apply Lemma \ref{lem:onK} with $\rho= 1/2$.  Then apply Corollary 
\ref{coro:nottoonegative} with $\rho=\rho_0 \geq 1/2$ such that
\[ C' \cdot \frac{\rho_0^2-1}{2}~ \geq~ -\frac{1}{2} \cdot \gamma \left(\frac{1}{2} \right). \]
Since $s \mapsto \dot{E}_s- \dot{\lambda}^*_s$ is positive on $K(t,\rho_0) \supset K(t, \frac{1}{2})$, 
we find that
\[  \int_{0}^{t} \left( \dot{E}_s- \dot{\lambda}^*_s\right)~ ds~
  \geq~ \frac{1}{2} \cdot \gamma \left(\frac{1}{2} \right) \cdot t^{\dt}.
\]
Since $\lim_{t \rightarrow 0}~ E_t - \lambda_t^*=0$, we have the desired conclusion.
\end{proof}


\appendix

\section{Eigenvalue branches of $a_t^{\ell}$} \label{appendix:airyA}

In this appendix, we compute the asymptotics of 
each real-analytic eigenvalue branch of $a_t^\ell$ for each ${\ell} \in \Zbb^+$.

\begin{prop} \label{lem:lambda_asymp}
Let $\ell \in \Zbb^+$ and let 
$t \mapsto \lambda_t$ be a real-analytic eigenvalue branch of $a_t^{\ell}$
for $t>0$. Then  
\begin{equation}\label{eq:lambda_asymp}
\lambda_t~ =~ (\ell \pi)^2~ +~ a \cdot t^{\dt}~ +~ O\left(t^{\qt}\right). 
\end{equation}
where $a=\left( 2 (\pi \ell)^2\right)^{\dt} \cdot (-\zeta)$ and
$\zeta$ is a zero of the derivative of the Airy function $A_-$ 
defined in (\ref{eq:Airy_asymp}). Moreover, 
\[ \lim_{t \rightarrow 0^+}~ \dot{\lambda}_t \cdot  t^{\unt}~ =~ \dt \cdot a. \] 
\end{prop}

To prove Proposition (\ref{lem:lambda_asymp}), we will first transform the 
eigenvalue problem into an eigenvalue problem that is easier to analyse.  
If $v$ is an eigenfunction of $a_t^{\ell}$ with respect to $\|\cdot \|$
with eigenvalue $\lambda$, then 
for each $w \in C_0^{\infty}([0,\infty[)$ and $t>0$, we have 
\[  t^2 \int_1^{\infty} v'\cdot w'\, dy~ +~ \mu \int_1^{\infty} v \cdot w\, dy~ =~ 
  \lambda \int_{1}^{\infty} \frac{v \cdot w}{y^2}~ dy, 
\]
where we have set $\mu= \ell^2\pi^2.$

Hence 
\[  t^2 \int_1^{\infty} v'\cdot w'\, dy~ 
     +~ \mu \int_1^{\infty} \frac{\left(y-1\right)\cdot(y+1)}{y^2} \cdot v \cdot w\, dy~ =~ 
  \left(\lambda-\mu\right) \int_{1}^{\infty} \frac{v \cdot w}{y^2}~ dy. 
\]
By  making the change of variable $y=t^{\dt}\cdot x +1$, letting 
$\overline{v}(x)= v\left(t^{\dt}\cdot x+1\right)$ and 
$\overline{w}(x)= w\left(t^{\dt} \cdot x+1 \right)$, and dividing by $t^{\qt}$,
we find that  
\[  \int_0^{\infty} \overline{v}_t'\cdot \overline{w}'\, dx~ 
     +~ \mu \int_0^{\infty} x \cdot g\left(t^{\dt} \cdot x\right) \cdot
    \overline{v}_t \cdot \overline{w}\, dx~ =~ 
  t^{-\dt} \cdot \left(\lambda-\mu\right) \int_{0}^{\infty} f(t^{\dt} \cdot x)
  \cdot \overline{v}_t
    \cdot \overline{w}\, dx. 
\]
where 
\[ f(z)~ =~  \frac{1}{(z+1)^2} \]
and 
\[ g(z)~ =~  \frac{z+2}{(z+1)^2} \]

This leads us to set $s = t^{\unt}$ and define for 
each $w \in C^{\infty}_0([0,\infty[)$ the quadratic forms 
\[  \Acal_s(w)~=~ \int_0^{\infty} (w')^2\, dx~ 
     +~ \mu \int_0^{\infty} x \cdot g(s^2 \cdot x) \cdot
    w^2 \, dx~
\]
and 
\[ \Ncal_s(w)~ =~  \int_{0}^{\infty} f(s^2 \cdot x) \cdot w^2\, dx. 
\]
Define $\Hcal := L^2([0,\infty),~ \frac{1}{x^2+1}\,dx)$ and 
$\Dcal := H^1([0,\infty))$ (i.e. the set of functions $u\in
L^2([0,\infty))$ such that the distributional derivative also is in $L^2$).

Then for each $s>0$, the form $\Ncal_s$ is a bounded quadratic
form on $\Hcal$ and $\Acal_s$ is a closed quadratic form on $\Hcal$
with domain $\Dcal$.

Since $w \mapsto \overline{w}$ maps bijectively 
$C_{0}^{\infty}([1, \infty])$ onto $C_{0}^{\infty}([0, \infty])$, the function $\overline{v}$ is an eigenfunction
of $\Acal_s$ with respect to $\Ncal_s$ with eigenvalue 
$\nu=s^{-2} \cdot \left(\lambda-\mu\right)$.

It follows from the perturbation theory of generalized 
eigenvalue problems (see \S VII.6 in \cite{Kato})  that
the eigenvalues of $\Acal_s$ with respect to $\Ncal_s$ 
can be organized into real-analytic eigenvalue branches for $s>0$.\footnote{At 
$s=0$ the domains of $\Acal_s$ and $\Ncal_s$ change, and hence analytic perturbation theory can not be applied.}

Since the generalized eigenvalue problem 
$\Acal_s(u,v)= \nu \cdot \Ncal_s(u,v)$ corresponds to 
a Sturm-Liouville problem with Neumann condition at $x=0$,
the eigenspaces are 1-dimensional.  Hence, we may enumerate the real-analytic 
eigenvalue branches $\nu_s^i$ so that for each $i \geq 0$ and $s>0$ 
we have 
\begin{equation} \label{order}
 \nu^i_s~ <~ \nu_s^{i+1}. 
\end{equation}

\begin{lem} \label{lem:bounded_branches}
For each $i$, there exists $s_0>0$ and $C$ so that if $s<s_0$ then 
\begin{equation}  \label{est:nu_dot}
 \left| \dot{\nu}_s^i \right|~ \leq~ C \cdot s. 
\end{equation}
In particular, there exists $a$ so that for small $s>0$ 
\begin{equation}  \label{eq:nu_asymp}
  \nu_s^i~ =~ a~ +~ O(s^2). 
\end{equation}
Moreover, $-a/(2 \mu)^{\dt}$ is a zero of the derivative of the Airy function
$A_-$. 
\end{lem}

\begin{proof}
First, we show that each $\nu_i^s$ is bounded. To this end, define 
\begin{equation*} 
  \Bcal(v)~ =~ \int_0^{\infty} \left(v'(x)\right)^2~ + 2  \mu \cdot x \cdot v(x)^2\,
   dx.
\end{equation*}
Since $g$ is bounded above by 2, we have $\Acal_s(v)\leq \Bcal(v)$ for 
each $s>0$ and $v \in C^{\infty}_0([0,\infty[)$. Note that for each $s \leq 1$,
we have $\Ncal_s(v) \geq  \Ncal_1(x)$ and hence 
\begin{equation}  \label{est:BtoA}
 \frac{\Acal_s(v)}{\Ncal_s(v)}~ \leq~  \frac{\Bcal(v)}{\Ncal_1(v)}. 
\end{equation}

Integration by parts shows that the 
eigenfunctions of $\Bcal$ with respect to $\Ncal_1$ are solutions
to the Sturm-Liouville problem
\[  -v''(x)~ + 2 \mu\cdot x \cdot v(x)~ =~  \frac{v(x)}{(1+x)^2}. 
\]
Standard convexity estimates on solutions to
ordinary differential equations imply that each 
eigenfunction belongs to the domain $\Dcal$ of $\Acal_s$ for each $s>0$. 
In particular,  the sum of the first $i$ eigenspaces of $\Bcal$ with 
respect to $\Ncal_1$ belongs to $\Dcal$.

Therefore, using (\ref{est:BtoA}), the minimax principle, and 
(\ref{order}) we find that $\nu_s^i$ is bounded by the $i^{{\rm th}}$
eigenvalue of $\Bcal$ with respect to $\Ncal_1$.

In the remainder of the argument we drop the superscript $i$
and focus on an individual real-analytic eigenfunction branch $u_s$
with eigenvalue $\nu_s$.
For $s>0$, we have 
\[ \dot{\nu}_s~ =~ \frac{ \dot{\Acal}_s( u_s)}{\Ncal_s(u_s)}~ -~ 
     \nu_s \cdot \frac{\dot{\Ncal}_s( u_s) }{ \Ncal_s(u_s)}
\]
where $\cdot$ indicates differentiation with respect to $s$. 
A computation gives that for each $w$
\[   \dot{\Acal}_s( w)~ 
    =~ 2 s \cdot \mu \int_0^{\infty} x^2 \cdot g'(s^2 \cdot x) \cdot w(x)^2 \, dx, 
\]
and
\[    \dot{\Ncal}_s(w)~ =~ 2s \int_0^{\infty} x \cdot f'(s^2 \cdot x) \cdot w(x)^2\, dx. 
\] 

Let $u_s$ be a real-analytic eigenfunction branch of $\Acal_s$ with respect to $\Ncal_s$ 
associated to the real-analytic eigenvalue branch $\nu_s$.
Integration by parts gives
\begin{equation} \label{eq:A-N-ODE}
   -u_s''(x)
     +~ \mu \cdot x \cdot g\left(s^2 \cdot x\right) \cdot u_s(x)~ 
    =~ \nu_s \cdot f(s^2 \cdot x) \cdot  u_s(x).
\end{equation}
Let $M$ be the upper bound on $\nu_s$ proven above. 
If $s \leq 1$ and 
$x > x_0:=\max\{1,M/\mu\}$, then  
\[  \mu \cdot x \cdot g(s^2\cdot x)~ -~ 
\nu_s \cdot f(s^2\cdot x)~ \geq~ \frac{\mu}{2} \]
and hence  $u_s''u_s(x)~ \geq~ \frac{\mu}{2} \cdot u_s^2(x)$ for $s \leq 1$.
It follows that $(u^2_s)''(x) \geq \mu \cdot u_s^2(x)$ for $x\geq x_0$. 
Thus, since $\Ncal(u_s)$ is finite, 
we find that for $x_0 \leq x \leq y$ 
\begin{equation}  \label{A-N-convexity}
   \frac{u_s(y)^2}{u_s(x)^2}~ \leq~ 
  \frac{\exp\left(-\sqrt{\mu} \cdot y \right)}{\exp\left(-\sqrt{\mu} 
   \cdot x \right)}.
\end{equation}
Integrating from $x_0$ to $2x_0$, we find a constant $C$ (that depends
on $x_0$) such that, for $y >2x_0$  we have
\[  y^2 \cdot u(y)^2\, \leq~
      C  \cdot  y^2\exp(-\sqrt{\mu}y)\cdot 
\int_{x^{0}}^{2x^0} \frac{u(x)^2}{1+x^2}\, dx.
\] 
From this we find constants $C$ such that that
\[ \left|\dot{\Acal}_s(u_s)\right|~ \leq~ C\cdot s \cdot \Ncal_s(u_s).
\]
A similar argument shows that
\[ \left|\dot{\Ncal}_s(u_s)\right|~ \leq~ C\cdot s \cdot \Ncal_s(u_s).
\]
Therefore, (\ref{est:nu_dot}) holds, and via integration we 
find $a$ so that (\ref{eq:nu_asymp})
holds true.

Continuity of solutions to ordinary differential 
equations with respect to coefficients applies
to (\ref{eq:A-N-ODE}) with fixed initial 
conditions $u_s'(0)=0$ and $u_s(0)=1$. 
In particular, we have a solution $u_0$ to 
\[  -u_0''(x)
     +~ 2 \mu \cdot x \cdot u_0(x)~ 
    =~ a \cdot  u_0(x).
\]
It follows that 
\[  v(z)~ :=~ u_0 \left((2\mu)^{-\unt}\cdot z + (2\mu)^{-1} \cdot a) \right)
\]  
is a solution
to $v''(z)= z \cdot v(z)$. Estimate (\ref{A-N-convexity}) applies to
$u_0$,  and hence it follows from (\ref{eq:Airy_asymp}) that 
$v$ is a multiple of $A_-$. The function $u$ satisfies the 
Neumann condition $u'(0)=0$ and hence $v'(- (2\mu)^{-\dt}\cdot a)=0$ as desired.
\end{proof}

\begin{proof}[Proof of Proposition \ref{lem:lambda_asymp}]
If $v_t$ is a real-analytic eigenfunction branch of of $a_t^{\ell}$ 
associated to the eigenvalue branch $\lambda_t$, then $\overline{v}_{s^3}$ is a 
real-analytic eigenfunction branch of $\Acal_s$ with eigenvalue branch 
$\nu_s= s^{-2} (\lambda_{s^3}-\mu)$.  Lemma \ref{lem:bounded_branches} 
implies that
\[  \lambda_t~ =~ \mu + a \cdot t^{\dt}~ +~ O(t^{\qt}).
\]
By differentiating $\lambda_{s^{3}} = \mu + s^2 \cdot \nu_s$, we find that
\[  3  \dot{\lambda}_{s^3}~ =~  \dot{\nu}_s~ +~ 2 \cdot \nu_s \cdot s^{-1}.\]
By  Lemma \ref{lem:bounded_branches}, both 
$\dot{\nu}_s$ and $\nu_s$ are bounded. Therefore, $\dot{\lambda}_{s^3}=O(s^{-1})$
and hence $\dot{\lambda}_{t}=O(t^{-\unt})$. 
\end{proof}


\section{The off-diagonal estimates} \label{appendix:airyB}

Let $(E_t, u_t)$ be a real-analytic eigenbranch of $q_t$  such that
$\lim_{t \rightarrow 0} E_t = E_0 =(\pi \cdot k)^2$ for some positive integer $k$. 
For a fixed constant $C>0$, let 
\[  I~ =~ [E_0-C, E_0 + C]. \]
As in \S \ref{sec:nota}, let $w_t$ denote the orthogonal 
projection of $u_t$ onto 
the sum of the eigenspaces of $a_t$ whose eigenvalues lie in $I$.

The purpose of this appendix is to prove the following fact that is
crucially used in the proof of Proposition \ref{prop:Xings}. We recall
that $b_t$ is the quadratic form defined in (\ref{eq:defb_t}).

\begin{prop} \label{prop:blower}
Let $\eta>0$. There exists $\kappa>0$, $\delta>0,$ and  $t_0>0$ such that,
if $t<t_0$ and if $\psi^0$ is an eigenfunction of $a_t^0$ 
with  eigenvalue $\lambda^0$ satisfying 
\begin{equation} \label{est:5/3}
|\lambda^0-E_t|~ \leq~  \eta \cdot t^{\frac{5}{3}}
\end{equation}
then 
\begin{equation}\label{eq:offdiagbt}
\left | b_t(u_t,\psi^0 \otimes 1)\right|~ \geq~ \kappa \cdot t^{\dt}\cdot 
\left( \| w^k_t\|-t^\delta\cdot \|u_t\|\right) \cdot \|\psi^0\|.   
\end{equation} 
\end{prop}

\begin{remk}
The condition on $\lambda^0$ is only used to ensure that, when $t$
tends to $0$, $\lambda^0$ tends to $k^2\pi^2.$
\end{remk}

\begin{proof}
Proposition \ref{prop:q-a-tb} says that the quadratic form $b_t$ is
controlled by $\wat$:  there exists a
constant $C$ such that for $u,v\in \dom(a_t)$ 
\[
 \left | b_t(u,v)\right |\,\leq \, C \cdot
\wat(u)^\und\cdot \wat(v)^\und.
\] 
Thus, lemma \ref{lem:u_w_close} and the fact that $\wat(\psi^0)\,=\,O(\|\psi_0\|^2)$
imply that 

\[ b_t(u_t-w_t,\psi^0 \otimes 1)~ =~  O (t) \cdot \|u_t\| \cdot \|\psi^0\|, \]
and hence it suffices to bound $b_t(w_t, \psi^0)$ from below.

Observe also that lemma \ref{lem:u_w_close} also implies that
$\|u_t\| \sim \|w_t\|$ in the limit $t\rightarrow 0$ so that we can
freely replace $\|u_t\|$ by $\|w_t\|$ and vice-versa in each
(multiplicative) estimate.

By the discussion \S \ref{subsec:separation} and \S \ref{sec:nota}, 
for each $t$ we can uniquely write 
\[  w_t(x,y)~ =~ \sum_{\ell \leq k}~ 
  \sum_{\lambda \in {\rm spec}(a_t^{\ell})\cap I_t}~ \psi_{\lambda}^\ell(y)  \cdot e_\ell(x)
\]
where each $\psi_{\lambda}^\ell(y)$ is an eigenfunction of $a_t^{\ell}$
with eigenvalue $\lambda \in I_t$.  Set
\begin{equation}  \label{eq:vt}
  v_{t}^{\ell}(y)~  =~  \sum_{\lambda \in {\rm spec}(a_t^{\ell})\cap I_t}~ \psi_{\lambda}^\ell(y). 
\end{equation}
and 
\[  w_t^{\ell}(x,y)~ =  v_t^\ell(y) \cdot e_\ell(x). \]
By linearity
\begin{equation}  \label{eq:linearity}
 b_t(w_t,\psi^0 \otimes 1)~ =~ \sum_{\ell \leq k}~ 
   b_t \left( w_t^{\ell} , \psi^0 \otimes 1\right). 
\end{equation}

From (\ref{eq:defb_t}) we have  
\[  b_t\left(w_t^{\ell}, \psi^0 \otimes 1 \right)~ 
    =~ \int_{1}^{\ua} \int_{0}^{1}  \widetilde{\nabla}_t w_t^{\ell} \cdot 
  \left( \begin{array}{cc} 0 & x \cdot p(y) \\ x \cdot p(y) & 0  \end{array} \right) \cdot
     \left( \widetilde{\nabla}_t \psi^0(y)\right )^*~ dxdy.
\]
where $\widetilde{\nabla}_tf = [\partial_x f, t \partial_y f]$
and $p(y)$ is defined in (\ref{defn:p}).
Since $\partial_x \psi(y) \equiv 0$, and $e_\ell(x) =
2^{-\und}\cos(\ell \pi x)$, we find that 
\begin{equation*} 
b_t\left(w_t^{\ell}, \psi^0 \otimes 1\right)~
  =~  \left( - 2^{-\und}\ell \pi \cdot \int_0^1 x \cdot \sin(\ell \pi x) 
   dx\right)
\cdot \left(\int_1^{\ua} p(y) \cdot v_t^\ell(y) 
   \cdot \left(t \cdot (\psi^0)'(y)\right)dy\right)   
\end{equation*}
If $\ell=0$, then $\sin(\ell \pi x)\equiv 0$, and so $b_t(w_t^{\ell}, \psi^0 \otimes 1)=0$. 
For $0< \ell< k$, apply Lemma \ref{lem:bwellpsi} below to find that
\begin{equation} \label{eq:vlower}
\left| b_t(w_t^{\ell}, \psi^0 \otimes 1) \right|~ 
   =~ O_{\ell}(t) \cdot \| v_t^{\ell}\| \cdot \|\psi^0\|. 
\end{equation} 
Since $w_t^{\ell}$ and $w_t^{\ell'}$ are orthogonal 
if $\ell \neq \ell'$, we have
\[
  \sum_{\ell=1}^{k-1}  \|v_t^{\ell}\|^2~ =~
  2^{-\und} \sum_{\ell=1}^{k-1} \| w_t^{\ell} \|^2~
 \leq~ \|w_{t}\|^2
\]
Thus, by summing (\ref{eq:vlower}) over $\ell \in \{0,\ldots, k-1\}$,
we obtain 
\begin{equation}  \label{eq:sumvl}
\begin{split}
\left| \sum_{\ell=0}^{k-1}~ 
   b_t \left(   w_t^\ell, \psi^0 \otimes 1 \right)
\right|~  
& \leq O(t) \cdot \left( \sum_1^{k-1}
  \|v_t^\ell\|\right)\cdot\|\psi_t^0\| \\
& \leq O(t) \left( \sum_{1}^{k-1} \| v_t^\ell\|^2\right)^\und\cdot \|
\psi_t^0\| \\
& \leq~ O(t)  \cdot    \| w_{t}\| \cdot\|\psi^0_t \|
\end{split}
\end{equation}

For $\ell=k$, we have 
\[  k\pi \int_0^1 ~x\sin(k\pi x)\, dx 
  =~  (-1)^k~ \neq~ 0
\] 
Thus, from Lemma \ref{lem:bwkpsi} and Lemma \ref{lem:p_positive}, there exists
$\kappa'>0$ so that 
\[
\left | b_t(w_t^k,\psi^0 \otimes 1) \right|~ \geq~ 
 \kappa' \cdot t^\dt \cdot \left( \| w_t^k \|-t^\delta \|u_t\|\right) \cdot \|\psi^0\| 
\]
for some $\kappa'>0$. 
The latter estimate, combined with (\ref{eq:linearity}),
 (\ref{eq:sumvl}), and the triangle inequality, yield the claim.
\end{proof}

\begin{lem}\label{lem:bwellpsi}
For each smooth function $g: [1, \ua] \rightarrow \Rbb$,
there exists $C>0$ and $t_0>0$ such that if $t\leq t_0$ and $0<\ell<k$,
then  
\begin{equation} \label{eq:ellsmall}
\left | \int_1^{\ua} g(y) \cdot v_t^{\ell}(y) \cdot
    \left( t \cdot \left(\psi^0\right)'(y)\right)~ 
   dy \right |~
 \,\leq \, C\cdot t \cdot  \left\|u_t \right \| \cdot \left\|\psi^0 \right\|,
\end{equation}

\end{lem}

\begin{lem}  \label{lem:bwkpsi}
For each smooth function $g: [1, \ua] \rightarrow \Rbb$
with $g(1) \neq 0$, there exists $\kappa,\delta,t_0 >0$ such that, for each $t<t_0$ 
\[
\left| \int_1^{\ua} g(y) \cdot v_t^k(y) \cdot 
  \left( t \cdot \left(\psi^0\right)'(y)\right)~ dy\right|~ 
\geq~ \kappa \cdot t^{\dt} \cdot 
\left( \|w_t^k\|\, - t^\delta \left\|u_t \right\|\right)\cdot \|\psi^0\|.
\]
\end{lem}

The remainder of this appendix is devoted to proving the preceding lemmas.


\subsection{The proof of Lemma \ref{lem:bwellpsi}}

Define  $r_t^{\ell}:[1, \infty[\, \rightarrow \Rbb$ by
\begin{equation} \label{eq:rdefn}
  r_t^{\ell}(y)~ =~ t^2 \cdot (v_t^{\ell})''(y)~ 
+~  \left(\frac{E_t}{y^2}~ -~ (\ell \cdot \pi)^2 \right) \cdot v_t^{\ell}(y).
\end{equation}
\begin{lem} \label{lem:restimate}
There exists $t_0>0$ and $C$ so that if $t<t_0$, then for each $\ell \in \Nbb$
\[ 
\int_{1}^{\infty} \left|r_t^{\ell}(y) \right|^2~ dy~ \leq~ C \cdot t^2 \cdot \|u_t\|^2
\]
\end{lem}

\begin{proof}
Multiply both sides of (\ref{eq:rdefn}) by a smooth function with
compact support $\phi^\ell$ and integrate 
over $y \in [1, \infty[$, then integrate by parts to obtain
\[ \int_1^{\infty} r_t^{\ell}(y)\phi^\ell(y)~ dy~ =~
-a_t^{\ell}(v_t^{\ell}, \phi)~ +~ E_t \cdot 
  \langle v_t^{\ell}, \phi \rangle. 
\]
Observe that $a_t^\ell(v_t^\ell,\phi^\ell)-E_t\cdot \langle
v_t^{\ell}, \phi \rangle \,=\,a_t(w_t,\phi^\ell \otimes e_\ell)-E_t
\langle w_t,\phi^\ell\otimes e_\ell\rangle$ 
so that by applying Lemma \ref{coro:wisqm} 
to the test function
$\phi^\ell\otimes e_\ell$, there exists $t_0>0$ and $C'$ such that for
$t<t_0$, we have 
\[  \left |\int_1^\infty~ r_t^\ell(y) \phi^\ell(y) \, dy \right |
    \leq~ C' \cdot t \cdot \|u_t\| \cdot \|\phi^\ell \|.
\]
Recalling that the $L^2$-norm on the right hand side has the weight $y^{-2},$ this implies that 
\[
\int_1^\infty y^2 \left | r_t^\ell(y)\right|^2 \, dy \,\leq\, \left(
  C'\cdot t\cdot \|u_t\|\right)^2. 
\]
The claim follows since $y^2\geq 1$ on the interval over which we
integrate. 
\end{proof}

The strategy of the proof of Lemma \ref{lem:bwellpsi} is as follows. 
By (\ref{eq:rdefn}), the function $v^{\ell}_t$ is a solution to the inhomogeneous equation 
\begin{equation} \label{eq:inhomoODE}
  t^2 \cdot v''~  +~ f_{\mu}^{\ell} \cdot v~ =~   r~
\end{equation}
where $\mu=E_t$ and 
\begin{equation} \label{eq:ftl-defn}
 f_{\mu}^{\ell}(y)~ :=~ \frac{\mu}{y^2}~ -~ (\ell \cdot \pi)^2.
\end{equation}
The function $\psi^0_t$ is a solution to the homogeneous equation
\begin{equation} \label{eq:homoODE}
 v''~  +~ t^{-2} \cdot f_{\mu}^{\ell} \cdot v~ =~   0
\end{equation}
where $\mu=\lambda^0$.
Our choice of $\beta$ in (\ref{eq:beta}) implies that $f_{\mu}^{\ell}$ is 
bounded below by a constant $\delta_1>0$ for all small $t$, $\ell<k$,
and $\mu \in I$. 
Hence we can use WKB type estimates to find a 
basis $v_{\pm}$ of solutions to the homogeneous equation (\ref{eq:homoODE}). 
We will then use `variation of parameters' to express each solution 
to (\ref{eq:inhomoODE}) in terms of this basis, and we use 
Lemma \ref{lem:restimate} to provide control of the inhomogeneous term $r$. 
Finally, we will estimate the integral in (\ref{eq:ellsmall}) 
using a Riemann-Lebesgue type estimate.

\begin{proof}[Proof of Lemma \ref{lem:bwellpsi}]
For $\ell<k$ and $\mu \in I$, we have $f_{\mu}^{\ell} \geq \delta_1>0$, and hence we
can apply Theorem 6.2.1 in \cite{Olver}, to obtain a 
basis $(v_{\mu,+}^{\ell},v_{\mu,-}^{\ell})$ of solutions to the homogeneous equation
(\ref{eq:homoODE}) that satisfy
\begin{equation}  \label{eq:WKB}
v^{\ell}_{\mu,\pm}(y)\,=\, \left | f_{\mu}^\ell(y)\right |^{-\unq}
\exp\left(\pm \frac{i}{t} \int_{1}^{y}\left| f_{\mu}^\ell(z)
\right|^{\und}\, dz\right)(1\,+\,\eps(y)) 
\end{equation}
and 
\begin{equation} \label{eq:WKBderivate}
t \cdot \left(v_{\mu,\pm}^{\ell}\right)'(y)\,=\, \pm i\cdot \left | f_{\mu}^\ell(y)\right |^{\unq}
\cdot \exp\left(\pm \frac{i}{t} \int_{1}^{y}\left| f_{\mu}^\ell(z)
\right|^{\und}\, dz\right)(1\,+\,\overline{\eps}(y)) 
\end{equation}
where,  for $\mu \in I$ and $\ell<k$, the smooth functions
$\eps$ and $\overline{\eps}$ have $C^0$ norm that is uniformly $O(t)$.

Observe that $v^\ell_{\mu, \pm}$ have $L^2([1,\beta])$ norms that are
uniformly bounded above and away from $0.$ Moreover, since $v_{\mu,+}^{\ell}$ and 
$v_{\mu,-}^{\ell}$ are highly oscillatory for small $t$,
an integration by parts argument shows that the 
$L^2([1,\beta])$ inner product $\langle v_{\mu,+}^{\ell},v_{\mu,-}^{\ell}\rangle$
is $O(t)$. It follows that there exists $m>0$ such that if
$(a_+,a_-) \in \Cbb$, then 
\begin{equation} \label{eq:basis_equivalence}
m\cdot \left( |a_+|^2+|a_-|^2 \right)^\und \,
  \leq \, \left\| a_+ \cdot v_{\mu,+}^{\ell} \,+\,a_- \cdot v_{\mu,-}^{\ell} \right\| 
  \,\leq\, m^{-1} \cdot \left(|a_+|^2+|a_-|^2 \right)^\und.
\end{equation}
for all sufficiently small $t$.  Here $\|\cdot\|$ denotes the  $L^2([1,\beta])$ norm.

By the method of `variation of constants', each solution to 
\begin{equation} \label{eq:generic_inhomo}
  v''~ +~  t^{-2} \cdot f^{\ell}_{\mu} \cdot  v~ =~ t^{-2} \cdot r
\end{equation}
is of the form
\begin{equation}\label{eq:defv_var}
v~  =~ \left( a_+ + h^{\ell,r}_{\mu,+}\right) \cdot v_{\mu,+}^{\ell}~ 
 +~ \left( a_- +h^{\ell,r}_{\mu,-}\right) \cdot v_{\mu,-}^{\ell}
\end{equation}
where $(a_+, a_-) \in \Cbb^2$, 
\[
h^{\ell,r}_{\mu,\pm}(y)~  
  =~ \pm \, t^{-2} \cdot \wcal^{-1}\int_1^y r(z) \cdot v_{\mu,\mp}^{\ell}(z)~ dz,
\]
and $\wcal= v_{\mu,+}' \cdot v_{\mu,-} - v_{\mu,-}' \cdot v_{\mu,+}$
is the Wronskian.

In particular, for each $\ell$ and each $t$, 
there exists $(a_{t,+}^{\ell}, a_{t,-}^{\ell}) \in \Cbb^2$
so that the function $v_t^{\ell}$ of (\ref{eq:vt}) satisfies
\begin{equation*}
\begin{split}
   v^{\ell}_{E_t}~  
   &=~ \left( a_{t,+}^{\ell} + h_{E_t,+}^{\ell,r^{\ell}_t} \right) \cdot v_{E_t,+}^{\ell}~ 
    +~ \left( a_{t,-}^{\ell} +h_{E_t,-}^{\ell,r^{\ell}_t} \right) \cdot v_{E_t,-}^{\ell}\\
\end{split}
\end{equation*}
The eigenfunction $\psi^0$ of $a_t^{0}$ satisfies (\ref{eq:generic_inhomo})
with $r=0$, and hence there exists $(c_+, c_-) \in \Cbb^2$ so that
\[   \psi^0~ =~ c_+ \cdot v_{\lambda^0,+}^0~ 
    +~ c_- \cdot v_{\lambda^0,-}^0. 
\]

The integral in (\ref{eq:ellsmall}) is equal to  
\[ 
  \int_1^{\ua}  g \cdot 
  \left( \sum_{\pm} 
\left( a_{t,\pm}^{\ell} + h_{E_t,\pm}^{\ell, r^{\ell}_t} \right) 
  \cdot v_{E_t\pm}^{\ell} \right)
   \cdot 
   \left(\sum_{\pm} {c_{\pm}} 
  \cdot \left(t \cdot {v_{\lambda^0,\pm}^0} \right)' \right)~ dy.
\]
By expanding the product of sums, one  obtains a sum of $2^3$ integrals. 
By substituting the expressions (\ref{eq:WKBderivate}) and (\ref{eq:WKB}),
integration by parts, and applying standard estimates, we find that
each integral is $O(t)$. 

For example, consider the terms of the form
\begin{equation}  \label{eq:aa}
a_{t,\pm}^{\ell} \cdot {c_\pm} \int_1^{\ua} g \cdot 
 \left( \frac{f_{\lambda^0}^{0}}{f_{E_t}^{\ell}}\right)^{\unq}
    \exp\left( \frac{i}{t} 
 \int_1^y  \pm \left(f_{E_t}^{\ell}\right)^{\und} \mp 
   \left(f_{\lambda_0}^{0}\right)^{\und} \right)\cdot (1+\eps^*)~dy.
\end{equation}
Since $\ell>0$, an elementary computation shows that there exists $\delta>0$
so that if $z \in [1, \beta]$ and $t$ is sufficiently small, then 
\begin{equation} \label{eq:phase}
 \delta~ \leq~
  \left|\left(f_{E_t}^{\ell}(z)\right)^{\und} \pm \left(f_{\lambda^0}^{0}(z)\right)^{\und}\right|
\end{equation}
Thus, we may integrate by parts to find a
constant $C$ so that the integral in (\ref{eq:aa})  is at most 
$C \cdot t \cdot \|g\|_{C^1}$. 
It follows that all the terms of this form are bounded above by 

\begin{equation}\label{eq:type1}
C' \cdot t \cdot \|g\|_{C^1} \cdot \| \psi^{0}\|\cdot \left(
  \sum_{\pm}|a_{t,\pm}^\ell |\right).
\end{equation}

We also have terms of the form 
\begin{equation}  \label{eq:aA}
 {c_\pm} \int_1^{\ua} g \cdot h_{E_t,\pm}^{\ell, r^{\ell}_t} \cdot 
 \left( \frac{f_{\lambda^0}^{0}}{f_{E_t}^{\ell}}\right)^{\unq}
     \exp\left( \frac{i}{t} 
 \int_1^y  \pm \left(f_{\mu}^{\ell}\right)^{\und} \mp 
   \left(f_{\mu}^{0}\right)^{\und} \right)\cdot (1+\eps^*)~dy.
\end{equation}
We integrate by parts as above, but this time we need 
to also bound $h_{\pm}= h_{E_t,\pm}^{\ell,r^{\ell}_t}$ and its derivative. 

From (\ref{eq:WKB}) and (\ref{eq:WKBderivate}) we find that
there exists $t_1$ so that if $t < t_1$, then  
$ \left|t \cdot \wcal \right| \geq 1$. From (\ref{eq:WKB}) and (\ref{eq:ftl-defn})
we find that for each $\ell$  
\begin{equation}
 \sup_{y \in [1, \beta]}~ \left|v_{E_t,\mp}^{\ell}(y) \right|~ \leq~ \frac{2}{\sqrt{\delta_1}}
\end{equation}
for all sufficiently small $t$.
Hence, using the
Cauchy-Schwarz inequality and  Lemma \ref{lem:restimate} we have, for 
all $y\in [1,\beta]$,  
\begin{eqnarray*}  \label{eq:boundApm}
  \left | h_{\pm}(y)\right | &\leq &  \nonumber
    \frac{1}{t} \cdot \left|\int_1^{y} r(y) \cdot v_{E_t,\mp}^{\ell}(y)~ dy \right|  \\
    &\leq &  \nonumber 
    \frac{1}{t} \cdot \left(\int_1^{y} r(y)^2~ dy \right)^{\und} \cdot 
       \left(\int_1^{y} v_{E_t,\mp}^{\ell}(y)^2~ dy \right)^{\und} \\
    &\leq &  
    \frac{1}{t} \cdot C \cdot t \cdot \|u_t\| \cdot 
       \sqrt{\beta-1} \cdot 2 \cdot \delta_1^{-\und} 
\end{eqnarray*}
for all sufficiently small $t$. For the derivative of $h^{\ell,r}_{\mu,\pm}$, we have 
\[  
\left | h_{\pm}'(y) \right |~ \leq~ 
    \frac{3}{t}  \cdot \left|r(y)\right| \cdot 2 \delta^{-\und}.
\]
Applying Cauchy-Schwarz and Lemma \ref{lem:restimate} gives 
\begin{eqnarray} \label{eq:boundAdotpm}
   \int_{1}^{\ua} \left| h_{\pm}'(y) \right|~ dy & \leq & 
    \frac{6 \cdot \delta^{-\und}}{t} \sqrt{\beta-1}\cdot 
     \left(\int_1^{\beta} \left|r(y)\right|^2\,dy \right)^{\und}~ \\
 & \leq & 
    \frac{2 \cdot \delta^{-\und}}{t} \sqrt{\beta-1}\cdot 
     C \cdot t \cdot \|u_t\|. \nonumber
\end{eqnarray}

Finally, we apply integration by parts to (\ref{eq:aA}). 
The resulting terms that do not contain $h_{\pm}'$ have uniformly
bounded $C^0$ norm.  The term that contains  $h_{\pm}'$ can be bounded
using (\ref{eq:boundAdotpm}). 
It follows that all the term of this form are bounded by 
\begin{equation}\label{eq:type2}
C\cdot t\cdot \| u_t\| \cdot \| \psi^0\|.
\end{equation}
The final step consists in bounding $\sum |a_\pm|$ by $\|v^\ell_t\|$ 
to control the terms of eq. (\ref{eq:type1}).

Using (\ref{eq:basis_equivalence}) and (\ref{eq:defv_var}) we have 
\[
m\left( |a_+|^2\,+\,|a_-|^2\right)^\und \,\leq\,\|v_t^\ell\|
\,+\,(\beta-1)\cdot \frac{2}{\sqrt{\delta_1}} \sup_{[1,\beta]}\{|h_+(y)|+|h_-(y)|\} 
\]
By orthogonality we have $\|v_t^\ell\|\leq \|u_t\|$ and, 
using the bound on $|h_\pm(y)|$ we finally obtain  
\[
\left( |a_+|^2\,+\,|a_-|^2\right)^\und \,\leq\,C\cdot \|u_t\|.
\]
This finishes the proof.
\end{proof}


\subsection{The proof of Lemma \ref{lem:bwkpsi}}

As in the previous subsection, 
the function $v_t^{k}$ is a solution to the inhomogeneous 
equation (\ref{eq:inhomoODE}) with $\mu=E_t$ and $r$ defined by 
(\ref{eq:rdefn}). However, for $\ell=k$, the function 
\[  f_{t}^{k}(y)~ =~ \frac{E_t}{y^2}~ -~ k^2\pi^2 
 \]
is negative for large $y$. In fact, since $E_t$ decreases 
to $(\pi k)^2$, the function $f_{t}^{k}$ changes sign nearer and
nearer to $y=1$. 
Since the solution $v_t^{k}$ belongs to $L^2(\Rbb,y^{-2} dx dy)$, 
we expect it to decay exponentially as soon as $y$ moves away from $1$. 
For $y$ near 1, we will approximate $v_t^{\ell}$ using Airy functions. 
In this subsection we will make these approximations precise
and use them to give a proof of Lemma  \ref{lem:bwkpsi}

\subsubsection{Normalization of  $\psi^0$}

By Lemma \ref{ZeroLemma},  $\psi^0$ is a constant multiple of
$\psi$ defined in (\ref{ZeroEigenvector}). 
Because both sides of the estimate in Lemma \ref{lem:bwkpsi} 
are homogeneous functions of degree 1 in $\psi^0$, it 
suffices to assume that $\psi^0=\psi$.

Let $\left| f\right|_0$ denote the supremum norm of $f$ over $[1, \beta]$.

\begin{lem}  \label{lem_psi_0_bound}
There exists $t_0>0$ such that if $t<t_0$ and $\lambda^0 \in I$, then 
\begin{equation}  \label{est:psi_0_sup}  
  \und~ \leq~  |\psi|_0~  \leq~ 2 \sqrt{\beta},
\end{equation}
\begin{equation}  \label{est:psi_0_L2}
     \frac{\sqrt{\ln(\beta)}}{2}~
    \leq~  \|\psi\|~ \leq~\sqrt{\ln(\beta)}
\end{equation}
and 
\begin{equation}  \label{eq:psi_zero_deriv}
 \left| t \cdot \psi'\right|_{0}~ \leq~  
   2 \sqrt{\sup(I)}.
\end{equation}
\end{lem}

\begin{proof} 
We have  
\begin{equation} \label{eq:psi_0_decomp} 
\psi(y)~ =~ \omega^+(r,y)~ -~ (2r)^{-1} \cdot \omega^-(r,y)
\end{equation}
 where 
\[  \omega^+(r, y)~ =~ y^{\und} \cdot \cos(r \cdot \ln(y)), \]
\[  \omega^-(r, y)~  =~ y^{\und} \cdot \sin(r \cdot \ln(y)), \]
and  $\lambda^0= t^2 \cdot (1/4+r^2)$. In particular, for sufficiently small $t$ 
\begin{equation} \label{est:r_asymp}
    \frac{\sqrt{\inf(I)}}{2t}~    \leq~ r~  \leq~ \frac{\sqrt{\sup(I)}}{t}.   
\end{equation}
Thus, since $|\omega^{\pm}|_0\leq \sqrt{\beta}$ and $|\omega^+|_0\geq 1,$ the triangle inequality
applied to (\ref{eq:psi_0_decomp})
implies that  (\ref{est:psi_0_sup}) holds for sufficiently small $t$.

\begin{equation}
\begin{split}
\int_1^\beta \left | \omega_+(y)\right|^2 y^{-2}dy & = \int_1^\beta
|\cos(r\ln(y)|^2 \frac{dy}{y} \\
& = \int_0^{\ln \beta} |\cos(rz)|^2 dz \\
&= \frac{1}{2}\ln \beta + O(\frac{1}{r}). 
\end{split}
\end{equation}
The same estimate applies for $\omega_-.$
Hence the triangle inequality and (\ref{est:r_asymp}) imply that 
(\ref{est:psi_0_L2}) holds for sufficiently small $t$.

The bound on $t \cdot \psi'$ is proven in a similar fashion using the fact that
\begin{equation} \label{eq:psi_prime}  \psi'(y)~ 
  = - \left(r+ \frac{1}{4r}\right) \cdot y^{-1} \cdot \omega^-(r,y)
\end{equation}
together with (\ref{est:r_asymp}) and the fact that $r$ is of order $t^{-1}$.
\end{proof}

\subsubsection{Localization near $y=1$}

The following proposition provides a quantitative description of the 
concentration of solutions to $t^2 \cdot v''+ f_t^k \cdot v=r$ near $y=1$.

\begin{prop}\label{prop:expdecay}
Let $k \in \Zbb^+$. For each $\alp \in~ ]0,\frac{2}{3}[$, 
there exists $t_0>0$ and $C$, such that if $v$ is a solution to  
$t^2 \cdot v''+ f_t^k \cdot v=r$ and $t<t_0$, 
\begin{equation}\label{eq:expdecay}
\int_{1+2t^\alp}^{\infty} \left| v(y)\right |^2\,dy~
\leq~ C \cdot \left(  t^{-2\alp} \cdot \int_1^{\infty} |r|^2~
+~ \exp\left( -t^{\frac{3\alp-2}{2}}\right)\cdot \int_{1+t^\alp}^\infty
v^2(y) \, y^{-2} dy \right).
\end{equation}
\end{prop}

\begin{proof}
By Proposition \ref{Tracking} and Lemma \ref{lem:lambda_asymp}, there exists $C$ 
so that if $t$ is sufficiently small, then 
\begin{equation} \label{23window}
    E_t~ \leq~  (k\pi)^2~ +~ C \cdot t^\dt,
\end{equation}
Hence, for $y \geq  1+ t^{\alp}$, one finds that  
\[ (k\pi)^2~ -~ \frac{E_t}{y^2}~ \geq~ (k\pi)^2 \cdot \left( \frac{ 2\cdot t^{\alp}\,
+\,  t^{2 \alp}}{1+ 2 t^{\alp} + t^{2\alp}} \right)~ -~ C \cdot t^{\dt}.
\]
Thus, since $\alpha< 2/3$, there exists $t_1>0$ such that if $t\leq t_1$ 
then for all $y \geq 1+t^{\alpha}$ we have
\begin{equation} \label{eq:LowerBound}
  (k\pi)^2~ -~ \frac{E_t}{y^2}~ \geq~ (k \pi)^2 \cdot t^{\alp}.
\end{equation}

For each smooth function $\varphi$ with support in $]1+t^{\alpha}, \infty[$, define 
\[  L_t(\varphi)~ =~  -t^2 \cdot  \varphi''~
   +~  \left( (k\pi)^2~ -~ \frac{E_t}{y^2} \right) \cdot \varphi.
\]
Extend $L_t$ to a self-adjoint operator  on
$L^2( [1+t^{\alpha}, \infty[, dy)$. 
It follows from (\ref{eq:LowerBound}) 
that the spectrum of $L_t$ lies in $[(k \pi)^2 \cdot t^{\alpha}, \infty[$.
Hence $L_t$ is invertible and the operator norm of $\| L_t^{-1}\|$ is
bounded above by $\frac{t^{-\alp}}{k^2\pi^2} .$  Therefore 
\begin{equation}   \label{est:particular}
\int_{1+t^\alp}^{\infty} \left| L_t^{-1}(r)\right |^2\,\frac{dy}{y^2}~
\leq ~\int_{1+t^\alp}^{\infty} \left| L_t^{-1}(r)\right |^2\,dy~ 
   \leq~  \frac{t^{-2\alp}}{(k \pi)^4} \int_{1+t^\alp}^{\infty} \left| r(y)\right |^2\,dy.
\end{equation}

The function $L_t^{-1}(r)$ is a solution to (\ref{eq:inhomoODE}) on 
$[1+t^\alp,\infty)$, and hence $w:=v-L_t^{-1}(r)$ is a solution to the homogeneous equation 
(\ref{eq:homoODE}). It follows from (\ref{eq:LowerBound}) that
\begin{equation} \label{est:w2convex}
\left(w^2\right)''(y)~ 
 \geq~  t^{\alp-2} \cdot w^2(y).
\end{equation}
if $y \in [1+t^{\alpha}, \infty[$.  
In particular, $w^2$ is convex, moreover $w^2$ is non-negative and in 
$L^2([1+t^\alp,\infty),y^{-2}dy)$ since $v\in L^2([1+t^\alp,\infty),y^{-2}dy)$ and 
$L_t^{-1}r \in L^2([1+t^\alp,\infty),dy)\subset L^2([1+t^\alp,\infty),y^{-2}dy).$
This implies that $\lim_{y \rightarrow \infty} w^2(y)=0$. Indeed,
since $w^2$ is convex, $(w^2)'$ has a limit $m$ in $\R\cup
\{+\infty\}.$ If this limit is positive then it implies that
$w^2(y)\geq \frac{m}{2}y$ for large $y$ and this contradicts the fact
that $w^2(y)y^{-2}$ is integrable. In particular, $(w^2)'$ is bounded
so that by integrating (\ref{est:w2convex}) we find that $w^2\in
L^1([1+t^\alp,\infty),dy).$ The argument also shows that $(w^2)'$ is non
positive for large $y$ so that $w^2$ has a limit when $y\rightarrow
\infty.$ Since $w^2$ is integrable, this limit is $0.$      

For each $y\in [1+t^\alp,\infty)$ the function $e_y,$ that is defined by 
\[  e_y(z)~ =~ w^2(y) \cdot \exp\left(- t^{\frac{\alp-2}{2}} \cdot \left(z-y \right) \right) \]
satisfies $e_y''(z) = t^{\alp-2} \cdot e_y(z)$ with $e_y(y)=w^2(y)$ and 
$\lim_{z \rightarrow \infty} e_y(z)=0$.  Therefore, by comparison with
(\ref{est:w2convex}), and using the maximum principle,  
we find that if $z\geq y$, then $w^2(z)~ \leq~  e_z(y)$. Applying this
to $z=y+t^\alp,$ we find that for each $ y\geq 1+t^\alp$
\[
 w^2(y+t^\alp)~ \leq~  \exp\left(-t^{\frac{3\alp-2}{2}}\right) w^2(y).
\]

By integration we obtain 
\[
\int_{1+2t^\alp}^\infty w^2(y) \, dy \leq 
\exp\left(-t^{\frac{3\alp-2}{2}}\right)\cdot \int_{1+t^\alp}^\infty
w^2(y) \, dy
\]

This implies 
\[
\int_{1+2t^\alp}^\infty w^2(y) \, dy \leq 
\frac{\exp\left(-t^{\frac{3\alp-2}{2}}\right)}{1-\exp\left(-t^{\frac{3\alp-2}{2}}\right)}
\cdot \int_{1+t^\alp}^{1+2t^\alp}
w^2(y) \, dy
\]
It follows that for $t$ small enough we have 
\begin{equation}
\begin{split}
\int_{1+2t^\alp}^\infty w^2(y) \, dy & \leq\,
\exp\left(-t^{\frac{3\alp-2}{2}}\right) \int_{1+t^\alp}^{1+2t^\alp}
w^2(y) \, dy \\
& \leq \und \exp\left(-t^{\frac{3\alp-2}{2}}\right) \int_{1+t^\alp}^{1+2t^\alp}
w^2(y) \, y^{-2}dy \\
& \leq \und \exp\left(-t^{\frac{3\alp-2}{2}}\right) \int_{1+t^\alp}^{\infty}
w^2(y) \,y^{-2} dy 
\end{split}
\end{equation}
We now use (\ref{est:particular}) and the triangle
inequality to obtain 
\begin{equation*}
  \begin{split}
\| v\|_{[1+2t^\alp,+\infty)} & \leq \|w\|_{[1+2t^\alp,+\infty)}\,+\,
  \|L_t^{-1}(r)\|_{[1+2t^\alp,+\infty)} \\
& \leq \exp
\left(-t^{\frac{3\alp-2}{2}}/2\right)\|w\|_{[1+t^\alp,+\infty),y^{-2}}%
\,+\, Ct^{-\alp}\|r\|_{[1;+\infty)} \\
& \leq \exp
\left(-t^{\frac{3\alp-2}{2}}/2\right)\left(
  \|v\|_{[1+t^\alp,+\infty),y^{-2}}\,+\,\|L_t^{-1}(r)\|_{[1+t^\alp,+\infty),y^{-2}}\right)\\
&~~ \,+\, Ct^{-\alp}\|r\|_{[1,+\infty)} \\ 
  \end{split}
\end{equation*}
The claim follows.
\end{proof}

\begin{coro}\label{coro:normlocal}
For each $\alp\in]0,\dt[$, there exists $C$ and $t_0$ such that, for
each $t<t_0$,
\[
\int_{1+2t^\alp}^\infty \left | v_t^k(y) \right |^2 \frac{dy}{y^2}%
\leq\,  C\cdot t^{2-2\alp}\|u_t\|^2.
\]
\end{coro}
\begin{proof}
Using orthogonality, we have that 
$\|v_t^k\|_{L^2 \left(\frac{dy}{y^2} \right)}\,\leq \|w_t\| \sim \|u_t\|.$ 
Since $y^{-2} \leq 1$ on the interval $[1+t^\alpha,\infty)$, the integral with $v_t^k$ on the
right-hand side is bounded by $C\|u_t\|^2.$ The integral with $r$ is
controlled via Lemma
\ref{lem:restimate}.
\end{proof}

\begin{coro} \label{coro:localization}
There exist $C$ and $t_0>0$ so that if $t<t_0$
\begin{equation} \label{est:exp_range} 
\left | \int_{1+2t^{\alpha}}^{\ua} g(y) \cdot v_t^k(y) \cdot (t\psi)'(y)\, dy\right |~ 
  \leq~ C \cdot t^{1-\alp}  \cdot \|u_t\|\cdot \|\psi\|.
\end{equation}
\end{coro}

\begin{proof}
Use the boundedness of $g$, the Cauchy-Schwarz inequality,  
the preceding corollary, Lemma \ref{lem:restimate},
and Lemma \ref{lem_psi_0_bound}. 
\end{proof}

This corollary holds for each $\alp \in\, ]0, \frac{2}{3}[.$ 
However we will want this contribution to be $o(t^{\dt})$ so that we will 
need to take $\alp \in\, ]0,\unt[.$


\subsubsection{The Airy approximation}

For small $t$, the function $f^k_t$ has a simple zero near $y=1$.
Thus, to approximate solutions of 
$t^2 \cdot v'' + f^k_{E_t} \cdot v= r$ near $y=1$ we will
use solutions to Airy's differential equation  $w''(x) - x \cdot w=0$ 
where $x=y-1$.

We first describe the link to Airy's equation. 
If we define $W(x):= v^k_t(x+1)$, then we have 
\begin{equation}  \label{eq:ODE2} 
 -t^2 \cdot W''(x)~ 
+~ \left( (k\pi)^2 - \frac{E}{(x+1)^{2}}
  \right) \cdot W(x)~ 
  =~ \tilde{r}(x).
\end{equation}
where $\tilde{r}(x)= r(x+1)$.
Let $\rho$ be the smooth function that satisfies 
\[
\frac{1}{(x+1)^2}~  =~  1~  -~ 2x~ +~  x^2 \cdot \rho(x).
\]
By substituting the latter expression into (\ref{eq:ODE2}) 
and by dividing by $2 E_t$, we find that
\[
-\frac{t^2}{2E_t} \cdot W''~ +~ x \cdot W~
 +~ \und \cdot \left(\frac{(k \pi)^2}{E_t}-1\right) \cdot W~
=~ \frac{\widetilde{r}}{2 E_t}~  -~ \frac{x^2}{2} \cdot \rho \cdot W. 
\]
Setting 
\begin{eqnarray}  \label{eq:s_z_s}
  s~ & =&  \frac{t}{\sqrt{2E_t}}  \\
  z_s~ &= &  \und\cdot \left(1~ -~ \frac{(k \pi)^2}{E_t} \right), 
\end{eqnarray}
we have  
\begin{equation} \label{eq:newAiry}
-s^2 \cdot W''(x)~ +~ (x-z_s) \cdot W(x)~ =~ R_t(x)
\end{equation}
where
\[ R_t(x)~ =~  (2E_t)^{-1} \cdot \widetilde{r}(x)~ 
   -~ 2^{-1} \cdot x^2 \cdot \rho(x) \cdot W(x).
\]

In the next few subsections, we will analyse the solutions to (\ref{eq:newAiry}). 
But first, we provide an estimate of the $L^2([0, 3t^{\alpha}])$ norm of $R_t$. 
\begin{lem}\label{lem:Airyapprox}
For each $\alp<\unt$, there exists $C>0$ and $t_0>0$ such that for each $t<t_0$  
\begin{equation}\label{eq:boundAW}
\int_0^{3t^{\alp}} \left|R_t(x)\right|^2~ dx~ \leq~ C\cdot t^{2(2\alp+\unt)}
\cdot \|u_t\|^2. 
\end{equation}
\end{lem}

\begin{proof}
We have $E \geq (k \pi)^2 \geq 1$, and hence by Lemma \ref{lem:restimate}, 
\[
  \left\| \frac{r}{E} \right\|^2~ \leq~ \|r_t\|^2 ~ \leq~ C\cdot t^2 \cdot \|u_t\|^2.  
\]
Let $\psi^*=\psi^*_t$ be the tracking eigenvalue branch associated to $E=E_t$.
The eigenvalue $\lambda_t^*$ corresponding to $\psi^*$ is in a $O(t)$ neighbourhood 
of $E_t.$ Moreover, when $t$ tends to $0$, 
Proposition (\ref{lem:lambda_asymp}) implies that $\lambda_t$ is at a
distance of order $t^\dt$ of the rest of the spectrum. Using
(\ref{eq:lambda_asymp}), (\ref{eq:rdefn}), Lemma \ref{lem:restimate}
and a resolvent estimate, we have $\|v_t^k- \psi^*_t\|=O(
t^{\frac{1}{3}})\|u_t\|$ and hence
\begin{equation}
    \int_{0}^{3t^\alpha}
    \left| x^2 \cdot \rho(x) \cdot (v_t-\psi^*_t)(x+1)\right|^2~ dx~
    \leq~  C \cdot t^{4 \alpha} \cdot t^{\frac{2}{3}}\|u_t\|^2.
\end{equation}
The tracking eigenfunction $\psi^*$ satisfies (\ref{eq:inhomoODE}) with $r=0$.
Hence, by Proposition \ref{prop:expdecay}, if $\alpha \leq
\widetilde{\alpha}<2/3$, we obtain  
\[ \int_{2t^{\widetilde{\alpha}}}^{3t^{\alpha}} 
\left|x^2 \cdot \rho(x) \cdot \psi^*(x+1)\right|^2~ dx~
\leq~ C \cdot  t^{4\alpha} \cdot 
\exp\left(- t^{\frac{3\widetilde{\alp}-2}{2}}\right) \|\psi^*\|^2.
\]
Observe that, by orthogonality, $\|\psi^*\| \leq \|u_t\|$ ; it follows 
\[ \int_{0}^{3t^{\alpha}} 
\left|x^2 \cdot \rho(x) \cdot \psi^*(x+1)\right|^2~ dx~
\leq~ C' \cdot \left(t^{4\widetilde{\alpha}}~ +~ 
t^{4\alp}\exp\left(- t^{\frac{3\widetilde{\alp}-2}{2}}\right) \right)
\cdot \| u_t\|^2.
\]
Since $2(2\alp+\unt)<2$ (because $\alp<\unt$), we may thus take $\tilde{\alp}=\und$ and 
the biggest term is then of order $t^{2(2\alp+\unt)}$. The claim follows. 
\end{proof}

\subsubsection{The inhomogeneous, semi-classical  Airy equation} 

According to (\ref{eq:newAiry}), an estimate of $v$ will result from
estimating the solutions to 
\begin{equation} \label{eq:inhomoAiry}
-s^2 \cdot W''~ +~ (x-z_s) \cdot W~ =~ R
\end{equation}
for $s^{-\dt} \cdot z_s$ in a fixed compact set. 

We first construct solutions to the associated homogeneous equation 
\begin{equation} \label{eq:homoAiry}
-s^2 \cdot W''~ +~ (x-z_s) \cdot W~ =~ 0
\end{equation}
using Airy functions. In particular, it is well-known that there exists
a basis $\{A_+, A_-\}$ of solutions to $-A''(x) + x \cdot A(x)=0$ 
such that
\begin{equation} \label{eq:Airy_asymp}
\lim_{x \rightarrow \infty}~ \frac{A_{\pm}(x)}{ x^{-\frac{1}{4}} \cdot 
  \exp \left( \pm \dt \cdot x^{\frac{3}{2}}   \right)}~ =~ 1,
\end{equation}
and 
\begin{equation} \label{eq:Airy_deriv_asymp}
\lim_{x \rightarrow \infty}~ \frac{A_{\pm}'(x)}{ x^{\frac{1}{4}} \cdot 
  \exp \left( \pm \dt \cdot x^{\frac{3}{2}}   \right)}~ =~ \pm 1.
\end{equation}
One checks that
\begin{equation} \label{eq:defn_W}
W_{\pm}(x)~ =~ A_{\pm} \left(s^{-\dt}(x-z_s) \right)
\end{equation}
defines a basis of solutions to (\ref{eq:homoAiry}).

It follows from well-known identities that the Wronskian---$A_+'A_--A_+A_-'$---of 
$\{A_+,A_-\}$ is $2$. Hence the Wronskian of $\{W_+, W_-\}$
is $2s^{-\dt}$. Therefore, by the method of variation of constants, 
for each $\overline{x}>0$, the function
\begin{equation}   \label{eq:Wparticular}
  W_{\overline{x}}(x)~ =~
 \und \cdot s^{-\qt}  \left( W_+(x) \int_x^{\overline{x}}  R \cdot W_-~ 
    +~ W_-(x) \int_0^x  R \cdot W_+~ \right)
\end{equation}
is a solution to (\ref{eq:inhomoAiry}).

\begin{lem}\label{lem:estWp}
For each compact set $K \subset \Rbb$, there 
exists $C$ such that, for each $\overline{x}>0$,  if $s^{-\dt} \cdot z_s\in K$ and $x\in [0,\overline{x}]$, then 
\begin{equation} \label{est:Wparticular}
 \left| W_{\overline{x}}(x) \right|~ \leq~ C\cdot s^{-\qt} \cdot s^{\unt} \cdot \|R\|
\end{equation}
and 
\begin{equation}  \label{est:Wparticularprime}
 \left| W_{\overline{x}}'(x) \right|~ \leq~ 
  C\cdot s^{-\qt} \cdot s^{-\unt} \cdot \|R\|
\end{equation} 
where $\|R\|$ denotes the $L^2$ norm of $R$ over $[0, \overline{x}]$. 
Moreover, there exists $M$ so that if $x \in \left[ M \cdot s^{\dt}, \overline{x} \right]$,then 
\begin{equation}
 \  \left | W_{\overline{x}}(x)\right |~ 
\leq~ 2 \cdot s^{-\qt} \cdot  s^{\unt} \cdot s^{\und}\cdot \|R\| \cdot x^{-\tq}
\end{equation}
and
\begin{equation} \label{est:W_deriv_decay}
 \  \left | W_{\overline{x}}'(x)\right |~
\leq~ 2 \cdot s^{-\qt} \cdot 
 s^{-\unt} \cdot s^{\uns} \cdot \|R\|\cdot x^{-\unq}.
\end{equation}
\end{lem} 

\begin{proof}
Using the Cauchy-Schwarz inequality, we have for $x \in [0, \overline{x}]$
\[ \left| \int_x^{\overline{x}}  R(z) \cdot W_-(z)~ dz\right|~
  \leq~  \left( \int_0^{\overline{x}}  R(z)^2~ dz\right)^{\und} \cdot
  \left( \int_x^{\infty}  W_-(z)^2~ dz\right)^{\und}
\]
and 
\[ \left| \int_0^{x}  R(z) \cdot W_+(z)~ dz\right|~
  \leq~  \left( \int_0^{\overline{x}}  R(z)^2~ dz\right)^{\und} \cdot
  \left( \int_0^{x}  W_+(z)^2~ dz\right)^{\und}
\]
Thus, from (\ref{eq:Wparticular}) and the triangle inequality, we have
\begin{equation}
\begin{split}
   \left| W_{\overline{x}} (x)\right|~ 
  & \leq~ \und s^{-\qt} \cdot  \left(
\left|W_+(x)\right| \cdot \|R\| \cdot \left( \int_{x}^{\infty} W_-^2 \right)^{\und}~
\right .\\
&\left . ~  +~ \left|W_-(x)\right| \cdot \|R\|  
\cdot \left( \int_{0}^{x} W_+^2 \right)^{\und} \right) 
\end{split}
\end{equation}
Estimate (\ref{est:Wparticular}) then follows from Lemma \ref{lem:Wdt} below.

To prove (\ref{est:Wparticularprime}) we apply a similar argument to
\[ W_{\overline{x}}'(x)~ =~  2 \cdot s^{-\qt} \cdot  \left( W_+'(x) \int_x^{\overline{x}}  R \cdot W_-~  
    +~ W_-'(x) \int_0^x  R \cdot W_+~ 
    \right).
\]
\end{proof}

Define $I_{+}(x) = [0,x]$ and $I_-(x)=[x, \infty]$.

\begin{lem} \label{lem:Wdt}
There exists $C$ so that if $x \geq 0$ and $s^{-\dt} \cdot z_s \in K$, then 
\begin{equation}\label{eq:WintW}
  W_{\pm}(x)^2 \int_{I_{\mp}(x)} W_{\mp}(y)^2~ dy~ \leq~
   C \cdot s^{\dt}, 
\end{equation}
and
\begin{equation}  W_{\pm}'(x)^2 \int_{I_{\mp}(x)} W_{\mp}(y)^2~ dy~ \leq~
   C \cdot s^{-\dt}. 
\end{equation}
Moreover, there exists a constant $M$ so that if $x> M\cdot  s^{\dt}$,
then 
\begin{equation}
  W_{\pm}(x)^2 \int_{I_{\mp}(x)} W_{\mp}(y)^2~ dy~ \leq~
   4 \cdot \sqrt{2} \cdot s^{\dt} \cdot s \cdot x^{-\frac{3}{2}}. 
\end{equation}
and
\begin{equation}  W_{\pm}'(x)^2 \int_{I_{\mp}(x)} W_{\mp}(y)^2~ dy~ \leq~
   2 \sqrt{2} \cdot s^{-\dt} \cdot s^{\unt} \cdot x^{-\und}. 
\end{equation}

\end{lem}

\begin{proof}
The proof is a straightforward consequence of the continuity and known asymptotics 
of $A_{\pm}$ and $A_{\pm}'$. From (\ref{eq:Airy_asymp}) and integration by parts we 
find that
\begin{equation}\label{eq:normApm}
\int_{I_{\mp}(u)} \left | A_{\pm}(r)\right |^2 \, dr~ 
 \sim~ \und \cdot u^{-1} \cdot \exp\left( \pm \qt\cdot u^\td\right),
\end{equation}
as $u$ tends to $\infty$.

Thus there exists $u^*$ so that if $u \geq u^*$,
then
\[  
\int_u^{\infty} A_-(r)^2~ dr \,\leq\,u^{-1}  \exp\left(- \qt \cdot u^{\frac{3}{2}} \right).
\]
Therefore, for $u \geq u^*$,
\begin{equation} \label{est:AplusAminus}
  A_+(u)^2 \int_u^{\infty} A_-(r)^2~ dr~ \leq~ 2 \cdot u^{-\frac{3}{2}},
\end{equation}
and, using (\ref{eq:Airy_deriv_asymp}), 
\begin{equation} \label{est:AplusAminusDeriv}
 A_+'(u)^2 \int_u^{\infty} A_-(r)^2~ dr~ \leq~ 2 \cdot u^{-\und}.
\end{equation}
The expressions on the left hand sides of (\ref{est:AplusAminus})
and (\ref{est:AplusAminusDeriv})
are continuous in $u$, and hence are bounded by a constant $C$
for $u \in \check{K} \cup [0, \infty[$ where $u\in \check{K} \Leftrightarrow -u\in K$.

By (\ref{eq:defn_W}) and the change of variable $r= s^{-\dt}\cdot (y-z_s)$,
we have 
\begin{equation}
 \int_{x}^{\infty} W_-(y)^2~ dy~ =~ s^{\dt} \cdot \int_{u_s(x)}^{\infty} A_-(r)^2~ dr.
\end{equation}
where $u_s(x)= s^{-\dt}\cdot (x-z_s)$.  Since for each $x>0$ and $u_s(x)\geq -\sup K$ estimate 
(\ref{eq:WintW}) follows. 
  
Moreover, from (\ref{est:AplusAminus})
and (\ref{est:AplusAminusDeriv}) we have
\begin{equation} \label{est:WplusWminus}
  W_+(x)^2 \int_{u_s(x)}^{\infty} W_-(y)^2~ dy~ \leq~ 2 \cdot s^{\dt} \cdot 
   u_s(x)^{-\frac{3}{2}},
 \end{equation}
and 
\begin{equation} \label{est:WplusWminusDeriv}
 W_+'(x)^2 \int_{u_s(x)}^{\infty} W_-(y)^2~ dy~ \leq~ 2 \cdot s^{-\dt} 
  \cdot u_s(x)^{-\und}
\end{equation}
provided $u_s \geq u^*$. Let $M = u^* + 2 \sup K$. If 
$x > M \cdot s^{\dt}$, then $x-z_s> x/2$ and $u_s(x)> u_*$. 
The desired estimates in the $+$/$-$ case  follow. The estimates
in the $-$/$+$ case are proved in a similar fashion.
\end{proof}


\subsubsection{The end of the proof of Lemma  \ref{lem:bwkpsi}}

By (\ref{est:exp_range}) it suffices to estimate
\begin{equation} \label{eq:small_x_integral}
 \int_{1}^{1+3t^{\alpha}} g(y)  \cdot v_t^k(y) \cdot \left(t\psi \right)'(y)\, dy~ =~
\int_{0}^{3t^{\alpha}} \widetilde{g}(x) \cdot W_t(x) \cdot \left(t \widetilde{\psi} \right)'(x)\, dx.
\end{equation}
where $W_t=v_t^k(x+1)$, $\widetilde{g}(x)=g(x+1)$, and 
$\widetilde{\psi}(x)= \psi(x+1)$. By assumption, the $C^1$ norm 
of $g$ and hence $\widetilde{g}$ is uniformly bounded.

The function $W_t$ satisfies the inhomogeneous equation (\ref{eq:inhomoAiry})
with $s=t/\sqrt{E_t}$, and the inhomogeneity $R_t$ satisfies (\ref{eq:boundAW}).  
In order to estimate $W_t$ and hence (\ref{eq:small_x_integral}), we write 
\[  W_t~ =~ W_{p,t}~ +~ W_{h,t} \]
where  $W_{p,t}$ is taken to be the particular solution $W_{\overline{x}}$ to (\ref{eq:inhomoAiry}) 
defined by (\ref{eq:Wparticular}) where we set $\overline{x}=3t^\alp$. The function $W_{h,t}$
is then a solution to the associated homogeneous equation. 

\begin{lem} \label{lem:particular}
For each $\alpha \in~ ]\frac{13}{42}, \unt[$ there exists $\delta>0$ such that
\[ 
\left | \int_{0}^{2t^{\alpha}} \widetilde{g} \cdot W_{p,t} 
 \cdot \left(t\widetilde{\psi}\right)'~ dx~\right |
\leq~ 
C \cdot t^{\dt+\delta} \cdot
\left\|\psi\right\| \cdot  \left\|u\right\|.
\]
for all $t$ sufficiently small.
\end{lem}

\begin{proof}
Integration by parts gives
\begin{equation}
\int_{0}^{2t^{\alpha}} \widetilde{g} \cdot W_{p,t} \cdot \left(t\widetilde{\psi}\right)'
= \left.\widetilde{g} \cdot W_{p,t} \cdot t\widetilde{\psi}~ \right|_{0}^{2t^{\alpha}} 
-  \label{eq:Wpt_integral}
 \int_{0}^{2t^{\alpha}} \partial_x \left(\widetilde{g} \cdot W_{p,t}\right) \cdot 
t\widetilde{\psi}.
\end{equation}
Using Lemmas \ref{lem:Airyapprox} and \ref{lem:estWp}, one finds that
there exists $C$ such that, for $x \in [0,3t^{\alpha}]$ 
\[  \left| W_{p,t}(x)  \right|~ 
\leq~ C\cdot t^{-1} \cdot t^{2\alp+\unt}
\cdot \|u_t\|  \]
Thus, using (\ref{lem_psi_0_bound}), we conclude that the first term on the 
right side of (\ref{eq:Wpt_integral}) is 
$O(t^{2 \alpha +\unt}) \cdot \|u_t\| \cdot \|\psi\|$. 

To bound the second term on the right side of (\ref{eq:Wpt_integral}), 
we separately consider the integral of $(\partial_x \tg) \cdot W_{p,t} \cdot t \tpsi$
and the integral of $\tg \cdot (\partial_x W_{p,t}) \cdot t \tpsi$.
For the first integral we can use the same uniform bound 
on $W_{p,t}$ as above to obtain a contribution that is
$t^\alpha \cdot O(t^{2\alpha+\unt})\cdot \|u_t\|\cdot \|\psi\|.$ 

To estimate the second integral, we choose $\talpha$ so that $\dt> \widetilde{\alpha} > \alpha$,
and separately estimate the integral over $[0, t^{\talpha}]$ and the integral over 
$[t^{\talpha}, 3 t^{\alpha}]$. Observe that since $\widetilde{\alp}<\dt$ then 
$t^{\widetilde{\alp}}>Ms^\dt.$ Using Lemmas \ref{lem:Airyapprox} and \ref{lem:estWp}, we find $C$ 
so that for all sufficiently small $t$
\[
\int_{t^{\talpha}}^{2t^{\alpha}} \left| \widetilde{g} \cdot W_{p,t}'(x) \right|~ dx~
\leq~  C \cdot t^{-\td} \cdot t^{-\unq \cdot \talpha} \cdot t^{2\alp+\unt} \cdot t^{\alpha}
\cdot \|u_t\|,
\]
and
\[  \int_{0}^{t^{\talpha}}\left| \widetilde{g} \cdot W_{p,t}'(x) \right|~ dx~
\leq~ C \cdot t^{-\frac{5}{3}}  \cdot t^{2\alp+\unt} \cdot t^{\talpha} \cdot \|u_t\|.
\]

By combining these estimates and using  (\ref{eq:psi_zero_deriv}), we find that 
\[ 
\begin{split}
\int_{0}^{2t^{\alpha}} \widetilde{g} \cdot W_{p,t} \cdot 
\left(t \cdot \widetilde{\psi}\right)'~ & \leq~ C\cdot t^{2\alpha+\unt}
\left \|u_t\right \|\,+\, C\cdot t \cdot \int_0^{2t^\alpha} \left| \widetilde{g} \cdot W_{p,t}'(x) \right|~ dx~\\
& \leq ~ C \cdot \left( t^{2\alpha+\unt}~ +~ t^{-\frac{1}{6}+3\alpha- \frac{1}{4}\talpha}~
   +~  t^{-\unt+2 \alpha+\talpha}      \right)
 \left\|u_t\right\|.
\end{split}
\]
The claim will follow provided we can choose $(\alpha, \talpha)$ 
so that  $\alpha<\frac{1}{3}$, $\alpha < \talpha$,
and each power of $t$ appearing on the righthand side 
is greater than $2/3$. The solution set  to this problem is the 
open triangle in $\Rbb^2$ bounded by the lines $\alpha<1/3$, $2 \alpha+\talpha=1$, 
and $3 \alpha -\talpha/4=5/6$.
The two latter lines intersect for $\alp=\frac{13}{42}.$ 
The claim follows.
\end{proof}

The same kind of argument allows us to estimate the norm of $W_{p,t}$

\begin{lem}\label{lem:normWpt}
For all $\alp \in\, ]\frac{7}{33}, \frac{1}{3}[,$ there exists $\delta >0$ and $C>0$ such that 
\begin{equation}\label{eq:normWp}
\| W_{p,t}\|_{[0,3t^\alp]} \,\leq \, C\cdot t^\delta\cdot \|u_t\|.
\end{equation}
\end{lem}

\begin{proof}
As above we consider $\alp<\unt$ and take some $\talpha>\alp.$ Using Lemmas \ref{lem:Airyapprox} and \ref{lem:estWp}, 
one finds that
\begin{equation*}
\begin{split}
\| W_{p,t}\|_{[t^{\talpha},3t^\alp]} & \leq C \cdot t^{\frac{5\alp}{2}-\frac{3\talpha}{4}-\frac{1}{6}} \|u_t\| \\
\| W_{p,t}\|_{[0,t^{\talpha}]} & \leq C \cdot t^{2\alp+\frac{\talpha}{2}-\frac{2}{3}} \|u_t\|.
\end{split}
\end{equation*}

The claim will follow provided we can find $\alp<\talpha$ and $\alp<\frac{1}{3}$ such that 
$\frac{5\alp}{2}-\frac{3\talpha}{4}-\frac{1}{6}>0$ and $2\alp+\frac{\talpha}{2}-\frac{2}{3}>0.$
Here the solution set is a quadrilateral whose projection onto the $\alp$-axis is the interval 
$]\frac{7}{33},\unt[.$
\end{proof}
Finally, we consider the integral corresponding to the homogeneous
part $W_{h,t}$ of $W_t$:
\begin{equation}   \label{eq:Wht_integral}
\int_{0}^{2t^{\alpha}} \widetilde{g}(x) \cdot W_{h,t}(x) \cdot 
  \left( t \cdot \widetilde{\psi} \right)'(x)~ dx.
\end{equation} 
There exist constants $a_+$, $a_-$, depending on $t$, such that 
\[ W_{h,t}~ =~ a_+ \cdot W_{+}~ +~ a_- \cdot W_{-}. 
\] 
where $W_+$ and $W_-$ are as defined in (\ref{eq:defn_W}) 
with the parameter $s$ and $z_s$ defined in (\ref{eq:s_z_s}). 

We first prove a lemma that roughly says that in the decomposition 
$W = W_{p,t}\,+\,a_+W_+\,+\,a_-W_-$ 
the $L^2$ norm  is mainly supported by $a_-W_-.$

\begin{lem}\label{lem:norms}
For all $\alp \in~ ]\frac{7}{33}, \frac{1}{3}[,$ there exists $\delta >0$ such that 
\begin{equation} \label{est:a+W+}
\| a_+W_+\|_{[0,2t^\alp]}~ =~ O(t^\infty) \cdot \| u_t\| 
\end{equation}
where $O(t^\infty)$ is a function that is of order $t^n$ for each $n$,  
and 
\begin{equation}  \label{est:a-W-}
\|a_- W_-\|_{[0,2t^\alp]}~ \geq~ \und \cdot \left( \|w_t^k\| -C\cdot t^\delta \|u_t\|\right).
\end{equation}
\end{lem}

\begin{proof}
Using the behavior of the norm of $A_\pm$ we find that 
\[
\| a_+W_+\|_{[0,2t^\alp]}~ =~ O(t^\infty) \cdot \|a_+W_+\|_{[2t^\alp,3t^\alp]}
\]
and 
\[
\|a_- W_-\|_{[2t^\alp,3t^\alp]}~ \leq~ C \cdot \|a_- W_-\|_{[0,2t^\alp]}.
\]
We thus have 
\begin{eqnarray*}
\| a_+W_+\|_{[0,2t^\alp]} &= & O(t^\infty) \cdot \|a_+W_+\|_{[2t^\alp,3t^\alp]}\\
& \leq & O(t^\infty) \cdot \left ( \|W\|_{[2t^\alp,3t^\alp]}\,+\, \|a_-
  W_-\|_{[2t^\alp,3t^\alp]}\,+\,\|W_{p,t}\|_{[2t^\alp,3t^\alp]}\right)
    \\
  & \leq &O(t^\infty) \cdot \left( \|u_t\| \,+\,\|W_-\|_{[0,2t^\alp]} \,+\,
  t^\delta \|u_t\|\right) \\
& \leq &O(t^\infty) \cdot \left( \|u_t\| \,+\, \|W\|_{[0,2t^\alp]}\,+\,
  \|a_+W_+\|_{[0,2t^\alp]}\,+\,\|W_{p,t}\|_{[0,2t^\alp]}\right) \\
  & \leq &O(t^\infty)\left( \|a_+W_+\|_{[0,2t^\alp]}\,+\,\|u_t\|\right).  
\end{eqnarray*}
Estimate (\ref{est:a+W+}) then follows by absorbing the norm of $a_+W_+$ into
the left hand side. 

To prove estimate (\ref{est:a-W-}), we first observe that 
by using the triangle inequality we find that
\[
\|W\|_{[0,2t^\alp]} \leq \|W_{p,t}\|_{[0,2t^\alp]} \,+\,\|a_- W_-\|_{[0,2t^\alp]}\,+\,\|a_+W_+\|_{[0,2t^\alp]}. 
\]
The first term on the righta hand side is $O(t^\delta) \|u_t\|$ and the last one
is $O(t^\infty)\|u_t\|$ so that we obtain 
\[
\| a_-W_-\|_{[0,2t^\alp]} \geq \left ( \|W\|_{[0,2t^\alp]} -
    O(t^\delta) \|u_t\|\right ).
\] 
The claim then follows by observing that Corollary
\ref{coro:normlocal} implies that 
\[
\| W\|_{[0,2t^\alp]} \geq \und \left( \|w_t^k\| - O(t^{1-\alp})
  \|u_t\|\right). 
\]
\end{proof}
\begin{lem}  \label{W-}
We have  
\[ \int_{0}^{2t^{\alpha}} \widetilde{g} \cdot 
  W_{-} \cdot \left(t \cdot \widetilde{\psi} \right)'~ dx~
  =~
 (\pi k) \cdot t \cdot g(1)\cdot A_-\left(-s^{-\dt} z_s \right)~ 
  +~ O\left(t^{\qt}\right) 
\]
for $t$ small.
\end{lem}

\begin{proof}
From (\ref{eq:psi_prime}), we have
\[ \widetilde{\psi}'(x)~ =~ 
 -(x+1)^{-\und} \cdot \left(r+ \frac{1}{2r}\right) 
     \cdot \sin(r \cdot \ln(x+1)). 
\]
Thus, the integral we want to estimate can be written as 
\[
t\left( r +\frac{1}{2r}\right) \int_0^{2t^\alp} a_0(x)
A_-(s^{-\frac{2}{3}}(x-z_s))\sin(r\cdot \ln(x+1))\, dx,
\]
where we have set $a_0(x) := -(x+1)^{-\und}\widetilde{g}.$
Denote by $I(t)$ the integral 
\[
I(t) = r\int_0^{2t^\alp} a_0(x)
A_-(s^{-\frac{2}{3}}(x-z_s))\exp(ir\cdot \ln(x+1) \,dx.
\]
Integration by parts shows that 
\begin{equation*}
\begin{split}
I(t) & = \left. -i a_1(x) A_-(s^{-\frac{2}{3}}(x-z_s))\exp(ir\cdot
  \ln(x+1)\right |_0^{2t^\alp} \\
&
- \frac{1}{ir}\int_0^{2t^\alp} \partial_x\left( a_1(x)
  A_-(s^{-\frac{2}{3}}(x-z_s) \right) \left( r\exp(ir\cdot
  \ln(x+1))\right) \, dx 
\end{split}
\end{equation*}
where we have set $a_1(x) = a_0(x)(x+1).$

Since $\alp<\unt <\dt$  and $s$ is of order $t$, and since $A_-$ is
rapidly decreasing, the boundary term at $2t^\alp$ is $O(t^\infty).$
Observe that we have a global $\frac{1}{r}$ prefactor in front of the
integral term. Thus,
when the $\partial_x$ is applied to $a_1$, we gain $1/r$, that is, something
of order $t.$ When $\partial_x$ hits the Airy function, we
lose a $s^{-\dt}$ so that the global prefactor is of order
$\frac{s^{-\dt}}{r}$ which is $O(t^\unt)$. Summarizing, integrating by
parts gains at least a prefactor $t^\unt.$ 

By repeated integration by
parts we thus observe that we can write, for each $N$  
\begin{equation*}
I(t)~ =~ \sum_{n=0}^{N-1}\sum_{k+\ell = n} r^{-k}\left(
  \frac{s^{-\dt}}{r}\right)^\ell a_{k,\ell}A^{(\ell)}(-s^{-\dt}z_s)~ +~  R_N~  +~ O(t^\infty),
\end{equation*}
where the $a_{k,\ell}$ are some constants and the remainder term $R_N$ 
can be written 
\[
R_N(t) := \sum_{k+\ell = N}  r^{-k}\left(
  \frac{s^{-\dt}}{r}\right)^\ell \int_0^{2t^\alp} a_{k,\ell}(x)
A_-^{\ell}(s^{-\dt}(x-z_s))\left(  r\exp(ir\cdot
  \ln(x+1))\right) \, dx 
\]
for some smooth functions $a_{k,\ell}.$ 
If we fix some order $t^M$ then, using that $A_-$ and all its
derivatives are rapidly decreasing, we can find $N$ such that the 
remainder $R_N$ is $O(t^N).$ This tells us that $I(t)$ admits a
complete asymptotic expansion of the form 
\[
I(t) \sim a_{00}r^{-1} A(-s^{-\dt}z_s)\,+\,\sum_{k,l\geq 1} a_{k,l} r^{-k}\left(
  \frac{s^{-\dt}}{r}\right)^\ell a_{k\ell}A^{(\ell)}(-s^{-\dt}z_s).
\]
From the first integration by parts we see that 
\[
a_{00}= i \widetilde{g}(0)
\] 
and the second term is then of order $t^\unt.$  
The claim follows by taking the imaginary part.
\end{proof}

We will use the following to verify that the leading order term does not vanish.

\begin{lem} \label{lem:Nonzero}
We have 
\[ \lim_{s \rightarrow 0}~  s^{-\dt} \cdot z_s~
 =~ -\zeta \]
where $-\zeta$ is a zero of the derivative of $A_-$.
\end{lem}

\begin{proof}
From (\ref{eq:s_z_s}) we have
\[  s^{-\dt} \cdot z_s~ =~ 
    \frac{E_t~ -~ (\pi k)^2}{ 2^{\dt} \cdot t^{\dt} \cdot E_t^{\dt}}.
\]
By combining Lemma \ref{Tracking} and Lemma \ref{lem:lambda_asymp} we have
\[  E_t~ -~ (\pi k)^2~ =~  2^{\dt} \cdot (\pi k)^{\qt} \cdot 
   (-\zeta) \cdot t^{\dt}~ +~ O\left(t\right). \]
where $\zeta$ is a zero of $A_-'$. Since $\lim_{t\rightarrow 0} E_t=
(\pi k)^2$, the claim follows.
\end{proof}

\begin{coro}
There exists $\kappa'>0$ and $t_0>0$ so that if $t<t_0$, then 
\[ \left|  \int_{0}^{3t^{\alpha}} \widetilde{g}(x) \cdot 
  W_{-}(x) \cdot \left(t \cdot \widetilde{\psi} \right)'(x)~ dx \right|~
\geq~ \kappa'  \cdot t^{\dt}  \cdot \|W_-\| 
\]   
for $t$ sufficiently small.
\end{coro}

\begin{proof}
Let $\zeta$ be the zero of $A_-$ that comes from Lemma \ref{lem:Nonzero}.
Since $A_-$ is a nontrivial solution to a second order differential equation, 
$A_-$ can not vanish at a zero of the derivative $A_-'$. Hence, 
for sufficiently for small $t$, we have 
$|A_-(-s^{\dt} \cdot z_s)| > \und |A_-(\zeta)| >0$. 

By arguing as in the proof of Lemmas \ref{lem:Nonzero} and 
\ref{lem:Wdt} and using $s \sim t$, we find $c_1>0$ so that
\[  \int_0^{\beta-1} W_-(x)^2~ dx~ \geq~ \frac{1}{4} \cdot c_1 \cdot t^{\dt} 
\] 
where $k_1= \int_{-\sup(K)}^{\infty} |A_-(u)|^2$ and $t$ is sufficiently small.
In particular,
\begin{equation} \label{est:Wminusnorm}
   1~  \leq~ \frac{2}{\sqrt{c_1}} \cdot  t^{-\unt} \cdot \|W_-\|.
\end{equation}
Hence the claim follows from Lemma \ref{W-}.
\end{proof}

The estimate in the latter corollary is homogeneous so that we can
multiply $W_-$ by $a_-$. 

Using Lemma \ref{lem:norms} we then have 
\[
\left|  \int_{0}^{2t^{\alpha}} \widetilde{g}(x) \cdot 
 a_- W_{-}(x) \cdot \left(t \cdot \widetilde{\psi} \right)'(x)~ dx \right|~
\geq~ \kappa'  \cdot t^{\dt}  \cdot \|a_-W_-\| \geq \frac{\kappa'}{2}\cdot t^{\dt}\left (
  \|w_t^k\|-t^\delta \|u_t\|\right)  
\]
and 
\[
\left|  \int_{0}^{2t^{\alpha}} \widetilde{g}(x) \cdot 
 a_+ W_{+}(x) \cdot \left(t \cdot \widetilde{\psi} \right)'(x)~ dx \right|~
\leq~ O(t^\infty)  \cdot \|u_t\| 
\]

Putting all the different pieces together yields the estimate.


\end{document}